\documentclass[12pt,letterpaper,titlepage]{amsart}
\usepackage{amsmath, amssymb, amsthm, amsfonts,amscd,amsaddr}
\usepackage{fullpage}

\newenvironment{Cases}{\hspace{-8pt}\begin{cases}}{\end{cases}\hspace{-20pt}}

\newcounter{zeile}
\newcommand{\zz}{\stepcounter{zeile}(\arabic{zeile})}

\theoremstyle{plain}
\newtheorem{theorem}{Theorem}[section]
\newtheorem{proposition}[theorem]{Proposition}
\newtheorem{cor}[theorem]{Corollary}

\newtheorem{prop}[theorem]{Proposition}
\newtheorem{lemma}[theorem]{Lemma}

\theoremstyle{definition}

\newtheorem{rmk}[theorem]{Remark}
\numberwithin{equation}{section}
\newtheorem*{theoremA*}{Theorem A}
\newtheorem*{theoremB*}{Theorem B}
\newtheorem*{theorem1*}{Theorem A'}
\newtheorem*{theoremC*}{Theorem C}
\newtheorem*{theoremD*}{Theorem D}
\newtheorem*{theoremE*}{Theorem E}
\newtheorem*{theoremF*}{Theorem F}
\newtheorem*{theoremE2*}{Theorem E2}
\newtheorem*{theoremE3*}{Theorem E3}
\newcommand{\bs}{\backslash}

\newcommand{\C}{\mathbb{C}}

\newcommand{\K}{\mathbb{K}}

\newcommand{\Hb}{\mathbb{H}}

\newcommand{\Z}{\mathbb{Z}}

\newcommand{\R}{\mathbb{R}}

\newcommand{\Aut}{\operatorname{Aut}}

\newcommand{\Sl}{\operatorname{SL}}

\newcommand{\SO}{\operatorname{SO}}

\newcommand{\Sp}{\operatorname{Sp}}

\newcommand{\Ad}{\operatorname{Ad}}

\newcommand{\ad}{\operatorname{ad}}
\newcommand{\diag}{\operatorname{diag}}

\newcommand{\Spin}{\operatorname{Spin}}

\newcounter{pairnumber}

\newcommand{\GPV}[5]{
	\parbox{110pt}{\tiny \pnlabel{#4}\strut\hfill {\tiny$#5$}
		 		 \vspace{0,5cm}
		 \begin{center}
			$#1 \hfill #2$ \\
			{\begin{picture}(43,23)
				\put(0,20){\circle*{4}}      \put(40,20){\circle*{4}}     
				\put(0,20){\line(1,-1){20}} \put(40,20){\line(-1,-1){20}} 
				\put(20,0){\circle*{4}}                                  
								\end{picture}
			} \\
			$#2$\\  
			{\begin{picture}(5,12)\put(0,10){\line(0,-1){10}}
			\end{picture}}
			
			 			 {$#3$}  \hspace{2pt}
			 
		\end{center}\vspace{0,5cm}
	}
}

\newcommand{\GPVIntro}[3]{
	\parbox{110pt}{\tiny \strut\hfill {\tiny$#3$}

		 \begin{center}
		
		{
			\begin{picture}(90,50)
				\put(0,37){$#1$} \put(65,37){$#2$}
				\put(23,30){\circle*{4}}      \put(77,30){\circle*{4}}    
				 \put(50,10){\line(4,3){27}} \put(50,10){\line(-4,3){27}}
				  \put(50,10){\circle*{4}} 
				\put(32,0){$#2$}                                 
				\end{picture}}

		\end{center}\vspace{0,5cm}
	}
}

\newcommand{\GPNIntro}[4]{%
	\parbox{110pt}{\tiny \strut\hfill{$#4$}

		\begin{center}
			{
			\begin{picture}(90,50)
				\put(0,37){$#1$} \put(65,37){$#2$}
				\put(20,30){\circle*{4}}      \put(80,30){\circle*{4}}    
				\put(20,30){\line(0,-1){20}} \put(80,30){\line(0,-1){20}} \put(80,10){\line(-3,1){60}}
				\put(20,10){\circle*{4}}  \put(80,10){\circle*{4}} 
				\put(5,0){$#3$} \put(65,0){$#2$}                                 
				\end{picture}}
			
		\end{center}\vspace{0,5cm}

	}
}

\newcommand{\GPNi}[6]{%
	\parbox{95pt}{\tiny \pnlabel{#5}\strut\hfill{$#6$}
	\vspace{0,5cm}
		\begin{center}
			{
				\begin{picture}(83,90)
				\put(0,75){$#1$} \put(60,75){$#2$}
				\put(20,70){\circle*{4}}      \put(80,70){\circle*{4}}    
				\put(20,70){\line(0,-1){20}} \put(80,70){\line(0,-1){20}} \put(80,50){\line(-3,1){60}}
				\put(20,50){\circle*{4}}  \put(80,50){\circle*{4}} 
				\put(0,40){$#3$} \put(60,40){$#2$}  
				
				\put(20,37){\line(0,-1){15}} 
				
				\put(0,13){$#4$}                                
				\end{picture}
			}
		\end{center}\vspace{0,5cm}
	}
}

\newcommand{\triple}[4]{%
	\parbox{90pt}{\tiny \pnlabel{#4}\strut\hfill{$#3$}
	\vspace{0,5cm}
		\begin{center}
			{
				\begin{picture}(125,80)
				\put(0,75){$#1$} \put(55,75){$#1$} \put(105,75){$#1$}
				\put(20,70){\circle*{4}}      \put(70,70){\circle*{4}} \put(120,70){\circle*{4}} 
				\put(20,70){\line(2,-1){47}} \put(70,70){\line(0,-1){20}} \put(120,70){\line(-2,-1){46}}
				 \put(70,46){\circle*{4}} 
				\put(55,35){$#1$}  
				
				\put(70,33){\line(0,-1){15}} 
								\put(55,10){$#2$}                                
				\end{picture}
			}
		\end{center}\vspace{0,5cm}
	}
}

\newcommand{\GPNisp}[7]{%
	\parbox{150pt}{\tiny \pnlabel{#6}\strut\hfill{$#7$}
	\vspace{0,5cm}
		\begin{center}
			{
				\begin{picture}(83,90)
				\put(-10,75){$#1$} \put(60,75){$#2$}
				\put(20,70){\circle*{4}}      \put(80,70){\circle*{4}}    
				\put(20,70){\line(0,-1){20}} \put(80,70){\line(0,-1){20}} \put(80,50){\line(-3,1){60}}
				\put(20,50){\circle*{4}}  \put(80,50){\circle*{4}} 
				\put(10,40){$#3$} \put(60,40){$#2$}  
				
				\put(20,37){\line(0,-1){15}}
				\put(25,22){\line(4,1){50}}
				\put(80,37){\line(0,-1){15}}
				
				\put(10,13){$#4$}   \put(50,10){$#5$}                             
				\end{picture}
			}
		\end{center}\vspace{0,5cm}
	}
}

\newcommand{\GPNiv}[7]{%
	\parbox{110pt}{\tiny \pnlabel{#6}\strut \hfill{$#7$}
	\vspace{0,5cm}
		\begin{center}
			{
				\begin{picture}(83,90)
				\put(0,75){$#1$} \put(65,75){$#2$}
				\put(20,70){\circle*{4}}      \put(80,70){\circle*{4}}    
				\put(20,70){\line(0,-1){20}} \put(80,70){\line(0,-1){20}} \put(80,50){\line(-3,1){60}}
				\put(20,50){\circle*{4}}  \put(80,50){\circle*{4}} 
				\put(0,40){$#3$} \put(65,40){$#2$}  
				
				\put(20,37){\line(0,-1){15}} \put(26,37){\line(3,-1){45}}  \put(80,37){\line(0,-1){15}}
				
				\put(0,13){$#4$} \put(65,13){$#5$}                               
				\end{picture}
			}
		\end{center}\vspace{0,5cm}
	}
}

\newcommand{\GPNv}[7]{%
	\parbox{110pt}{\tiny \pnlabel{#6}\strut \hfill{$#7$}
	\vspace{0,5cm}
		\begin{center}
			{
				\begin{picture}(83,90)
				\put(0,75){$#1$} \put(65,75){$#2$}
				\put(20,70){\circle*{4}}      \put(80,70){\circle*{4}}    
				\put(20,70){\line(0,-1){20}} \put(80,70){\line(0,-1){20}} \put(80,50){\line(-3,1){60}}
				\put(20,50){\circle*{4}}  \put(80,50){\circle*{4}} 
				\put(0,40){$#3$} \put(65,40){$#2$}  
				
				\put(20,37){\line(0,-1){15}} \put(28,37){\line(1,-1){16}}  
				\put(76,37){\line(-1,-1){16}} \put(80,37){\line(0,-1){15}}
				
				\put(0,13){$#4$} \put(65,13){$#5$}                               
				\end{picture}
			}
		\end{center}\vspace{0,5cm}
	}
}

\newcommand{\GPNii}[7]{%
	\parbox{105pt}{\tiny \pnlabel{#6}\strut \hfill{$#7$}
		\vspace{0,5cm}
		\begin{center}
			{
				\begin{picture}(83,90)
				\put(0,75){$#1$} \put(65,75){$#2$}
				\put(20,70){\circle*{4}}      \put(80,70){\circle*{4}}    
				\put(20,70){\line(0,-1){20}} \put(80,70){\line(0,-1){20}} \put(80,50){\line(-3,1){60}}
				\put(20,50){\circle*{4}}  \put(80,50){\circle*{4}} 
				\put(0,40){$#3$} \put(65,40){$#2$}  
				
				\put(20,37){\line(0,-1){15}} 
				 \put(80,37){\line(0,-1){15}}
				
				\put(15,13){$#4$} \put(65,13){$#5$}                               
				\end{picture}
			}
		\end{center}\vspace{0,5cm}
	}
}

\newcommand{\GPANl}[8]{%
	\parbox{100pt}{\tiny \pnlabel{#7}\strut\hfill{$#8$}
		\vspace{0,5cm}
		\begin{center}
			{%
				\begin{picture}(83,90)
				\put(0,75){$#1$} \put(60,75){$#2$}
				\put(30,70){\circle*{4}}      \put(80,70){\circle*{4}}    
				\put(30,70){\line(-1,-1){20}} \put(30,70){\line(5,-4){25}} \put(30,70){\line(5,-2){50}} \put(80,70){\line(0,-1){20}} 
				\put(10,50){\circle*{4}} \put(55,50){\circle*{4}} \put(80,50){\circle*{4}}   
				
				\put(0,40){$#3$} \put(45,40){$#4$} \put(70,40){$#2$}    
				
				\put(10,37){\line(0,-1){15}} \put(55,37){\line(0,-1){15}}
				
				\put(0,13){$#5$} \put(50,22){\line(-5,2){36}}\put(45,13){$#6$}
				                           
				\end{picture}
			}
		\end{center}\vspace{0,5cm}
	}
}

\newcommand{\GPNn}[6]{%
	\parbox{95pt}{\tiny \pnlabel{#6}\strut
		\vspace{0,5cm}
		\begin{center}
			{
				\begin{picture}(83,90)
				\put(15,77){$#1$} \put(65,77){$#2$}
				\put(20,70){\circle*{4}}      \put(80,70){\circle*{4}}    
				\put(20,70){\line(0,-1){20}} \put(80,70){\line(0,-1){20}} \put(80,50){\line(-3,1){60}}
				\put(20,50){\circle*{4}}  \put(80,50){\circle*{4}} 
				\put(5,40){$#3$} \put(65,40){$#2$}  
				
				\put(20,37){\line(0,-1){15}} \put(26,37){\line(3,-1){45}} 
				
				\put(10,13){$#4$} \put(65,13){$#5$}                               
				\end{picture}
			}
		\end{center}\vspace{0,5cm}
	}
}

\newcommand{\GPNl}[6]{%
	\parbox{110pt}{\tiny \pnlabel{#5}\strut\hfill{$#6$}
		\vspace{0,1cm}
		
		\begin{center}
			{
			\begin{picture}(100,60)
				\put(0,37){$#1$} \put(65,37){$#2$}
				\put(20,30){\circle*{4}}      \put(80,30){\circle*{4}}    
				\put(20,30){\line(0,-1){20}} \put(80,30){\line(0,-1){20}} \put(80,10){\line(-3,1){60}}
				\put(20,10){\circle*{4}}  \put(80,10){\circle*{4}} 
				\put(5,0){$#3$} \put(65,0){$#2$}                                 
				\end{picture}
			} \\	
			$ $ \\ 
					
			{\rm[}$#4${\rm]} 
		\end{center}\vspace{1,1cm}
	}
}

\newcommand{\GPNll}[6]{%
	\parbox{110pt}{\tiny \pnlabel{#5}\strut\hfill{$#6$}
		\vspace{0,5cm}
		 
		 \begin{center}
			{
				\begin{picture}(83,90)
				\put(0,75){$#1$} \put(65,75){$#2$}
				\put(20,70){\circle*{4}}      \put(80,70){\circle*{4}}    
				\put(20,70){\line(0,-1){20}} \put(80,70){\line(0,-1){20}} \put(80,50){\line(-3,1){60}}
				\put(20,50){\circle*{4}}  \put(80,50){\circle*{4}} 
				\put(5,40){$#3$} \put(65,40){$#2$}     
				\put(80,36){\line(0,-1){15}}    
				\put(65,10){$#4$}                        
				\end{picture}
			} \\

		\end{center}\vspace{0,5cm}
	}
}

\newcommand{\repMi}[8]{%
	\parbox{150pt}{\tiny \pnlabel{#7}\strut\hfill{$#8$}
		\vspace{0,5cm}
		\begin{center}
			{%
				\begin{picture}(113,90)
				\put(0,75){$#1$} \put(80,75){$#2$}
				\put(20,70){\circle*{4}}      \put(100,70){\circle*{4}}    
				\put(20,70){\line(0,-1){20}} \put(20,70){\line(2,-1){40}}  \put(100,70){\line(-2,-1){40}} \put(100,70){\line(0,-1){20}}
				\put(20,50){\circle*{4}}  \put(60,50){\circle*{4}} \put(100,50){\circle*{4}}  
				\put(5,40){$#3$} \put(45,40){$#4$} \put(85,40){$#5$} 
				\put(100,37){\line(0,-1){15}}	
				\put(80,13){$#6$} 
				\end{picture}
			} 
		\end{center}\vspace{0,5cm}
	}
}

\newcommand{\repMii}[9]{%
	\parbox{150pt}{\tiny \pnlabel{#8}\strut\hfill{$#9$}
		\vspace{0,5cm}
		\begin{center}
			{%
				\begin{picture}(113,90)
				\put(0,75){$#1$} \put(80,75){$#2$}
				\put(20,70){\circle*{4}}      \put(100,70){\circle*{4}}    
				\put(20,70){\line(0,-1){20}} \put(20,70){\line(2,-1){40}}  \put(100,70){\line(-2,-1){40}} \put(100,70){\line(0,-1){20}}
				\put(20,50){\circle*{4}}  \put(60,50){\circle*{4}} \put(100,50){\circle*{4}}  
				
				\put(5,40){$#3$} \put(45,40){$#4$} \put(85,40){$#5$} 
				
				\put(20,37){\line(0,-1){15}}   \put(100,37){\line(0,-1){15}}	
				
				\put(0,13){$#6$}   \put(80,13){$#7$} 
				\end{picture}
			} 
		\end{center}\vspace{0,5cm}
	}
}

\newcommand{\repMivi}[9]{
    \def\tempa{#1}%
    \def\tempb{#2}%
    \def\tempc{#3}%
    \def\tempd{#4}%
    \def\tempe{#5}%
    \def\tempf{#6}%
    \def\tempg{#7}%
    \def\temph{#8}%
    \def\tempi{#9}%
    \repMivic
}

\newcommand{\repMivic}[1]{%
	\parbox{140pt}{\tiny \pnlabel{#1}\strut\hfill{$\tempi$}
		\vspace{0,5cm}
		\begin{center}
			{%
				\begin{picture}(113,90)
				\put(0,75){$\tempa$} \put(80,75){$\tempb$}
				\put(20,70){\circle*{4}}      \put(100,70){\circle*{4}}    
				\put(20,70){\line(0,-1){20}} \put(20,70){\line(2,-1){40}}  \put(100,70){\line(-2,-1){40}} \put(100,70){\line(0,-1){20}}
				\put(20,50){\circle*{4}}  \put(60,50){\circle*{4}} \put(100,50){\circle*{4}}  
				
				\put(0,40){$\tempc$} \put(45,40){$\tempd$} \put(80,40){$\tempe$} 

				\put(20,37){\line(0,-1){15}} \put(25,37){\line(2,-1){30}} \put(60,37){\line(0,-1){15}} \put(95,37){\line(-2,-1){30}} \put(100,37){\line(0,-1){15}}	
				
				\put(0,13){$\tempf$} \put(45,13){$\tempg$} \put(80,13){$\temph$} 
				\end{picture}
			} 
		\end{center}\vspace{0,5cm}
	}
}

\newcommand{\repMiii}[9]{
    \def\tempa{#1}%
    \def\tempb{#2}%
    \def\tempc{#3}%
    \def\tempd{#4}%
    \def\tempe{#5}%
    \def\tempf{#6}%
    \def\tempg{#7}%
    \def\temph{#8}%
    \def\tempi{#9}%
    \repMiiic
}

\newcommand{\repMiiic}[1]{%
	\parbox{140pt}{\tiny \pnlabel{#1}\strut\hfill{$\tempi$}
		\vspace{0,5cm}
		\begin{center}
			{%
				\begin{picture}(113,90)
				\put(0,75){$\tempa$} \put(80,75){$\tempb$}
				\put(20,70){\circle*{4}}      \put(100,70){\circle*{4}}    
				\put(20,70){\line(0,-1){20}} \put(20,70){\line(2,-1){40}}  \put(100,70){\line(-2,-1){40}} \put(100,70){\line(0,-1){20}}
				\put(20,50){\circle*{4}}  \put(60,50){\circle*{4}} \put(100,50){\circle*{4}}  
				
				\put(0,40){$\tempc$} \put(45,40){$\tempd$} \put(80,40){$\tempe$} 

				\put(20,37){\line(0,-1){15}} \put(60,37){\line(0,-1){15}}  \put(100,37){\line(0,-1){15}}	
				
				\put(0,13){$\tempf$} \put(45,13){$\tempg$} \put(80,13){$\temph$} 
				\end{picture}
			} 
		\end{center}\vspace{0,5cm}
	}
}

\newcommand{\repMiiiLD}[9]{
    \def\tempa{#1}%
    \def\tempb{#2}%
    \def\tempc{#3}%
    \def\tempd{#4}%
    \def\tempe{#5}%
    \def\tempf{#6}%
    \def\tempg{#7}%
    \def\temph{#8}%
    \def\tempi{#9}%
    \repMiiiLDc
}

\newcommand{\repMiiiLDc}[1]{%
	\parbox{140pt}{\tiny \pnlabel{#1}\strut\hfill{$\tempi$}
		\vspace{0,5cm}
		\begin{center}
			{%
				\begin{picture}(113,90)
				\put(0,75){$\tempa$} \put(80,75){$\tempb$}
				\put(20,70){\circle*{4}}      \put(100,70){\circle*{4}}    
				\put(20,70){\line(0,-1){20}} \put(20,70){\line(2,-1){40}}  \put(100,70){\line(-2,-1){40}} \put(100,70){\line(0,-1){20}}
				\put(20,50){\circle*{4}}  \put(60,50){\circle*{4}} \put(100,50){\circle*{4}}  
				
				\put(0,40){$\tempc$} \put(45,40){$\tempd$} \put(80,40){$\tempe$} 

				\put(20,37){\line(0,-1){15}}\put(24,22){\line(2,1){30}}  \put(60,37){\line(0,-1){15}}  \put(100,37){\line(0,-1){15}}	
				
				\put(0,13){$\tempf$} \put(45,13){$\tempg$} \put(80,13){$\temph$} 
				\end{picture}
			} 
		\end{center}\vspace{0,5cm}
	}
}

\newcommand{\repMiiiRD}[9]{
    \def\tempa{#1}%
    \def\tempb{#2}%
    \def\tempc{#3}%
    \def\tempd{#4}%
    \def\tempe{#5}%
    \def\tempf{#6}%
    \def\tempg{#7}%
    \def\temph{#8}%
    \def\tempi{#9}%
    \repMiiiRDc
}

\newcommand{\repMiiiRDc}[1]{%
	\parbox{140pt}{\tiny \pnlabel{#1}\strut\hfill{$\tempi$}
		\vspace{0,5cm}
		\begin{center}
			{%
				\begin{picture}(113,90)
				\put(0,75){$\tempa$} \put(80,75){$\tempb$}
				\put(20,70){\circle*{4}}      \put(100,70){\circle*{4}}    
				\put(20,70){\line(0,-1){20}} \put(20,70){\line(2,-1){40}}  \put(100,70){\line(-2,-1){40}} \put(100,70){\line(0,-1){20}}
				\put(20,50){\circle*{4}}  \put(60,50){\circle*{4}} \put(100,50){\circle*{4}}  
				
				\put(0,40){$\tempc$} \put(45,40){$\tempd$} \put(80,40){$\tempe$} 

				\put(20,37){\line(0,-1){15}} \put(60,37){\line(0,-1){15}} \put(65,37){\line(2,-1){30}}  \put(100,37){\line(0,-1){15}}	
				
				\put(0,13){$\tempf$} \put(45,13){$\tempg$} \put(80,13){$\temph$} 
				\end{picture}
			} 
		\end{center}\vspace{0,5cm}
	}
}

\newcommand{\repMiviw}[8]{%
	\parbox{140pt}{\tiny 
		\begin{center}
			{%
				\begin{picture}(113,90)
				\put(0,75){$#1$} \put(80,75){$#2$}
				\put(20,70){\circle*{4}}      \put(100,70){\circle*{4}}    
				\put(20,70){\line(0,-1){20}} \put(20,70){\line(2,-1){40}}  \put(100,70){\line(-2,-1){40}} \put(100,70){\line(0,-1){20}}
				\put(20,50){\circle*{4}}  \put(60,50){\circle*{4}} \put(100,50){\circle*{4}}  
				
				\put(0,40){$#3$} \put(45,40){$#4$} \put(80,40){$#5$} 

				\put(20,37){\line(0,-1){15}} \put(25,37){\line(2,-1){30}} \put(60,37){\line(0,-1){15}} \put(95,37){\line(-2,-1){30}} \put(100,37){\line(0,-1){15}}	
				
				\put(0,13){$#6$} \put(45,13){$#7$} \put(80,13){$#8$} 
				\end{picture}
			} 
		\end{center}\vspace{0,5cm}
	}
}

\newcommand{\repMAiii}[9]{
    \def\tempa{#1}%
    \def\tempb{#2}%
    \def\tempc{#3}%
    \def\tempd{#4}%
    \def\tempe{#5}%
    \def\tempf{#6}%
    \def\tempg{#7}%
    \def\temph{#8}%
    \def\tempi{#9}%
    \repMAiiicontinued
}

\newcommand\repMAiiicontinued[2]{%
	\parbox{140pt}{\tiny \pnlabel{#1}\strut \hfill{$#2$}
		\vspace{0,5cm}
		\begin{center}
			{%
				\begin{picture}(123,90)
				\put(-10,75){$\tempa$} \put(70,75){$\tempb$}
				\put(10,70){\circle*{4}}      \put(90,70){\circle*{4}}    
				\put(10,70){\line(0,-1){20}} \put(10,70){\line(2,-1){40}}  \put(90,70){\line(-2,-1){40}} \put(90,70){\line(0,-1){20}} \put(90,70){\line(3,-2){30}}
				\put(10,50){\circle*{4}}  \put(50,50){\circle*{4}} \put(90,50){\circle*{4}} \put(120,50){\circle*{4}}   
				\put(-5,40){$\tempc$} \put(35,40){$\tempd$} \put(75,40){$\tempe$} \put(110,40){$\tempf$}
				\put(10,37){\line(0,-1){15}}  \put(90,37){\line(0,-1){15}} \put(120,37){\line(0,-1){15}}
				\put(-10,13){$\tempg$} \put(60,13){$\temph$} \put(115,23){\line(-3,2){20}}\put(110,13){$\tempi$}          
				\end{picture}
			}
		\end{center}\vspace{0,5cm}
	}
}

\newcommand{\repMin}[9]{
    \def\tempa{#1}%
    \def\tempb{#2}%
    \def\tempc{#3}%
    \def\tempd{#4}%
    \def\tempe{#5}%
    \def\tempf{#6}%
    \def\tempg{#7}%
    \def\temph{#8}%
    \def\tempi{#9}%
    \repMinc
}

\newcommand{\repMinI}[6]{%
	\parbox{130pt}{\tiny  \strut\hfill{$#6$}
		\vspace{0,5cm}
		\begin{center}
			{%
				\begin{picture}(113,90)
				\put(0,77){$#1$} \put(95,77){$#2$}
				\put(20,70){\circle*{4}}      \put(100,70){\circle*{4}}    
				\put(20,70){\line(0,-1){20}} \put(20,70){\line(2,-1){40}}  \put(100,70){\line(-2,-1){40}} \put(100,70){\line(0,-1){20}}
				\put(20,50){\circle*{4}}  \put(60,50){\circle*{4}} \put(100,50){\circle*{4}}  
				
				\put(5,40){$#3$} \put(45,40){$#4$} \put(80,40){$#5$}

				\end{picture}
				
			} 
		\end{center}\vspace{0,5cm}}}

\newcommand{\repMinc}[1]{%
	\parbox{130pt}{\tiny  \pnlabel{#1}\strut\hfill{$\tempi$}
		\vspace{0,5cm}
		\begin{center}
			{%
				\begin{picture}(113,90)
				\put(0,77){$\tempa$} \put(95,77){$\tempb$}
				\put(20,70){\circle*{4}}      \put(100,70){\circle*{4}}    
				\put(20,70){\line(0,-1){20}} \put(20,70){\line(2,-1){40}}  \put(100,70){\line(-2,-1){40}} \put(100,70){\line(0,-1){20}}
				\put(20,50){\circle*{4}}  \put(60,50){\circle*{4}} \put(100,50){\circle*{4}}  
				
				\put(5,40){$\tempc$} \put(45,40){$\tempd$} \put(80,40){$\tempe$} 
				
				\put(20,37){\line(0,-1){15}} \put(95,37){\line(-2,-1){30}} \put(100,37){\line(0,-1){15}}	
				
				\put(0,13){$\tempf$} \put(55,13){$\tempg$} \put(85,13){$\temph$} 
				\end{picture}
				
			} 
		\end{center}\vspace{0,5cm}
	}
}

\newcommand{\repNM}[9]{
    \def\tempa{#1}%
    \def\tempb{#2}%
    \def\tempc{#3}%
    \def\tempd{#4}%
    \def\tempe{#5}%
    \def\tempf{#6}%
    \def\tempg{#7}%
    \def\temph{#8}%
    \def\tempi{#9}%
    \repNMc
}

\newcommand{\repNMc}[1]{%
	\parbox{180pt}{\tiny  \pnlabel{#1}\strut\hfill{$\temph$}
		\vspace{0,5cm}
		\begin{center}
			\hfill $\tempa$ \hfill $\tempb$ \hfill $\tempc$ \hfill \\
			{
				\begin{picture}(123,23)
				\put(20,20){\circle*{4}}      \put(60,20){\circle*{4}}   \put(100,20){\circle*{4}} 
				\put(0,0){\line(1,1){20}} \put(60,20){\line(-1,-1){20}} \put(100,20){\line(-1,-1){20}} \put(40,0){\line(-1,1){20}} \put(80,0){\line(-1,1){20}} \put(120,0){\line(-1,1){20}}
				
				\put(0,0){\circle*{4}} \put(40,0){\circle*{4}}  \put(80,0){\circle*{4}}    \put(120,0){\circle*{4}}                     
				\end{picture}
			} \\			
			$\tempd$ \hfill $\tempe$ \hfill $\tempf$ \hfill $\tempg$ \\
			$ $ \\
			$ \tempi$ 
		\end{center}\vspace{0,5cm}
	}
}

\newcommand{\repMM}[9]{
    \def\tempa{#1}%
    \def\tempb{#2}%
    \def\tempc{#3}%
    \def\tempd{#4}%
    \def\tempe{#5}%
    \def\tempf{#6}%
    \def\tempg{#7}%
    \def\temph{#8}%
    \def\tempi{#9}%
    \repMMc
}

\newcommand{\repMMc}[1]{%
	\parbox{200pt}{\tiny \pnlabel{#1}\strut \hfill{$\temph$}
		\vspace{0,5cm}
		\begin{center}
			$\hfill \tempa\hfill \tempb \hfill \tempc \hfill$ \\
			{
				\begin{picture}(123,23)
				\put(20,20){\circle*{4}}      \put(60,20){\circle*{4}}   \put(100,20){\circle*{4}} 
				\put(20,20){\line(-1,-1){20}} \put(60,20){\line(-1,-1){20}} \put(100,20){\line(-1,-1){20}} \put(20,20){\line(3,-1){60}} \put(80,0){\line(-1,1){20}} \put(120,0){\line(-1,1){20}}
				
				\put(0,0){\circle*{4}} \put(40,0){\circle*{4}}  \put(80,0){\circle*{4}}    \put(120,0){\circle*{4}}                     
				\end{picture}
			} \\			
			$\tempd \hfill \tempe \hfill \tempf \hfill \tempg$ \\
			$ $ \\ 
			$\tempi$
		\end{center}\vspace{0,5cm}
	}
}

\def\af{\mathfrak{a}}
\def\bfrak{\mathfrak{b}}
\def\e{\epsilon}
\def\gf{\mathfrak{g}}

\def\ff{\mathfrak{f}}
\def\cf{\mathfrak{c}}

\def\ef{\mathfrak{e}}
\def\hf{\mathfrak{h}}
\def\kf{\mathfrak{k}}

\def\lf{\mathfrak{l}}
\def\mf{\mathfrak{m}}
\def\nf{\mathfrak{n}}

\def\pf{\mathfrak{p}}
\def\spin{\mathfrak{spin}}
\def\qf{\mathfrak{q}}
\def\rf{\mathfrak{r}}
\def\sf{\mathfrak{s}}
\def\sl{\mathfrak{sl}}
\def\gl{\mathfrak{gl}}

\def\so{\mathfrak{so}}
\def\sp{\mathfrak{sp}}
\def\su{\mathfrak{su}}

\def\uf{\mathfrak{u}}

\def\zf{\mathfrak{z}}

\def\1{{\bf1}}

\def\Oc{\mathcal{O}}

\def\tilde{\widetilde}

\def\str{\mathfrak{s}}
\def\Str{\mathrm{S}}

\newcommand{\HH}{\mathbb{H}}

\newcommand{\sE}{\mathsf{E}} \newcommand{\sF}{\mathsf{F}}
\newcommand{\sG}{\mathsf{G}}

\newcounter{Tabelle}
\newcommand{\Tabelle}[1]{\refstepcounter{Tabelle}Table \arabic{Tabelle}\label{#1}}

\usepackage[usenames]{color}

\title[Classification of real spherical pairs]
{Classification of reductive real spherical pairs\\ II. the semisimple case}

\begin{document}
\date{September 6, 2018}

\begin{abstract} If $\gf$ is a real reductive Lie algebra and $\hf\subset\gf$ is a subalgebra, then the pair
$(\gf,\hf)$ is called real spherical provided that $\gf=\hf+\pf$ for some choice of a minimal parabolic 
subalgebra $\pf\subset \gf$.  This paper concludes the classification of real
spherical pairs $(\gf,\hf)$, where $\hf$ is a reductive real 
algebraic subalgebra. 
More precisely, we classify all such pairs which are strictly
indecomposable, and we discuss (in Section 6) how to construct from these all
real spherical pairs.
A preceding paper treated the case where 
$\gf$ is simple. The present work builds on that case and on the classification by Brion
and Mikityuk for the complex spherical case.
\end{abstract}

\author[Knop]{Friedrich Knop}
\email{friedrich.knop@fau.de}
\address{Department Mathematik, Emmy-Noether-Zentrum\\
FAU Erlangen-N\"urnberg, Cauerstr. 11, 91058 Erlangen, Germany} 

\author[Kr\"otz]{Bernhard Kr\"{o}tz}
\email{bkroetz@gmx.de}
\address{Universit\"at Paderborn, Institut f\"ur Mathematik\\Warburger Stra\ss e 100, 
D-33098 Paderborn, Deutschland}

\author[Pecher]{Tobias Pecher} 
\email{tpecher@math.upb.de}
\address{Universit\"at Paderborn, Institut f\"ur Mathematik\\Warburger Stra\ss e 100, 
D-33098 Paderborn, Deutschland}

\thanks{The second author was supported by ERC Advanced Investigators Grant HARG 268105}
\author[Schlichtkrull]{Henrik Schlichtkrull}
\email{schlicht@math.ku.dk}
\address{University of Copenhagen, Department of Mathematics\\Universitetsparken 5, 
DK-2100 Copenhagen \O, Denmark}

\maketitle

\setcounter{section}{-1}

\section{Introduction}

One of the most remarkable accomplishments of early Lie theory was the discovery by W.~Killing
that the simple complex Lie algebras are classified (up to isomorphism) by the 
irreducible root systems.
The classification was completed and extended to real Lie algebras by E.~Cartan, who determined all
real forms of the algebras from Killing's list and later used that to obtain a classification
(up to local isomorphism) of all Riemannian symmetric spaces. In turn, 
Cartan's list of symmetric spaces was extended to pseudo-Riemannian symmetric spaces
by M.~Berger \cite{Berger}. These classifications have played a profound role in the development of the
theory of semi-simple Lie groups and their symmetric spaces, for example by providing important examples 
for explicit calculations.

More recently the notion of a symmetric space has been generalized. Given a complex
reductive Lie group $G$ and a closed complex subgroup $H$ the homogeneous space $Z:=G/H$
is said to be {\it spherical} if 
there exists a Borel subgroup $B\subset G$ such that the orbit
$Bz\subset Z$ is open for some $z\in Z$.
In particular, this is the case when $G/H$ is symmetric.
The condition of being spherical is local, and it can be stated in terms
of the corresponding Lie algebras
as $\gf=\bfrak+\hf$, a vector space sum, for some Borel subalgebra $\bfrak$
of $\gf$. 
The pair $(\gf,\hf)$ is called {\it spherical} when this condition is fulfilled, 
and it is called {\it reductive} if $\hf$ is
reductive in $\gf$. By extending Cartan's list of symmetric pairs
the spherical reductive pairs of the simple complex Lie algebras were classified  
(up to isomorphism) by M.~Kr\"amer \cite{Kr}. 
Subsequently such a classification, but without the assumption of $\gf$ being 
simple, was obtained by M.~Brion \cite{Brion} and I.V.~Mikityuk \cite{Mik}.

A notion of spherical homogeneous spaces exists also for real Lie groups.
If $G$ is a real reductive Lie group and $H$ a closed subgroup, the space $G/H$
is called real spherical if 
there exists a minimal parabolic subgroup $P\subset G$ such that the orbit
$Pz\subset Z$ is open for some $z\in Z$. Again the notion is local and translates  
into $\gf=\pf+\hf$ for some minimal parabolic 
subalgebra $\pf$, in which case
$(\gf,\hf)$ is called a {\it real spherical pair}. 
It is clear that the pair $(\gf,\hf)$ is real spherical if the 
complexified pair $(\gf_\C,\hf_\C)$
is spherical, since   
(up to conjugation) the Borel subalgebra
$\bfrak$ is contained in the complexification
$\pf_\C$ of $\pf$.
In this case the real pair is said to be {\it absolutely spherical}.
In particular, this is always the case if $G/H$ is a symmetric space. However,
there do exist real spherical pairs which are not absolutely spherical.

In recent years a vivid research activity has taken place for real spherical spaces,
motivated in part by \cite{SV}. See for example \cite{DKS},  \cite{KK}, \cite{KKSS}, 
\cite{KKSS2}, \cite{KKS}, \cite{KKS2},  \cite{KM}, \cite{KoOs},  
\cite{KKOS}, \cite{KS1}, \cite{KS2}, \cite{Mol}. This paper was driven
by the desire to determine the scope of this new area, and by the need to
find relevant examples by which one can support investigations
through explicit computations.

In the  preceding paper \cite{KKPS}, from now on referred to as part I, we 
obtained a classification of the real spherical reductive pairs $(\gf,\hf)$ with
$\gf$ simple, thus providing a real analogue of the list of Kr\"amer.
The aim of this second part is then to obtain a real analogue of the
subsequent classification by Brion and Mikityuk. More precisely, 
all strictly indecomposable (a notion which will be explained below)
real spherical reductive pairs $(\gf,\hf)$ with $\gf$ semi-simple
will be classified up to isomorphism.

To be specific we consider a
real reductive Lie algebra $\gf$, that is, $\gf$ is a real Lie algebra such that 
$\gf=\zf(\gf) \oplus [\gf,\gf]$ is a direct Lie algebra sum with $\zf(\gf)$ the center of $\gf$, and 
with semi-simple derived subalgebra $[\gf,\gf]$.
A subalgebra $\hf\subset\gf$ is called 
\begin{itemize}  
\item {\it reductive (in $\gf$)} if $\ad_\gf|_\hf$ is completely reducible, 
\item {\it compact (in $\gf$)} if $e^{\ad_\gf \hf }\subset \Aut(\gf)$ is compact, 
\item {\it elementary (in $\gf$)}  if $\hf$ is reductive in $\gf$ and $[\hf,\hf]$ is compact, 
\item {\it symmetric (in $\gf$)} if $\hf$ is the fixed point set of an involutive automorphism of $\gf$.
\end{itemize}

{}From now on we let $\hf\subset\gf$ be a reductive subalgebra. It decomposes into
$\hf=\hf_{\rm n} \oplus \hf_{\rm el}$ where $\hf_{\rm el}$ is its largest elementary ideal 
and $\hf_{\rm n}$ is the sum of its non-elementary simple ideals.
In addition we require that $\hf$ is real algebraic, that is $\hf_\C:=\hf\otimes_\R \C$ is an algebraic 
subalgebra of $\gf_\C$. Recall that this is always the case when $\hf$ is semi-simple.

\par 
The classification of real spherical pairs $(\gf,\hf)$ readily reduces to the case where 
$\gf$ is semi-simple with all simple factors being non-compact 
(see Lemma \ref{lemma stand2}). We assume this from now on
and let $\gf=\gf_1\oplus \ldots\oplus \gf_k$ be the decomposition into simple factors. 
For a subalgebra $\hf\subset\gf$ we let $\hf_i$ be the projection of $\hf$ to $\gf_i$. It is easy to 
see that if $(\gf,\hf)$ is real spherical then all $(\gf_i, \hf_i)$ are real spherical 
(see Lemma \ref{lemma stand1}).  In this sense the real spherical pairs
$(\gf,\hf)$ with $\gf$ simple, which were classified in part I,
serve as building blocks for the general classification. 

\par The pair $(\gf,\hf)$ is called {\it decomposable} if  
$\gf=\gf_1\oplus \gf_2$ with $\gf_i\neq \{0\}$ and 
$\hf=(\hf\cap\gf_1 )\oplus( \hf\cap\gf_2)$, and {\it indecomposable} otherwise.
Clearly it suffices to classify indecomposable pairs. 
The step from simple to semi-simple is rather straightforward in the 
classification of symmetric pairs, as the only indecomposable symmetric pairs
with $\gf$ not simple are the so-called group cases
$(\gf,\hf)\simeq (\hf\oplus \hf, \diag\hf)$. For spherical pairs
the situation is far more complicated. This is seen already in the complex
case where Brion and Mikityuk found a multitude of indecomposable complex 
spherical pairs in addition to the group case. 

\par It would then be desirable to give a classification of all indecomposable
reductive real spherical pairs, but in order to obtain an efficient classification 
an additional requirement is necessary.
We call $(\gf,\hf)$ {\it strictly indecomposable} 
provided that $(\gf,\hf_{\rm n})$ is indecomposable,
and we shall only completely classify the strictly indecomposable 
real spherical pairs.
Such a stronger assumption appears also in the work of Brion and Mikityuk,
as explained in \cite[Sect. 5] {Mik} (see also \cite[p.~45-46]{Timashev}). 
In Section \ref{Section not strictly} we describe the reduction
of the indecomposable case to the strictly indecomposable case.

With the requirement of strict indecomposability it is a particular feature of the 
Brion-Mikityuk classification that the number 
of simple factors  of $\gf$ is at most $3$.
Among the most prominent examples of Brion and Mikityuk are the Gross-Prasad spaces:
$$\begin{tabular}{cc} $\GPNIntro{\sl(n+1,\C)}{\sl(n,\C)}{\gl(1,\C)}{n\geq 2}$& \qquad
$\GPVIntro{\so(m+1,\C)}{\so(m,\C)}{m\geq 3}$ \end{tabular} $$
Here one finds $\gf=\gf_1\oplus \gf_2$ in the top row, $\hf$ in the lower row,  
and the various lines 
indicate the factors of $\gf$ into which the  
 given factors $\hf'$ of $\hf$ embed (diagonally if more than one line branches from $\hf'$).
It is noteworthy
 that the second case specialized to $m=3$ yields  the well-known
triple case
$(\sl(2,\C)\oplus\sl(2,\C) \oplus\sl(2,\C) , \diag\sl(2,\C))$.

\par   
We now describe our classification in more detail. 
As in the complex case it turns out that the number of factors
is at most $3$.  The cases with two factors are listed in 
Theorem \ref{thm:semisimplespherical} and the remaining  ones
in Theorem \ref{Thm k=3}. 
In both theorems we use diagrams similar to those above to describe
$\gf$ and $\hf$, and in total there are approximately 50
such diagrams.
Given the classification of Brion-Mikityuk the absolutely spherical
pairs are fairly easy to determine,  and they 
are marked by {\it a.s.}~in our tables.
The majority of the diagrams contain cases which are not absolutely spherical.
We highlight here our two most exotic cases:

$$\begin{tabular}{cc} 
$\repMinI{\sp(n+1,\K)}{\sE_6^3}{\sp(n,\K)}{\sl(2,\R)}{\ \ \su(5,1)}{n\geq 1\atop\K=\R,\C}$&\qquad 
$\repMinI{\sp(n+1,\K)}{\sE_7^3}{\sp(n,\K)}{\sl(2,\R)}{\ \su(10,2)}{n\geq 1\atop \K=\R,\C}$
\end{tabular}\, .$$

A real spherical pair $(\gf,\hf)$ such that $(\gf\oplus \hf, \diag\hf)$ is real spherical, 
as for example the second Gross-Prasad pair above, is called {\it strongly spherical};  see
Theorem \ref{thm:semisimplespherical} 1b) which lists those with both $\gf$ and $\hf$ 
simple.   Strongly spherical pairs deserve special attention in view of their relevance for branching laws from 
$\gf$ to $\hf$  (see \cite{KM}, \cite{Mol}). Prior to this work strongly spherical pairs $(\gf,\hf)$ 
where $\hf\subset\gf$ is symmetric were classified in \cite{KM}.  The cases found 
in \cite{KM} can easily be extracted from our tables 
together with all additional 
cases for non-symmetric $\hf$, see
Table \ref{strsph} at the end of the paper.

\par  Let us now explain our approach.  Given a minimal parabolic subalgebra 
$\pf$ such that $\gf=\hf+\pf$ we recall from \cite{KKS}
the infinitesimal version of the local structure theorem:
There exists a unique parabolic subalgebra $\qf\supset \pf$ of $\gf$ 
with Levi decomposition $\qf=\lf \ltimes \uf$  such that
\begin{itemize} 
\item $\qf\cap\hf=\lf\cap \hf$, 
\item $\lf\cap\hf$ contains all non-compact simple ideals of $\lf$.
\end{itemize}
The subalgebra $\str(\gf,\hf):=\lf\cap\hf$ of $\hf$ is reductive and
is an invariant of the real spherical pair $(\gf,\hf)$. It is called 
the {\it structural algebra} (cf.~\cite{KnSt} where $\str(\gf,\hf)$ is called principal subalgebra). 
Suppose now that $\gf=\gf_1\oplus\gf_2$ and that $\hf\subset\gf$ is a 
reductive subalgebra with projection $\hf_i$  to $\gf_i$  for $i=1,2$.
A necessary condition for $(\gf,\hf)$ to be  real spherical 
is that both $(\gf_1,\hf_1)$ and $(\gf_2,\hf_2)$ are real 
spherical. Let us assume this and set $\hf':=\hf_1\oplus\hf_2$.
Then $\hf'$  is a real spherical subalgebra of $\gf$ which  contains $\hf$. Our main tool for the classification 
is the fact (see Lemma \ref{lemma stand1}) that in this case 
$\hf\subset\gf$ is real spherical if and only if 
there exist minimal parabolic subalgebras 
$\pf_{i}\subset\str(\gf_i,\hf_i)$ such that 
$$ \hf' = \hf + \pf_{1}+\pf_{2}. $$ 
In particular it is a necessary condition for $(\gf,\hf)$ to be real spherical that 
$$ \hf' = \hf + \str(\gf_1,\hf_1)+ \str(\gf_2,\hf_2)\,. $$ 
This condition turns out to be rather restrictive in view of the factorization results of 
Onishchik \cite{Oni}, which were
already used in  Part I, and which are recalled in Proposition \ref{Oni2}.
In the first step we therefore determine all structural Lie algebras $\str(\gf,\hf)$ for $\gf$ simple. 
Some part was already done in  Part I and in  Appendix \ref{appendix} we provide a complete list in
Tables \ref{lcaph_class_symm} - \ref{lcaph_KKPS_class}.
It is then a matter of efficient bookkeeping to obtain the classification for $k=2$ factors, which is recorded in 
Theorem \ref{thm:semisimplespherical}. This theorem is divided into three parts:  $\hf$ simple, 
$\hf$ semi-simple but not simple, and $\hf$ reductive but not semi-simple.
In all cases we also determine  the subalgebra $\str(\gf,\hf)$ of $\hf$.  Having obtained 
the classification for $k=2$ together with this information, it is then a rather quick task to 
derive both the classification for $k=3$ and the exclusion 
of $k\geq 4$  (see Theorem \ref{Thm k=3} and Proposition \ref{four exclusion}).

In Appendix \ref{AppB} at the end of the paper we prove a general result on the 
geometric structure of restricted root spaces for symmetric spaces, which generalizes a
theorem of Kostant for Riemannian symmetric spaces, \cite{Kostant} Thm.~2.1.7.
The result is applied in Section \ref{The case of two factors} to a particular symmetric space of $E_7$.

\subsection*{Acknowledgment}
It is a pleasure to thank a meticulous referee for some valuable suggestions which
resulted in an improved exposition. In particular, Prop.~\ref{referees suggestion} was
suggested by him. In addition we thank Jan Frahm for pointing out some inaccuracies
in an earlier version of the paper.

\section{Notation for classical and exceptional Lie groups}

Fix a real reductive Lie algebra $\gf$ and let $G_\C$ be a linear complex algebraic group
with Lie algebra $\gf_\C=\gf\otimes_\R\C$.   If $\gf_\C$ is classical, then 
 we shall denote by $G_\C$ the corresponding 
classical group, i.e. $G_\C= \Sl(n,\C), \SO(n,\C), \Sp(n, \C)$.  
To avoid confusion let us stress that we use the notation $\Sp(n,\R)$, $\Sp(n,\C)$ to indicate that the underlying classical vector space 
is $\R^{2n}$, $\C^{2n}$.  Further $\Sp(n)$ denotes the compact real form of $\Sp(n,\C)$ and likewise the underlying 
vector space for $\Sp(p,q)$ is $\C^{2p+2q}$.   Finally we use 
$\SO_0(p,q)$ for the identity component of ${\mathrm O}(p,q)$, the indefinite 
orthogonal group on $\R^{p+q}$.

 For the exceptional Lie algebras we
use the notation of Berger, \cite[p.~117]{Berger} , and
write $\sE_6^\C, \sE_7^\C$ etc.~for the complex simple 
Lie algebras of type $E_6, E_7$ etc., and $\sE_6, \sE_7$ etc.~for
the corresponding compact real forms. For the non-compact real forms we write
\begin{eqnarray*}\sE_6^1, \sE_6^2, \sE_6^3, \sE_6^4\qquad &\hbox{for}& \qquad \mathrm{ E\,I, E\,II, E\,III, E\,IV}\\
\sE_7^1, \sE_7^2, \sE_7^3\qquad &\hbox{for}& \qquad  \mathrm{ E\,V, E\,VI, E\,VII}\\
\sE_8^1, \sE_8^2\qquad &\hbox{for}& \qquad  \mathrm{ E\,VIII, E\,IX}\\
\sF_4^1, \sF_4^2 \qquad &\hbox{for}& \qquad \mathrm{ F\,I, F\,II}\end{eqnarray*}
and finally  $\sG_2^1$ for  the unique non-compact real form of $\sG_2^\C$. 

\par By slight abuse of notation we also denote the simply connected Lie groups with these Lie algebras
by the same symbols.

\section{{Real spherical pairs}}

\subsection{ Preliminaries}

Let $\gf$ be a real reductive Lie algebra and $\hf\subset\gf$ an algebraic 
subalgebra. Recall from the introduction
that the pair $(\gf,\hf)$ is called {\it real spherical} provided there exists a minimal 
parabolic subalgebra $\pf  \subset \gf$ such that 
$$ \gf = \hf +\pf\, .$$  
For later reference we record the obvious necessary condition
\begin{equation}\label{dim bd}
\dim \gf \leq \dim \hf +\dim \pf.
\end{equation}

A pair $(\gf,\hf)$ of a complex Lie algebra and a complex subalgebra
is called {\it complex spherical}  or just {\it spherical} if it is real spherical
when regarded as a pair of real Lie algebras.
Note that in this case the minimal parabolic subalgebras 
of $\gf$ are precisely the Borel subalgebras.
We recall also that a real reductive pair $(\gf,\hf)$ for which
$(\gf_\C, \hf_\C)$ is spherical is said to be {\it absolutely spherical},
and from  \cite[Lemma 2.1]{KKS2} we record:

\begin{lemma} \label{realform_of_complexspherical}
All absolutely spherical pairs are real spherical.
\end{lemma}

We denote by $G$ the connected Lie subgroup of $G_\C$ with Lie algebra $\gf$, and
for an algebraic subalgebra $\hf\subset\gf$ we denote by $H\subset G$ the corresponding  connected Lie subgroup. 
Let $P\subset G$ be a minimal parabolic subgroup. 
A real variety $Z$ on which $G$ acts is called real spherical 
if there is an open $P$-orbit on $Z$. In particular, this applies to 
the homogeneous space $Z:=G/H$.
 If $L\subset G$ we write $L_H:=L\cap H$ for the intersection with $H$.

\subsection{ Local structure theorem}

Let $Z=G/H$ be a real spherical space and choose a minimal parabolic subgroup 
$P\subset G$ such that $PH$ 
is open in $G$.  
The local structure theorem,
 \cite[Th. 2.3 and Th. 2.8]{KKS}, asserts that there is a  unique parabolic subgroup $Q\supset P$ 
 which admits a Levi decomposition $Q= L \ltimes U$ for which:
\begin{eqnarray}  
\label{LST1} PH&=&QH\, , \\ 
\label{LST2} Q_H&=&L_H\,, \\ 
\label{LST3} L_{\mathrm n}&\subset&  L_H\, ,
\end{eqnarray}
where $L_{\mathrm n}\triangleleft L$ is the connected normal subgroup with Lie algebra 
$\lf_{\mathrm n}$, the sum of all 
non-compact simple ideals of $\lf$.

The parabolic subgroup $Q$ depends on the homogeneous space $Z$, and on
the choice of minimal parabolic subgroup $P$. 
We refer to it as being {\it adapted} to $Z$ and $P$ or, for short,
just adapted to $Z$ when $P$ is clear from the context. 
 We emphasize that it is $Q$, but not the Levi part $L$ satisfying
(\ref{LST1})-(\ref{LST3}), which is uniquely determined
by $Z$ and $P$. However, it follows from (\ref{LST2}) that 
$L_H$ is unique in the same sense as $Q$. By (\ref{LST3}), $L_H$ is a reductive subgroup of $L$ and hence also
of $G$. 

On the Lie algebra level we let 
$\str(\gf, \hf):=\lf_\hf=\lf\cap \hf$,
and we call the subalgebra $\str(\gf, \hf)$ of $\hf$ the
{\it structural algebra} associated with $(\gf,\hf)$ and $P$.  The dependence of $\str(\gf,\hf)\subset \hf$ 
on $P$ is explained in the following remark, see also \cite[Lemma 13.5]{KK}. 
\begin{rmk} Suppose that $P'H$ is open for some minimal 
parabolic $P'\subset G$. Then $P'=P_g:= g^{-1}Pg$ for some $g\in G$ 
such that $PgH$ is open. Since $P_\C H_\C$ is the unique open $P_\C \times H_\C$-double coset 
in $G_\C$, it follows that the number of open $P\times H$-double cosets is finite and hence  represented 
by a finite set $W\subset G$.  Moreover, the local structure 
theorem implies that $W$ can be chosen such that elements $w\in W$ are of the form $w=t_\C h_\C$ with 
$t_\C \in Z(L_\C)$ (the center of $L_\C$) and $h_\C \in H_\C$.  In particular $\Ad(t_\C)\sf =\sf$.  We may assume that 
$g=wh$ for some $w\in W$ and $h\in H$ and then the structural algebra with respect to $P_g$ is 
$\sf_g=\Ad(g)^{-1}\sf$. 
In particular,  we see that all possible structural algebras are $\Ad(H)$-conjugate to some $\sf_w$, $w\in W$. Finally note  that $\sf_w=\Ad(h_\C)^{-1}\sf\subset \hf$ is a "real point" 
of  $\Ad(H_\C)\sf$. 
\end{rmk}

The structural algebra $\str(\gf, \hf)$
is the Lie algebra of the group $L_H$, for which
we use the notation $\Str(G,H)$ in case we need to 
specify the pair $(G,H)$ to which it is associated.

We recall from  Part I, Lemma 2.8 the following result.

\begin{lemma}\label{induced open} \label{min psgp LcapH} 
Let $G/H$ be real spherical and let $Q\supset P$ be adapted to it. Then
 $P_H$ is a minimal parabolic subgroup of  $L_H$.\end{lemma}

 \subsection{Iwasawa decomposition} \label{subsub Lan}
The local structure theorem enables us to  
choose an Iwasawa decomposition for $G$ which 
corresponds well with the structure of $G/H$. 
We denote by $N$ the unipotent radical of the minimal parabolic subgroup $P$,  
and remark that any unipotent subgroup of $P$ is contained 
in $N$. The fact that $P$ is minimal is equivalent to the property that 
$P/N$ is elementary. 

With $Q=LU$ as above we choose a Cartan involution of $L$ and inflate it
to a Cartan involution of $G$. Corresponding to that we can obtain an Iwasawa
decomposition $G=KAN$ of $G$, such that $A\subset L$ and such that  
$N$ is exactly the unipotent radical of $P$. With 
$M=Z_K(A)$, the centralizer of $A$ in $K$, we have the decomposition 
$P=(MA)\ltimes N$ with the elementary Levi part $MA$.
This decomposition and its Lie algebra version $\pf=\mf +\af+\nf$ are usually
referred to as Langlands decompositions. 

\par  Finally we recall that $\gf$ is called {\it split} if $\mf$, as defined above,
is zero, and {\it quasi-split} if $\mf$ is abelian, or equivalently, if
$\pf_\C$ is a Borel subalgebra of $\gf_\C$.

\subsection{Towers of spherical subgroups}

By a tower of subgroups of a group we mean just a nested sequence of 
two  subgroups.
Let 
\begin{equation} \label{tower}
H\subset H'\subset G
\end{equation}
be a tower of reductive subgroups in $G$
for which the upper quotient $Z'=G/H'$ is real spherical. We let
$P\subset G$ be a minimal parabolic subgroup
such that $PH'\subset G$ is open, and let
$Q'$ be the parabolic subgroup adapted to $Z'$ and $P$.
According to Lemma \ref{induced open}, 
$P_{H'}$ is a minimal parabolic subgroup of 
$S':=\Str(G,H')=Q'_{H'}$. 

We recall from Part I the following necessary and sufficient condition
that also the deeper quotient $Z=G/H$ of (\ref{tower})
is real spherical. 

\begin{prop}\label{tower-factor}
$G/H$ is real spherical if and only if $H'/H$ is
real spherical for the action of $S'$, i.e. 
$P_{H'}$ has an open orbit on $H'/H$. 
If this is the case then $H'=S' H$, 
and 
\begin{equation} \label{HH'}  H'/H\simeq S'/S'_H\, .\end{equation}
\end{prop}

\begin{proof} 
See Prop.~2.9 and Cor.~2.10 in Part I.  
\end{proof} 

We now assume that $G/H$ is real spherical, and we
want to relate the structural algebras of the various
quotients in (\ref{tower}) to each other. For that we
assume that $P$ has been chosen such that $PH$ is open in $G$.
This implies the previous assumption that
$PH'$ is open, and hence we can maintain the notation 
related to $Q'$ from above.

\begin{prop}\label{2nd-tower-factor-prop}
Let $Q\supset P$ be 
adapted to $G/H$ and let $S:=\Str(G,H)=Q_H$. Then the following assertions hold: 
\begin{enumerate}
\item \label{aa}$Q\subset Q'$ and $S\subset S'$. 
\item \label{bb}$Q_{S'}=Q_{H'}\supset P_{H'}$, and hence $Q_{H'}$ is a parabolic subgroup of $S'$
\item \label{cc}$P_{H'} S'_{H}$ is open in $S'$.
\item \label{dd}$Q_{H'}$ is adapted to the real spherical space $S'/S'_{H}$  and 
$P_{H'}\subset S'$. 
\item \label{ee}The associated structural algebras satisfy
\begin{equation} \label{L-tower} 
\str\big[\str(\gf,\hf'),\str(\gf,\hf')\cap\hf\big]=\str(\gf,\hf)\,. 
\end{equation}
\end{enumerate}
\end{prop}

\begin{proof} 
\eqref{aa} It is sufficient to establish $Q\subset Q'$ as this implies
$Q_H \subset Q'_{H'}$.  We recall the characterization of 
$Q_\C$ from 
\cite[Lemma 3.7]{KKSS}: $Q_\C=\{ g\in G_\C\mid gP_\C H_\C = P_\C H_\C\}$. Likewise 
$Q_\C'=\{ g\in G_\C \mid gP_\C H_\C' = P_\C H_\C'\}$.  
In particular $Q_\C H_\C=P_\C H_\C$, and by multiplying this with $H'_\C$ 
from the right we obtain $ Q_\C H_\C' = P_\C H_\C'.$
Thus $Q_\C \subset Q_\C'$  by the aforementioned characterization of $Q'$. 
Intersecting the obtained inclusion with $G$ finally yields 
$Q\subset Q'$.

\par  \eqref{bb} The asserted equality  $Q_{H'}=Q_{S'}$ amounts to showing
$Q_{H'}\subset S'$. Since $S'=Q'_{H'}$ 
by definition, this follows from $Q\subset Q'$ which we established in \eqref{aa}. 

\par \eqref{cc} 
The fact that $PH$ is open in $G$ implies 
that $P_{H'}H$ is open in $H'$. It then
follows from the isomorphism in (\ref{HH'})  
that $P_{H'}  S'_{H}$ is open in $S'$. 

\par \eqref{dd}
Let ${\mathcal L}\subset Q_{H'}$ be any choice of Levi subgroup with ${\mathcal L}\supset S$.  Since ${\mathcal L}\subset Q= LU$ we have 
${\mathcal L}_H=Q_H=S$ and ${\mathcal L}_{\rm n} =L_{\rm n}\subset S$. This then yields a Levi decomposition 
$Q_{H'}= {\mathcal L}\ltimes U_{H'}$ for which we need to verify \eqref{LST1} - \eqref{LST3}  for 
$S'/S'_H$. Now $PH=QH$ yields  $P_{H'} H = Q_{H'}H$ which is \eqref{LST1}. 
Further 
$Q_{H'}\cap H = S = Q_H={\mathcal L}\cap H$ which is \eqref{LST2}.  Finally, we have 
${\mathcal L}_H=S \supset {\mathcal L}_{\rm n} = L_{\rm n}$, that is \eqref{LST3}. 

\par \eqref{ee}
This follows immediately from \eqref{dd} since $Q_{H'}\cap H=Q_H$. 
\end{proof} 

\begin{rmk} 
  It is shown above that $Q\subset Q'$ for the parabolic subgroups of $G$ adapted to 
 $Z$ and $Z'$. By a careful analysis of the proof of the local structure theorem
 in \cite{KKS} one can show in addition that Levi decompositions $Q=LU$ and $Q'=L'U'$ can be chosen 
 such that $L\subset L'$ and such that all requirements 
 in the local structure theorem are satisfied. However,  we do not need this
 in the current paper.
 
 Another observation, which  likewise will not be used here, is that 
Proposition \ref{2nd-tower-factor-prop}
 and the additional property just mentioned, are valid without the assumption that $H$ and $H'$ are reductive. 
\end{rmk}

\begin{cor}\label{tower-corollary}
Let $\gf$ be a reductive Lie algebra and $\hf$ a reductive 
subalgebra. Assume $\hf=\hf_1\oplus\hf_2$ is a decomposition in ideals, 
and consider
$\hf\subset \gf\oplus\hf_2$ by the inclusion map
 $x\mapsto (x,x_2)$, where $x_2$ denotes the $\hf_2$-component of $x$.
Suppose that $(\gf,\hf)$ is real spherical, and let 
 $\sf_2$
denote the projection
of $\str(\gf,\hf)$ to $\hf_2$.
Then the following statements are equivalent: 
\begin{enumerate}
\item\label{rs1} $(\gf\oplus\hf_2,\hf)$ is real spherical. 
\item\label{rs2}  $(\hf_2\oplus\sf_2,\sf_2)$ is real spherical.
\end{enumerate} 
Moreover if these conditions hold, then 
\begin{equation} \label{l descent} 
 \str(\gf\oplus\hf_2,\hf)\simeq (\str(\gf,\hf)\cap \hf_1)
\oplus \str(\hf_2\oplus\sf_2,\sf_2)\, .\end{equation}
Here the second summand on the right-hand side of 
{\rm (\ref{l descent})} embeds
into $\hf\subset\gf\oplus\hf_2$ by the restriction 
of the diagonal embedding  $\sf_2\to \str(\gf,\hf)\oplus\sf_2$.
\end{cor}
 
\begin{proof}
 We apply Propositions \ref{tower-factor}- \ref{2nd-tower-factor-prop}
 to the tower $$H\subset H':=H\times H_2\subset G\times H_2$$
 where the first inclusion map is given 
 by $x\mapsto (x,x_2)$ as above. Then 
 $H'/H\simeq H_2$ and $\Str(G\times H_2,H')= S\times H_2$
 where $S=\Str(G,H)$. The action of
 $ S$ on $H'/H$ factors through the projection $H\to H_2$,
 and hence identifies with the action of  its projection
$S_2$ on $H_2$. Thus
 $H'/H$ is  real spherical for $ S\times H_2$ if and only if
  $[H_2\times S_2]/S_2$ is real spherical, and the  equivalence of (\ref{rs1}) and (\ref{rs2}) follows.
 Moreover (\ref{l descent}) follows from
(\ref{L-tower}).
\end{proof}

\subsection{ Factorizations}
In the situation of Proposition \ref{tower-factor}(2) we have 
$\hf'= \str (\gf, \hf') + \hf$. 
In general, if $\gf$ is a reductive Lie algebra, then a {\it factorization} of $\gf$
is a decomposition $\gf=\gf_1 + \gf_2$ of  the vector space $\gf$ as the sum
of two reductive subalgebras $\gf_1$, $\gf_2$. 
Factorizations of simple Lie algebras were 
classified by Onishchik (\cite{Oni}).  That classification was extensively used
already in Part I. For convenience we repeat it here.

It is easily seen that for a real Lie algebra $\gf$ and any two 
subalgebras $\gf_1,\gf_2$ we have the factorization $\gf=\gf_1+\gf_2$ 
if and only if $\gf_\C=\gf_{1,\C}+\gf_{2,\C}$ holds for the complexifications.
Hence it suffices to deal with the complex case.

\begin{proposition} \label{Oni2} Let $\gf$ be a
  complex simple Lie algebra and let $\gf_1$, $\gf_2$ be proper
	reductive subalgebras. Then $\gf=\gf_1+\gf_2$ if and only if
	the algebras occur in the following Table \ref{Onishchik}
	(up to isomorphism and interchange of $\gf_1$ and $\gf_2$).
\end{proposition}

\begin{table}[ht]
  \[
    \begin{array}{l l l l l l}
      \gf&\gf_1&\gf_2&\gf_1\cap\gf_2\\
      \hline
      \sl(2n,\C)&\sl(2n-1,\C)+\zf&\sp(n,\C)
                      &\sp(n-1,\C)+\zf&\zf\subset\C& n\ge2 \\
      \so(2n,\C)&\so(2n-1,\C)&\sl(n,\C)+\zf
                      &\sl(n-1,\C)+\zf&\zf\subset\C& n\ge4 \\
      \so(4n,\C)&\so(4n-1,\C)&\sp(n,\C)+\ff
                      &\sp(n-1,\C)+\ff&\ff\subset\sp(1,\C)& n\ge2 \\
      \so(7,\C)&\so(5,\C)+\zf&\sG_2^\C
                      &\sl(2,\C)+\zf& \zf\subset\C\\
      \so(7,\C)&\so(6,\C)&\sG_2^\C
                      &\sl(3,\C)\\
      \so(8,\C)&\so(5,\C)+\ff&\spin(7,\C)
                      &\sl(2,\C)+\ff& \ff\subset\sp(1,\C)\\
      \so(8,\C)&\so(6,\C)+\zf&\spin(7,\C)
                      &\sl(3,\C)+\zf& \zf\subset\C\\
      \so(8,\C)&\so(7,\C)&\spin(7,\C)
                      &\sG_2^\C\\
      \so(8,\C)&\spin(7,\C)_+&\spin(7,\C)_-
                      &\sG_2^\C\\
      \so(16,\C)&\so(15,\C)&\spin(9,\C)
                      &\spin(7,\C)\\
    \end{array}
  \]
  \centerline{\rm\Tabelle{Onishchik}}
\end{table}

In Table \ref{Onishchik} the subscripts $\pm$
to $\spin(7,\C)$ indicate representatives from 
its two conjugacy classes
in $\so(8,\C)$.

\section{Generalities and basic reductions}

\subsection{Two lemmas}

We state two simple lemmas in which
$\gf_1, \gf_2$ are reductive Lie algebras, 
$\gf=\gf_1\oplus \gf_2$, and 
$\hf\subset\gf$ is a reductive subalgebra with 
projections $\hf_1$, $\hf_2$ 
to $\gf_1, \gf_2$.  Moreover, $Z_i=G_i/H_i$.

\begin{lemma}  \label{lemma stand1}  
The pair $(\gf,\hf)$ is real spherical
if and only if the following two conditions
are both satisfied
\begin{enumerate} 
\item\label{standard3} $(\gf_1,\hf_1)$ and $(\gf_2,\hf_2)$ are real spherical, and
\item\label{standard4} $(H_1 \times H_2)/H$ is a
real spherical variety  for the action of
$\Str(G_1,H_1)\times \Str(G_2,H_2)$.
\end{enumerate}
 In this case
the action in (\ref{standard4}) is transitive.
\end{lemma}

\begin{proof} We consider the tower $H\subset H':=H_1\times H_2$ in 
$G=G_1\times G_2$.
If $G/H$ is real spherical, then so is $G/H'$, and hence (\ref{standard3}) holds.
On the other hand, if we assume (\ref{standard3}) and apply 
Proposition \ref{tower-factor},
we see that $G/H$ is real spherical if and only if (\ref{standard4}) holds. 
 The final statement comes from \eqref{HH'}.
\end{proof}

Note that as a consequence of the final statement in the lemma,
condition (\ref{standard4}) can be replaced by the following Lie algebraic version.
There exists for $i=1,2$ some minimal parabolic subalgebra $\pf_i$ of $\str(\gf_i,\hf_i)$
such that
$$\hf_1\oplus\hf_2=(\pf_1\oplus\pf_2)+\hf\,.$$

Recall from the introduction that a reductive Lie algebra $\gf$ is called 
{\it elementary} if $[\gf,\gf]$ is compact.

\begin{lemma} \label{lemma stand2}
The following assertions hold:
\begin{enumerate} 
\item\label{compact factor} If $\gf_2$ is elementary, then 
$(\gf,\hf)$ is real spherical if and only if $(\gf_1,\hf_1)$ is real spherical.
\item\label{reduction condition} If $(\gf,\hf)$ is real spherical, then so is $(\gf_1\oplus \hf_2,\hf)$.  
\end{enumerate}
\end{lemma}

\begin{proof}  For the proof of (\ref{compact factor}) we only need to 
observe that when the ideal $\gf_2$ is elementary
then it is contained in 
$\pf$.  

For (\ref{reduction condition}) we consider
the following two towers
\begin{align}
\label{tower1} &\hf\subset \hf_1 \oplus \hf_2\subset \gf\\
\label{tower2} &\hf\subset \hf_1 \oplus \hf_2\subset \gf_1 \oplus\hf_2.
\end{align}
Lemma \ref{lemma stand1} applied to (\ref{tower1}) gives
that  $(H_1\times H_2)/H$ is real
spherical for $\Str(G_1,H_1)\times \Str(G_2,H_2)$. 
Hence  it is real spherical also for $\Str(G_1,H_1)\times H_2$, and then 
the opposite implication of Lemma \ref{lemma stand1} applied to
(\ref{tower2}) implies the sphericality of $(\gf_1\oplus \hf_2,\hf)$. 
\end{proof}

\subsection{Notions of indecomposability and equivalence}

\par We say that the pair $(\gf,\hf)$ is {\it decomposable} if  there exists a 
non-trivial decomposition $\gf=\gf_1 \oplus \gf_2$ 
in ideals of $\gf$ such that $\hf=\hf_1 \oplus \hf_2$ with $\hf_i=\hf\cap \gf_i$.  
It is clear that in this case $(\gf,\hf)$ is real spherical if and only if
both components $(\gf_i,\hf_i)$ are real spherical.
If $(\gf,\hf)$ is not decomposable, then we call it {\it indecomposable}.
We conclude that we only need to classify the indecomposable 
real spherical pairs.

Furthermore, it follows immediately from Lemma \ref{lemma stand2}(\ref{compact factor})  
that it suffices to classify those real spherical pairs $(\gf,\hf)$ for which
$\gf$ is semi-simple and without compact factors.

\par However, as mentioned in the introduction we need to exclude a
certain circumstance in order to prevent that the 
classification becomes unwieldy. Recall that by $\hf_{\rm n}$ we denote  
the largest non-elementary ideal of $\hf$. 
We say that $(\gf,\hf)$ is {\it strictly indecomposable} 
provided $(\gf, \hf_{\rm n})$ is indecomposable.

By definition we consider two pairs $(\gf_1,\hf_1)$ and $(\gf_2,\hf_2)$
{\it equivalent} if there exists an isomorphism of $\gf_1$ onto $\gf_2$ which
carries $\hf_1$ onto $\hf_2$.

\smallskip\noindent
{\bf Outline.} The classification that will be given in  Sections 4 and 5 consists of 
{\it all strictly indecomposable real spherical pairs $(\gf,\hf)$, up to equivalence, 
for which 
$\gf$ is semi-simple and without compact factors and $\hf\subset \gf$ is  algebraic reductive.} 
 In Section 6 we discuss the indecomposable real spherical pairs 
$(\gf,\hf)$ as above which are not strictly indecomposable.

\subsection{A basic condition}
Before we commence with the classification we need one more observation.
We assume that $\gf =\gf_1 \oplus \ldots \oplus \gf_k$
is semi-simple with each $\gf_i$ simple and non-compact.
On the group level we let 
$G=G_1 \times \ldots\times G_k$ be a connected 
Lie group with Lie algebra $\gf$ and let 
$H\subset G$ be the connected subgroup corresponding to $\hf$. 

We write $p_i: \gf\to \gf_i$ for the various 
projections and set $\hf_i:=p_i(\hf)$.  
We also use $p_i$ for the projection on the group level $G\to G_i$
and set $H_i=p_i(H)$. It follows from Lemma \ref{lemma stand1}(\ref{standard3})
that if $(\gf,\hf)$ is real spherical then $(\gf_i, \hf_i)$ is real spherical 
for all $1\leq i \leq k$. 

For the purpose of proving our classification we can thus assume from the 
outset that each pair $(\gf_i, \hf_i)$ is real spherical.

\section{The case of two factors}\label{The case of two factors}
We assume that $\gf = \gf_1 \oplus \gf_2$ consists of two simple non-compact factors  and that 
$\hf\subset\gf$ is an algebraic  reductive subalgebra such that $(\gf,\hf)$ is a strictly indecomposable
real spherical pair. A complete classification of such pairs will be given in Theorem \ref{thm:semisimplespherical}.
\subsection{Preliminaries}\label{dec L_i} Let 
$$H'_1=\{x_1\in G_1\mid (x_1,\1)\in H\},\quad
H'_2=\{x_2\in G_2\mid (\1,x_2)\in H\}.$$
These are 
easily seen to be normal subgroups of $H_1$
and $H_2$, respectively. It follows that 
$H_1'\times H_2'$ is a normal subgroup of
$H_1\times H_2$, hence also of $H$. Let
$$H^0:=H/[H_1'\times H'_2]$$ 
 with projection map $p^0: H\to H^0$.
Note that $H^0$ is not elementary as otherwise $(\gf,\hf_n)$ would be
decomposable.

For $x\in H$ we have $p_1(x)\in H_1'\Leftrightarrow p_2(x)\in H_2'$,
and hence $p_1$ and $p_2$ induce diffeomorphisms $H^0\to H_1/H_1'$
and $H^0\to H_2/H_2'$. By means of the first we regard $H^0$ as a 
homogeneous space for $H_1$ (with trivial $H_1'$-action).
Likewise for the other, except that it will be more convenient
to equip $H^0$ with the right action of $H_2$
associated to $H^0\to H_2'\bs H_2$.

It is now easily seen that by
$$(p_1(x),p_2(y))\mapsto  p^0(xy^{-1}),\quad x,y\in H$$
we obtain a well-defined map $H_1\times H_2\to H^0$.
 This map is 
equivariant and induces an isomorphism of homogeneous spaces
\begin{equation} \label{iso H^0}
[H_1\times H_2] /H \xrightarrow{\sim} H^0.
\end{equation}

\par Let $\hf_i''\triangleleft \hf_i$ be a complementary ideal to $\hf_i'$. 
Observe that $\hf_i''\simeq \hf^0$
and 
\begin{equation} \label{decomp_h}
\hf\simeq  \hf_1' \oplus \hf^0 \oplus \hf_2' 
\end{equation}

Let  $\Str_i':= \Str(G,H_i)\cap H_i'$  and let
$S_i''\subset H_i/H_i'\simeq H^0$ be the projection of 
$\Str(G,H_i)$ for $i=1,2$.  This leads to the exact sequence 
$$ \1\longrightarrow  S_i' \longrightarrow \Str(G,H_i)\longrightarrow  S_i'' 
\longrightarrow  \1,$$
which splits on the Lie algebra level
and yields  a local isomorphism  (that is,
Lie algebras are isomorphic)
\begin{equation} \label{iso L_i}
 \Str(G,H_i)\simeq S_i' \times S_i''. 
\end{equation}

\par We recall from Lemma \ref{lemma stand1}(2) that 
$Z=G/H$ is real spherical if and only if $[H_1\times H_2]/H$ is 
real spherical for $\Str(G,H_1)\times \Str(G,H_2)$.  
Under the identifications (\ref{iso H^0}), (\ref{iso L_i}) this means that $H^0$ is real  spherical as 
 $S_1'' \times S_2''$-variety.  
Moreover, we record that in this case 
\begin{equation} \label{H^0-factor}H^0=
 [S_1'' \times S_2'']/[S_1'' \cap S_2''] 
\end{equation}
and by (\ref{L-tower}) there is  a local isomorphism
\begin{equation} \label{formula L}
 \Str(G,H)\simeq S_1'\times S_2'\times
\Str(S_1'' \times S_2'',S_1'' \cap S_2'') 
\, .
\end{equation}

\subsection{The diagrams}
In the following diagrams we describe a  real spherical pair $(\gf, \hf)$ together with its structural algebra 
$ \str(\gf,\hf)$ by a three-layered graph. The vertices in the top row consist of the simple factors 
$\gf_i$ 
of $\gf$ and those in the second row of the simple (or central) factors $\hf^{(j)}$ of $\hf$. There is an 
edge between $\gf_i$ and $\hf^{(j)}$ if $p_i(\hf^{(j)})$ is non-zero. Likewise in the third row we list 
the factors
of $ \str(\gf,\hf)$ and the edges between the middle and last row correspond to non-zero embeddings into 
the various of factors of $\hf$.
It might happen in the third rows  that certain indices become negative, e.g. $\su(p-2,q)$ for $p=0,1$,  
and then our convention is that the corresponding factor of $ \str(\gf,\hf)$ is $\{0\}$.
Cases which are absolutely spherical are marked with a.s. in the upper right corner of the corresponding box.

\begin{theorem} \label{thm:semisimplespherical} Let $\gf = \gf_1 \oplus \gf_2$ be 
the decomposition of $\gf$ into simple non-compact subalgebras, 
and let $\hf$ be  an algebraic  reductive  subalgebra such that
$(\gf,\hf)$ is strictly indecomposable.
Then $(\gf,\hf)$ is real spherical if and only if it is equivalent
to an element of
the following list. 
\begin{itemize}

\newcommand*{\pnlabel}[1]{{\rm (S\refstepcounter{pairnumber}\thepairnumber\label{#1})}}

\item[1)] If $\hf$ is simple, then $(\gf,\hf)$ is either

\begin{itemize}
	
	\item[1a)] of {\it group type}, i.e. $\gf_1=\gf_2$ and 
	
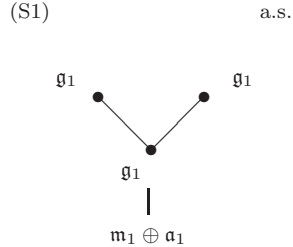
\vspace{0,5cm}
	$$\GPV{\qquad \gf_1}{\gf_1 \qquad } {\mf_1 \oplus \af_1}{case4.1} {{\rm a.s.}}$$
	with $\gf_1$ non-compact simple  and with $\mf_1$ and $\af_1$ defined as in Subsection \ref{subsub Lan}.
	
	\item[1b)] or  exactly one $\hf_i$ equals $\gf_i$. In this case, $(\gf,\hf)$ is equivalent 
	to one of the following pairs:  
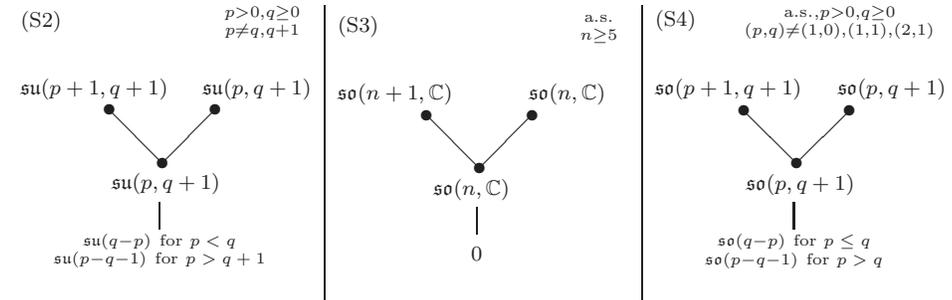
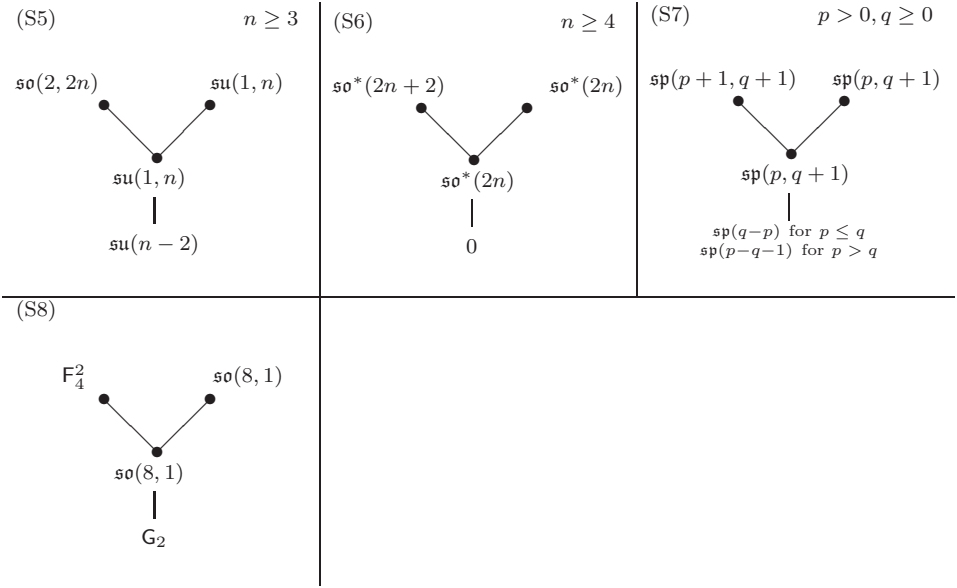
\vspace{5mm}	

\begin{tabular}{c|c|c} 
	$\GPV{\su(p+1,q+1)} {\su(p,q+1)}{{ \su(q-p) \ \text{{\rm for} 
	$p<q$}}\atop \su(p-q-1) \ \text{{\rm for} $p>q+1$}}{case4.5}
	{p>0, q\ge 0 \atop{p\neq q,q+1}}$&
	$\GPV{\so(n+1,\C)}{\so(n,\C)\quad }{0}{case4.2}{{\rm a.s.}\atop n\geq5}$ 	&
	$\GPV{\so(p+1,q+1)}{\so(p,q+1)}{{\so(q-p) \ \text{{\rm for} $p\leq q$}}\atop \so(p-q-1) \ 
	\text{{\rm for} $p>q$}}{case4.3}{{\rm a.s.}, p>0,  q\ge 0\atop(p,q)\neq(1,0), (1,1), (2,1)}$
	\\  \hline 
\end{tabular}

\begin{tabular}{c|c|c} 
	$\GPV{\so(2,2n)}{\su(1,n)\quad }{\su(n-2)}{case4.61}{n\geq3}$ 	&
	$\GPV{\so^*(2n+2)}{\so^*(2n)}{0}{case4.62}{n\geq 4}$& 
	$\GPV{\sp(p+1,q+1)} {\sp(p,q+1)}{{ \sp(q-p) \ \text{{\rm for} $p\leq q$}}\atop \sp(p-q-1) \ 
	\text{{\rm for} $p>q$}}{case4.4}{p>0, q\ge 0  }$
	\\  \hline 
	$\GPV{\qquad  \sF_4^2}{\so(8,1)\quad}{ \sG_2}{case4.6}{ }$    
\end{tabular}

\vspace{5mm}

\end{itemize}

\renewcommand*{\pnlabel}[1]{{\rm (SS\refstepcounter{pairnumber}\thepairnumber\label{#1})}}

\item[2)] If $\hf$ is semi-simple but not simple, then either

\begin{itemize}
\item[2a)] $\hf_2 = \gf_2$ and $(\gf,\hf)$ is equivalent to one of the 
following pairs, where $\K,\K'=\R$ or $\C$
and $\K'\subset\K$
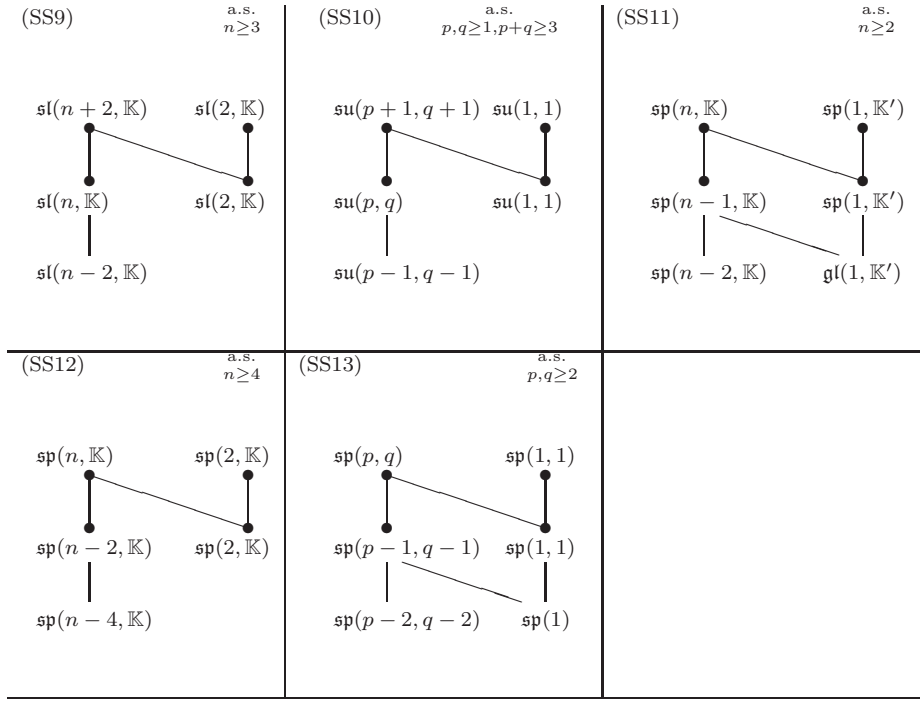
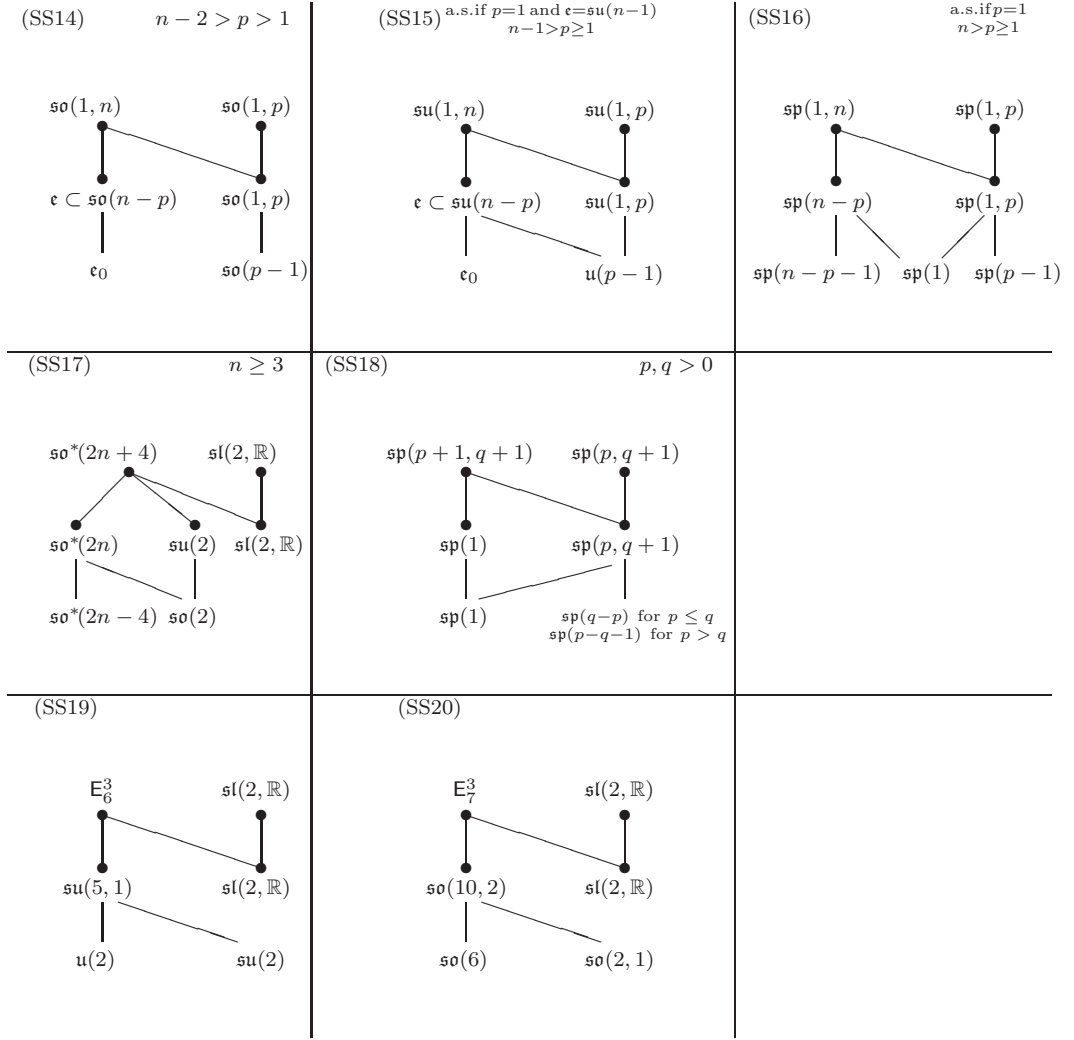
\vspace{5mm}

\noindent 
\begin{tabular}{c|c|c} 
$\GPNi{\sl(n+2,\K)}{\sl(2,\K)}{\sl(n,\K)}{ \sl(n-2,\K)}{case4.7}{{{\rm a.s.}\atop n\geq 3}}$ 
& $\GPNi{\su(p+1,q+1)}{\su(1,1)}{\su(p,q)}{\su(p-1,q-1)}{case4.9}{{\rm a.s.}\atop 
{p,q\ge 1, p+q\geq  3}}$
& $\GPNiv{\sp(n,\K)}{\sp(1,\K')}{\sp(n-1,\K)}{\sp(n-2,\K)}{\gl(1,\K')}{case4.10}{{\rm a.s.}\atop n\geq 2}$ 
\\
\hline 
$\GPNi{\sp(n,\K)}{\sp(2,\K)}{\sp(n-2,\K)}{\sp(n-4,\K)}{case4.12}{{ \rm a.s.}\atop n\geq {4}}$ & 
$\GPNiv{\sp(p,q)}{\sp(1,1)}{\sp(p-1,q-1)}{\sp(p-2,q-2)}{ \ \  \sp(1)}{case4.14}{{\rm a.s.}\atop 
{p,q\ge 2}}$ 
&

\\ \hline
\end{tabular}

\begin{tabular}{c|c|c} 
$\GPNii{\so(1,n)}{\so(1,p)}{\ef\subset \so(n-p)} { \ef_0}{\so(p-1)}{case4.16}{ n-2>p>1 }$
& $\GPNiv{\su(1,n)}{\su(1,p)}{\ef\subset \su(n-p)} {\qquad \ef_0}{\uf(p-1)}{case4.17}
{{{\rm a.s. if} \,p=1 \,{\rm and}\, \ef=\su(n-1)}\atop n-1> p\geq 1}$

& $\GPNv{\sp(1,n)}{\sp(1,p)}{\sp(n-p)} {\!\!\!\!\!\!\!\!\sp(n-p-1)}
{\!\!\!\!\!\!\!\!\!\!\!\!\!\!\sp(1)\quad \sp(p-1)}{case4.18}{{{\rm a.s. if} p=1}\atop n>p{\ge}1}$
\\ \hline
$\GPANl{\so^*\!(2n+4)}{\sl(2,\R)}{\so^*\!(2n)}{\su(2)}{\so^*\!(2n-4)}{\so(2)}{case4.19}{n\geq 3}$  
& $\GPNisp{\sp(p+1,q+1)}{\sp(p,q+1)}{\sp(1)}{\sp(1)}{{{ \sp(q-p) \ \text{{\rm for} $p\leq q$}}\atop \sp(p-q-1) \ 
\text{{\rm for} $p>q$}} } {case4.20}{p, q>0 }$  \\
\hline
$\GPNn{\sE_6^3}{\sl(2,\R)}{\su(5,1)}{\uf(2)}{\hspace{2mm} \su(2)}{case4.22}$ &	
$\GPNn{\sE_7^3}{\sl(2,\R)}{\so(10,2)}{\so(6)}{\so(2,1)}{case4.21}$ 

\end{tabular}
	
\vspace{5mm}	
For the various possibilities 
of $\ef$ and $\ef_0$ in {\rm (SS\ref{case4.16})}
and {\rm (SS\ref{case4.17})}
see Table \ref{e-table}.

	\item[2b)] or $\hf_i \neq \gf_i$ for $i = 1,2$ and $(\gf,\hf)$ is equivalent to one of the pairs where
	$\K,\K',\K''=\R,\C$ and  $\K''\subset\K\subset \K'$
\vspace{5mm}

\noindent \begin{tabular}{c|c} 
  $\repMii{\sl(n+2,\K)}{\sp(m+1,\K')}{\sl(n,\K)}{\sl(2,\K)}{\sp(m,\K')}{\sl(n-2,\K)}{\sp(m-1,\K')}{case4.23} { {\rm a.s. 
	if}\ \K=\K' \atop m\geq 1, n\geq 3} $ 
&  $\repMii{\su(p+1,q+1)}{\sp(m+1,\K)}{\su(p,q)}{\su(1,1)}{\sp(m,\K)}{\su(p-1,q-1)}{\sp(m-1,\K)}{case4.26}
{{ {\rm a.s. if}\  \K=\R}\atop { m\geq 1, p,q\ge  1 \atop p+q\geq 3}}$\\
\hline 

\end{tabular} 

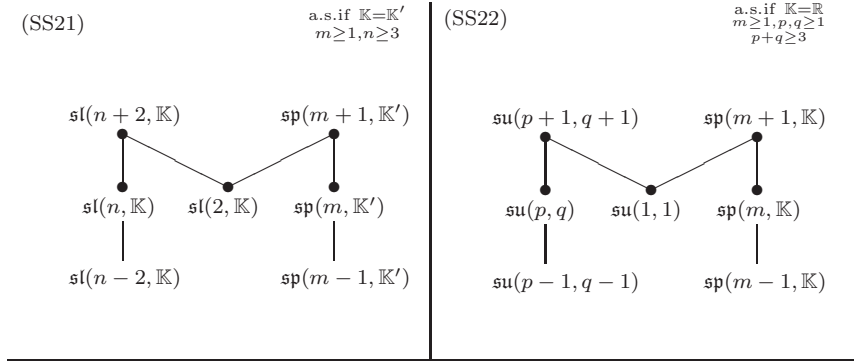
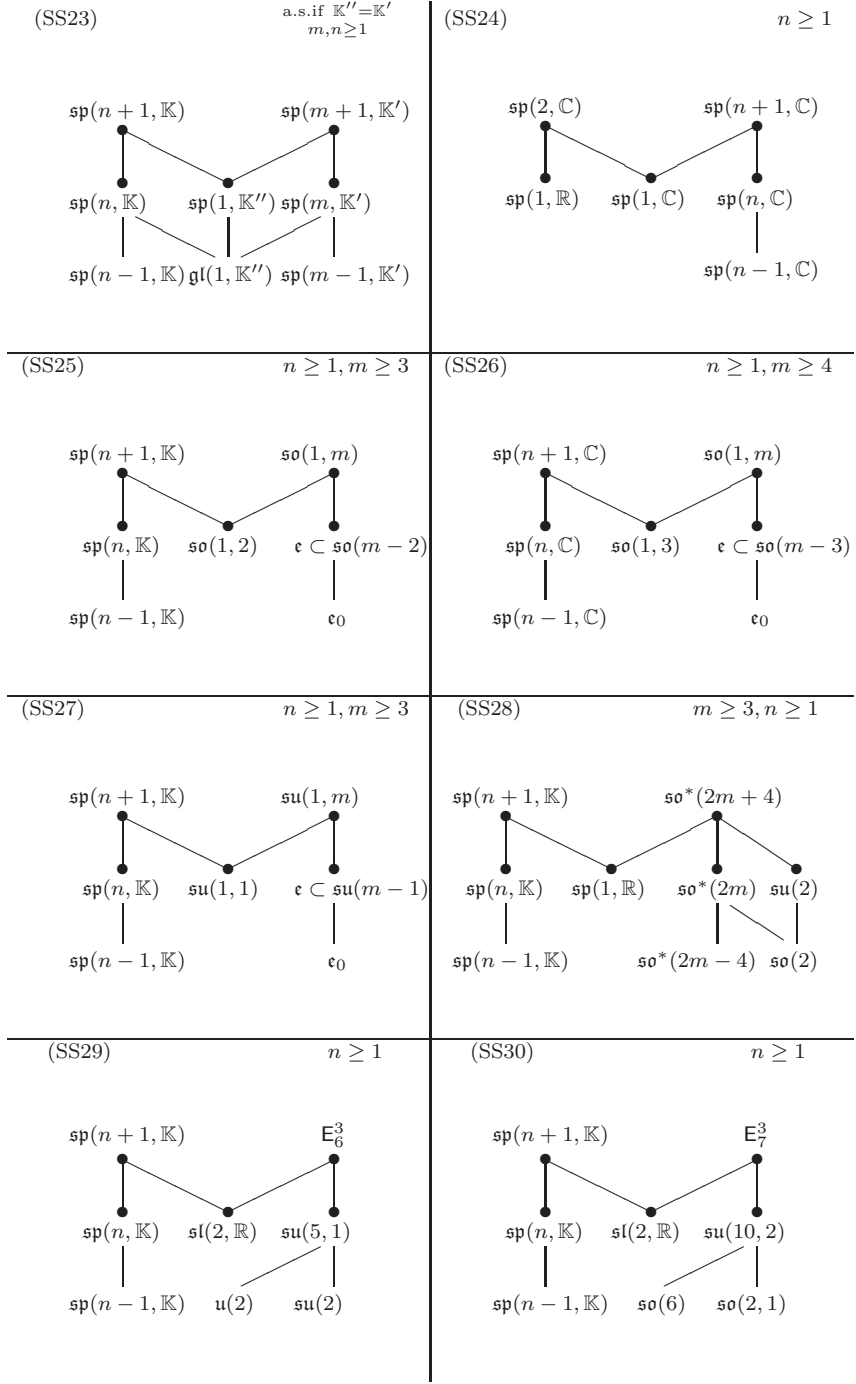
\begin{tabular}{c|c}
$\repMivi{\sp(n+1,\K)}{\sp(m+1,\K')}{\sp(n,\K)}{\sp(1,\K'')}{\sp(m,\K')}{\sp(n-1,\K)}{\gl(1,\K'')}{\sp(m-1,\K')}
{ {\rm a.s. 
	if}\  \K''=\K'\atop m,n\geq 1 }
{case4.27}$ 
& $\repMi{\ \ \sp(2,\C)}{\sp(n+1,\C)}{\sp(1,\R)}{\sp(1,\C)}{\sp(n,\C)}{\sp(n-1,\C)}{case4.28}{n\geq1 } $
\\ \hline
$\repMii{\sp(n+1,\K)}{\so(1,m)}{\sp(n,\K)}{\so(1,2)}{\ef\subset \so(m-2)}{\sp(n-1,\K)}{\qquad \ef_0}{case4.35}{n\geq1, m\geq 3 } $
&$\repMii{\sp(n+1,\C)}{\so(1,m)}{\sp(n,\C)}{\so(1,3)}{\ef\subset \so(m-3)}{\sp(n-1,\C)}{\qquad \ef_0}{case4.37}{n\geq 1, m\geq 4 }$\\ \hline 
$\repMii{\sp(n+1,\K)}{\su(1,m)}{\sp(n,\K)}{\su(1,1)}{\ef\subset \su(m-1)}{\sp(n-1,\K)}{\qquad \ef_0}{case4.36}{n\geq1, m\geq  3} $&$\repMAiii{\sp(n+1,\K)}{\so^*(2m + 4)}{\sp(n,\K)}{\sp(1,\R)}{\so^*(2m)}{\su(2)}{\sp(n-1,\K)}{\so^*(2m-4)}{\so(2)}{case4.31}{m\geq3, n\geq1}$
 \\ 
 \hline 



$\repMin{\sp(n+1,\K)}{\sE_6^3}{\sp(n,\K)}{\sl(2,\R)}{\su(5,1)}{\sp(n-1,\K)}{\uf(2)}{\su(2)}{n\geq 1}{case4.33}$&
 $\repMin{\sp(n+1,\K)}{\sE_7^3}{\sp(n,\K)}{\sl(2,\R)}{\su(10,2)}{\sp(n-1,\K)}{\so(6)}{\so(2,1)}{n\geq 1}{case4.38}$
	\end{tabular}
\vspace{5mm}
\end{itemize}

\renewcommand*{\pnlabel}[1]{{\rm (R\refstepcounter{pairnumber}\thepairnumber\label{#1})}}

\item[3)] If $\hf=[\hf,\hf]+\zf$ is not semi-simple, then there are two classes: 
\begin{itemize}
\item[3a)]  $(\gf,[\hf,\hf])$ 
is already real spherical, hence belongs to {\rm 1)} or 
{\rm 2)}. 
Then 
$(\gf,\hf)$ appears in the following list, 
where $\K\subset\K'$: 

\

\noindent
\begin{tabular}{c|c} 
$\GPNisp{\su(p+1,q+1)}{\su(p,q+1)}{\uf(1)}{\uf(1)}
{{{ \su(q-p) \ \text{{\rm for} $p< q$}}\atop \su(p-q-1) 
\ \text{{\rm for} $p>q+1$}} } {case4.48}
{{\rm a.s. }, p>0,q\ge 0 \atop{p\neq q,q+1}}
\mkern40mu$  
&
$\GPNisp{\sp(p+1,q+1)} {\sp(p,q+1)}{\uf(1)}{\uf(1)}
{{{ \sp(q-p) \ \text{{\rm for} $p\le q$}}\atop \sp(p-q-1) 
\ \text{{\rm for} $p> q$}}}
{case4.49}
	{p>0,q\ge 0}\mkern40mu$
\\ \hline

$\GPNi{\sl(n+2,\K)}{\sl(2,\K)}{\ff+\sl(n,\K)}{\ff+\sl(n-2,\K)}{case4.50}
{n\ge 3, 0\neq\ff\subset\gl(1,\K)\atop
{\rm a.s. if }\, \ff=\gl(1,\K)}$ 
&
$\GPNi{\su(p+1,q+1)}{\su(1,1)}{\uf(p,q)}{ \uf(p-1,q-1)}
{case4.51}
{{{\rm a.s.}\atop p+q\geq 3, p,q\ge1}}$
\\ \hline
$\GPNiv{\su(1,n)}{\su(1,p)}
{\mathfrak{i}\subset \uf(n-p)} 
{\qquad\mathfrak{i}_0}
{\uf(p-1)}{case4.52}{{\rm a.s. if}\  \mathfrak{i}=\uf(n\!-\!p)\ {\rm and}\
p=1,n\!-\!1 \atop n-1\ge p\geq 1}$
&
$\GPNii{\so(1,n)}{\so(1,p)}{{\mathfrak i} \subset \so(n-p)} 
{ \mathfrak{i}_0}{\so(p-1)}{case4.53a}{ n-2>p>1 }$

\end{tabular}

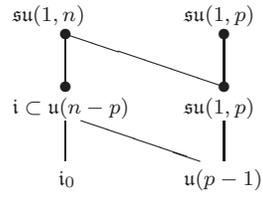
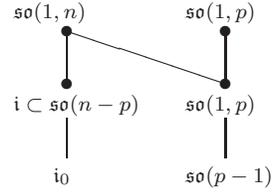
\begin{tabular}{c|c} 
	
	$\repMii{\sl(n+2,\K)}{\sp(m+1,\K')}{\mkern-20mu\ff+\sl(n,\K)}
	{\sl(2,\K)}{\sp(m,\K')}
	{\mkern-10mu\ff+\sl(n-2,\K)}
	{\sp(m-1,\K')}{case4.53}
	{ m\ge 1, n\ge 3, 0\neq\ff\subset\gl(1,\K)\atop
{\rm a.s. if }\, \K=\K'\, {\rm and }\,\ff=\gl(1,\K)}$ 
&
$\repMii{\su(p+1,q+1)}{\sp(m+1,\K)}{\uf(p,q)}{\su(1,1)}{\sp(m,\K)}
{\uf(p-1,q-1)} {\sp(m-1,\K)}{case4.54}
{ m\geq 1, p,q\ge 1\atop p+q\geq 3}$

\\ \hline

$\GPNl{\so^*(2n+2)}{\so^*(2n)}{\so^*(2)=\uf(1)}{0}{case4.45}
{ n\geq4}$ 
&$\GPNisp{\so(2,2n)}{\su(1,n)}{\uf(1)}
{\uf(1)\qquad\qquad\su(n-2)}{}{case4.46}{n\geq3 } $

\\ \hline


$\repMii{\sp(n+1,\K)}{\so(1,m)}{\sp(n,\K)}{\so(1,2)}{\mathfrak{i}\subset \so(m-{2})}
{\sp(n-1,\K)}{\qquad \mathfrak{i}_0}{case4.55a}{n\geq1, m\geq 3} $
&

$ \repMii{\sp(n+1,\C)}{\so(1,m)}{\sp(n,\C)}{\so(1,3)}{\mathfrak{i}\subset \so(m-{3})}
{\sp(n-1,\C)}{\qquad \mathfrak{i}_0}{case4.55b}{n\geq1, m\geq 4} $

\\ \hline

$\repMii{\sp(n+1,\K)}{\su(1,m)}{\sp(n,\K)}{\su(1,1)}{\mathfrak{i}\subset \uf(m-1)}
{\sp(n-1,\K)}{\qquad \mathfrak{i}_0}{case4.55}{n\geq1, m\geq 3} $ 
&
\end{tabular}

For the various possibilities for $({\mathfrak i}, {\mathfrak i}_0)$ see Table \ref{i-table}.

\item[3b)] $(\gf,[\hf,\hf])$ is not real spherical. 
Then 
$(\gf,\hf)$ appears in the following list 
where $\K\subset \K'$: 

\

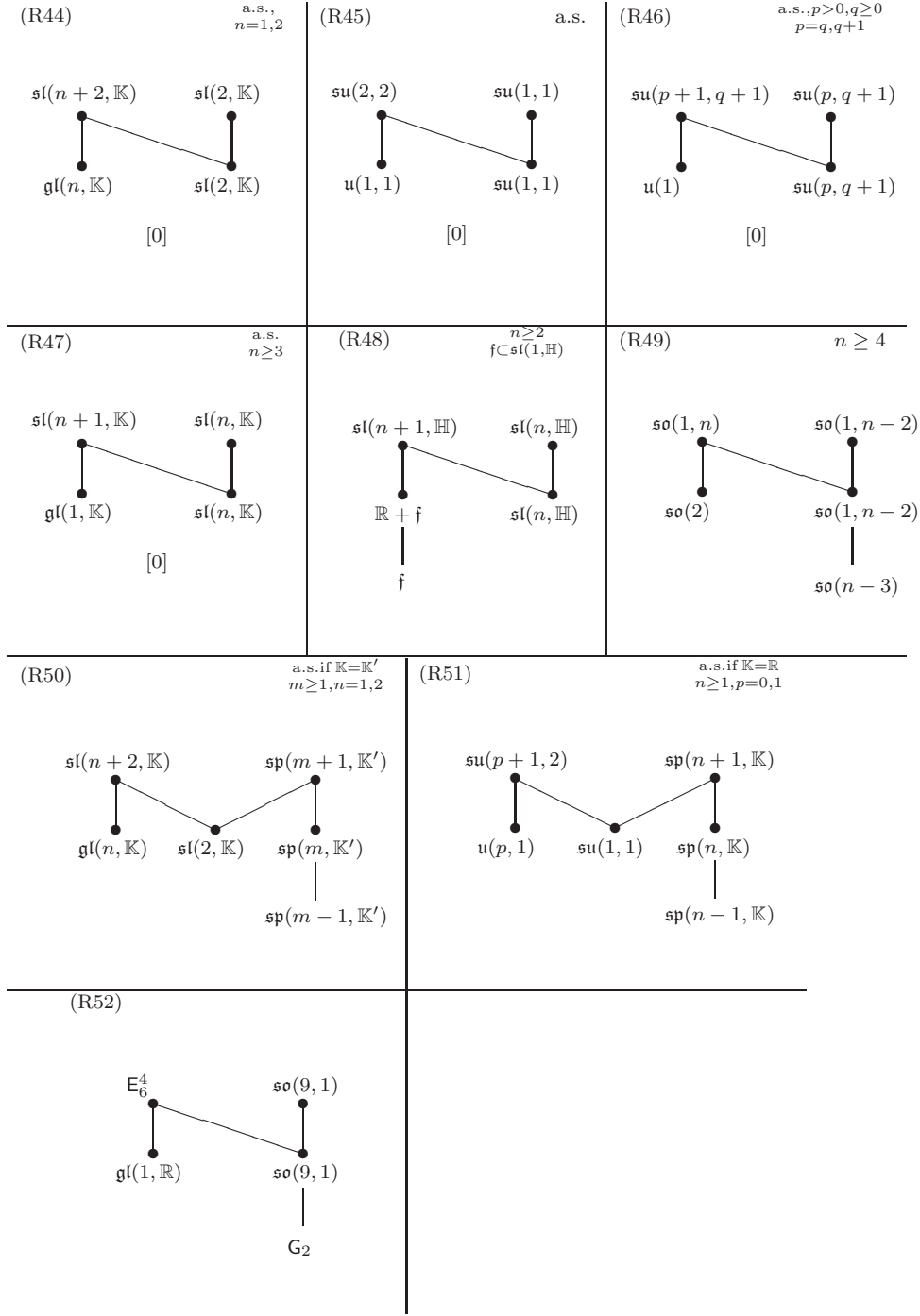
\noindent
\begin{tabular}{c|c|c} 
$\GPNl{\sl(n+2,\K)}{\sl(2,\K)}{\mkern-1mu\gl(n,\K)}
{0}{case4.56}{{{\rm a.s., } \atop n=1,2}}$ 
&
$\GPNl{\su(2,2)}{\su(1,1)}{\uf(1,1)}{0}{case4.57}
{{{\rm a.s.}} }$
&
$\GPNl{\su(p+1,q+1)}{\su(p,q+1)}{\mathfrak{u}(1)}{0}{case4.42}{{\rm a.s.}, p>0, q\ge 0\atop  p=q,q+1}$ 
\\ \hline
$\GPNl{\sl(n+1,\K)}{\sl(n,\K)}{\gl(1,\K)}{0}{case4.40}{{\rm a.s.}\atop n\geq 3}$
&

$\GPNi{\sl(n+1,\Hb)}{\ \sl(n,\Hb)}{\quad \R+\ff}{\hspace{18pt}\ff}{case4.43}{n\geq 2 \atop \ff \subset  \sl(1,\Hb)}$
& 
$\GPNll{\so(1,n)}{\so(1,n-2)}{\so(2)}{\so(n-3)}{case4.59}{n\geq4 } $ 
\\ \hline 
\end{tabular}

\noindent
\begin{tabular}{c|c} 
  $\repMi{\sl(n+2,\K)}{\sp(m+1,\K')}
	{\gl(n,\K)}{\sl(2,\K)}{\sp(m,\K')}
	{\sp(m-1,\K')}{case4.58} 
	{{\rm a.s. if}\, \K=\K'\atop{m\geq 1, n=1,2} }$ 
	&
$\repMi{\su(p+1,2)}{\sp(n+1,\K)}{\uf(p,1)}{\su(1,1)}{\sp(n,\K)}
{\sp(n-1,\K)}{case4.60}{{\rm a.s. if}\, \K=\R
\atop n\geq1, p=0,1} $
\\ \hline

$\GPNll{\ \ \ \sE_6^4}{\ \so(9,1)}{\gl(1,\R)}{ \quad\sG_2}{case4.47}{ }$ 
&
\end{tabular}

\vspace{5mm}

\end{itemize}
\end{itemize}
\end{theorem}

This theorem will be proven in various subsequent lemmas.

\subsection{The case where $\hf$ is simple}

If $\hf$ is simple, we have $\hf=\hf^0$ with respect to \eqref{decomp_h}. 
We can assume that $(\gf,\hf)$ is not of group type.  The building blocks for this situation 
will be when $\hf$ is isomorphic 
to one factor, say $\hf\simeq \gf_2$.  
For the treatment of this case we use simplified notation:  Let $\gf$ be a reductive
Lie algebra and $\hf\subset \gf$  a  reductive subalgebra.  Recall that we say that 
$(\gf,\hf)$ is 
{\it strongly real spherical}  if the pair $ (\gf\oplus \hf,\diag(\hf))$ is real spherical.  In the first step 
we classify all strongly real spherical pairs with $\gf$ and $\hf$ both non-compact simple. 

We let $Q=LU$ be the parabolic subgroup 
adapted to $G/H$, as in the local structure theorem,
and put $L_H=L \cap H$. 
An immediate consequence of Corollary \ref{tower-corollary} then is: 

\begin{cor} \label{grossprasad_cond}  Let $(\gf,\hf)$ be a real spherical pair with structural algebra $ \sf=\str(\gf,\hf)$,
where $\gf$ is reductive and $\hf$ reductive in $\gf$.
Then the following conditions are equivalent: 
\begin{enumerate}
\item $(\gf,\hf)$ is strongly real spherical. 
\item $(\hf, \sf)$ is strongly real spherical. 
\end{enumerate} 
Moreover if these conditions hold, then 
\begin{equation} \label{L descent} 
\str(\gf\oplus\hf,\hf)=\str(\hf\oplus \sf, \sf)
\, .\end{equation}
\end{cor}
 
\begin{proof} Let $ \hf_1=\{0\}$ and $ \hf_2=\hf$ 
in Corollary \ref{tower-corollary}.
\end{proof} 

Note that in particular, if $(G,H)$ is 
strongly real spherical, we have that $L_H$ is a spherical subgroup of $H$.

\begin{lemma} \label{strong spherical}
Let $\gf$ be a simple Lie algebra and $\hf\subsetneq \gf$ a simple  non-compact real spherical 
subalgebra.  Let $\sf=\sf(\gf,\hf)$ denote the
associated structural algebra.

\begin{itemize}

\item[(a)]  The pair $(\hf,\sf)$ is real spherical
if and only if $(\gf,\hf)$ belongs (up to equivalence)
to the following list of triples $(\gf,\hf,\sf)$:
\begin{enumerate} 
	\item\label{i1} $(\sl(n+2,\K), \sl(n+1,\K), \sl(n,\K))$  for $\K=\R,\C,\Hb$ and $n\geq 1$
	\item \label{i2}$(\su(p+1, q+1), \su(p,q+1),\su(p,q))$ for $p+q\geq 1$ 
        \item \label{i4}$(\so(n+1, \C), \so (n,\C), \so(n-1,\C))$  for $n\geq 5$
	\item \label{i5}$(\so(p+1, q+1), \so (p,q+1),\so (p, q))$ for $p+q\geq 2$ and $(p,q)\neq (1,1), (2,1)$
	\item \label{i6} $(\so^*(2n+4), \so^*(2n+2), \so^*(2n))$  for $n\geq 3$
	\item \label{i7} $(\so(2,2n), \su(1,n), \su(n-1)\oplus \su(1,1))$  for $n\geq 3$
	\item \label{i8} $(\so(7,\C), \sG_2^\C, \sl(3,\C))$
	\item \label{i9}$(\so(3,4), \sG_2^1, \sl(3,\R))$ 
	\item \label{i10}$(\sp(p+1,q+1), \sp(p,q+1),\sp(p,q))$ for $p>0$ 
        \item \label{i11}$(\sF_4^\C, \so(9,\C), \spin(7,\C))$
	\item \label{i12} $(\sF_4^1, \so(4,5), \spin(3,4))$
	\item \label{i13} $(\sF_4^2, \so(1,8), \spin(7))$
	\item \label{i14} $(\sG_2^\C, \sl(3,\C), \sl(2,\C))$ 
	\item \label{i15} $(\sG_2^1, \sl(3,\R), \sl(2,\R))$
	\item \label{i16} $(\sG_2^1, \su(2,1), \sl(2,\R))$ 
\end{enumerate}

\item[(b)] The strongly real  spherical pairs $(\gf, \hf)$  are:  

\qquad {\rm (\ref{i2})} for $p\notin\{q,q+1\}$,
{\rm (\ref{i4})}, {\rm (\ref{i5})}, 
{\rm (\ref{i6})}, {\rm (\ref{i7})}, {\rm(\ref{i10})} and {\rm(\ref{i13})}.
\par\noindent
Moreover, the structural algebras $\str(\gf\oplus\hf,\hf)$ are given by 
$$\begin{tabular}{c l}
$(\gf,\hf)$  & $\str(\gf\oplus\hf,\hf)$ \\ \hline
{\rm(\ref{i2})} &  $\begin{cases}  \su(q-p) & \text{for $p<q$}  \\\su(p-q-1) & \text{for $p>q+1$} \end{cases}$\\
{\rm(\ref{i4})} & $ \{0\}$\\ 
{\rm(\ref{i5})} &  $\begin{cases}  \so(q-p) & \text{for $p\leq q$}  \\\so(p-q-1) & \text{for $p>q$} \end{cases}$\\
{\rm(\ref{i6})} & $\{0\}$\\
{\rm(\ref{i7})} &  $\su(n-2)$\\
{\rm(\ref{i10})} &  $\begin{cases}  \sp(q-p) & \text{for $p\leq q$}  \\\sp(p-q-1) & \text{for $p>q$} \end{cases}$\\
{\rm (\ref{i13})} & $\sG_2$\\
\end{tabular}$$

\end{itemize}
\end{lemma}

\begin{proof}
Part (a) can be easily obtained 
from Tables \ref{lcaph_class_symm} - \ref{lcaph_KKPS_class}. We proceed with (b) and
let $(\gf,\hf)$ be a pair such that $(\gf,\hf, \sf)$ satisfies (a). 

In the cases where $(\gf,\hf)$ is complex, we can refer to \cite{Brion}, \cite{Mik}. 
Hence, it follows that (\ref{i4}) is strongly real spherical, 
while  (\ref{i8}), (\ref{i11}) and  (\ref{i14}) are not. 
Further (\ref{i12}) is a split real form of (\ref{i11}) and hence is not strongly 
real  spherical. However, (\ref{i5}) is a real form of (\ref{i4}), hence it is 
also strongly real  spherical  by Lemma~\ref{realform_of_complexspherical}.  

\par There is an alternative to verify that $(\ref{i4})$ is strongly spherical,
which has the advantage that it also yields the structural algebra $\str(\gf\oplus\hf,\hf)$.  
Indeed by Corollary \ref{grossprasad_cond},
$(\so(n+1,\C),\so(n,\C))$ is strongly real spherical if and only if $(\so(n,\C),\so(n-1,\C))$ 
is strongly real spherical, and the structural algebras for these two pairs are isomorphic.  
Hence iteration brings us down to the case $(\gf,\hf)=(\so(4,\C),\so(3,\C))$.
Because of the isomorphism $\so(4,\C)=\so(3,\C)\oplus\so(3,\C)$ this pair is 
known to be strongly real spherical with trivial structural algebra. 

In a similar way we can treat (\ref{i5}). If $p=q$, then both $\gf$ and $\hf$ 
are quasisplit, and $\str(\gf\oplus\hf,\hf)=\{0\}$ by what we have just shown. 
Suppose that $p<q$. Iteration brings us down to 
$(\so(1, q+1-p), \so(q-p+1))$ which is strongly real spherical by Lemma \ref{lemma stand2}(\ref{compact factor}).  
Moreover, $\str(\gf\oplus\hf,\hf)=
\so(q-p)$.  If $p>q$, then iteration leaves us with $(\so(p-q,1), \so(p-q))$ which has structural algebra
$\str(\gf\oplus\hf,\hf)=\so(p-q-1)$. 

We proceed analogously in (\ref{i2})  and (\ref{i10}).
For $p\leq q$ the sequence for $(\gf,\hf) = (\su(p+1,q+1),\su(p,q+1))$ 
ends up at $(\su(1,q-p+1),\su(q-p+1))$  which is strongly real spherical if
and only if $p<q$  by Lemma \ref{lemma stand2}(\ref{compact factor}). 
Moreover, if $p<q$ one has $\str(\gf\oplus\hf,\hf)=\su(q-p)$.  
Similarly, for $p>q$, the sequence 
terminates at $(\su(1,p-q),\su(p-q))$ which is strongly real spherical 
if and only if $p>q+1$ and then has $\str(\gf\oplus\hf,\hf)=
\su(p-q-1)$. The sequence for (\ref{i10}) terminates at $(\gf,\hf)=(\sp(1,q-p+1), \sp(q-p+1))$ 
when $p\leq q$, and at
$(\gf,\hf)= (\sp(1,p-q), \sp(p-q))$ when $p>q$. Both are strongly real  spherical 
by Lemma \ref{lemma stand2}(\ref{compact factor}). 
Moreover we obtain for $\str(\gf\oplus\hf,\hf)$   that it is equal to 
$\sp(q-p)$ for $p\leq q$ and $\sp(p-q-1)$ if $p>q$.

The case (\ref{i1}) can also be determined 
by iteration. Here we end with $(\gf,\hf)=(\sl(2,\K), \sl(1,\K))$ 
which cannot be 
strongly real  spherical as it is not even real spherical (see  Part I). 
Hence (\ref{i1}) is not real spherical.

\par We turn to (\ref{i7}).   It follows from Corollary 
\ref{tower-corollary} that $(\so(2,2n), \su(1,n))$ is 
strongly real spherical if and only if $(\su(1,n), \su(1,1)\oplus \su(n-1))$ is strongly real spherical. 
Applying Corollary \ref{tower-corollary} again we obtain that 
this is the case if and only if 
$(\su(1,1)\oplus\su(n-1), \diag (\uf(1)) \oplus \su(n-2))$ is strongly spherical.  Since 
$\su(n-1)$ and $\uf(1)$ are both compact the latter is equivalent to $(\su(1,1), \uf(1))$ being real spherical, which is true.
Further we obtain from Corollary \ref{tower-corollary} that 
\begin{eqnarray*}\str (\so(2,2n)\oplus \su(1,n), \su(1,n)) &=&  \str(\su(1,1)\oplus\su(n-1) \oplus \uf(n-2) , \uf(n-2))\\ 
&=&\su(n-2) \oplus \str(\su(1,1)\oplus \uf(1),\uf(1))=\su(n-2)\, .\end{eqnarray*}

The case (\ref{i6}) is again determined by induction. The sequence
terminates at $(\so^*(8),\so^*(6))\simeq (\so(2,6), \su(1,3))$
which is strongly real  spherical
with structural algebra $\{0\}$ by case (\ref{i7}).
 
\par It follows from Corollary \ref{grossprasad_cond} and the dimension bound  (\ref{dim bd}), that 
the cases (\ref{i9}), (\ref{i15}) and (\ref{i16}) are not strongly real spherical.

\par Finally (\ref{i13}) is strongly real  spherical  if and only if 
$(\so(8,1),\spin(7))$ is strongly real spherical, and by Lemma
\ref{lemma stand2}(\ref{compact factor})
if and only if that pair is real spherical. 
This is actually the case, by the classification in  Part I. 
Moreover, since $\Spin(7)\cap \SO(7)=\sG_2$ 
 (see \cite[p. 170]{VSV})
it follows that 
$\str(\gf\oplus\hf,\hf)=\sG_2$.
\end{proof}

Let us now discuss the general case with $\hf$ simple.

\begin{prop} \label{prop-simple}For $k=2$, a
strictly indecomposable real spherical pair $(\gf,\hf)$ with $\hf$ 
simple arises from a strongly real spherical pair, i.e.  $(\gf,\hf)=(\gf'\oplus\hf',\hf')$ with 
$(\gf',\hf')$ strongly spherical.  The cases are listed in  Theorem~\ref{thm:semisimplespherical} (1).
\end{prop}

\begin{proof} The strongly spherical pairs were determined in Lemma 
\ref{strong spherical}(b) and they correspond exactly to diagrams
(S\ref{case4.5}) -- (S\ref{case4.6}). Thus it only
remains to exclude further cases.

For this we can assume that $(\gf,\hf)=(\gf_1\oplus\gf_2,\hf)$ 
with $\gf_1, \gf_2 \neq \hf$.  Let $S_i=\Str(G,H_i)$ for 
$i=1,2$. 
Then 
$H=H^0$ is real spherical for the transitive
action of $S_1 \times S_2$
(see Lemma \ref{lemma stand1}). 
Further it follows from Lemma \ref{lemma stand2} (\ref{reduction condition})
that both  $ [G_1 \times H ]/ H $ and $ [G_2\times H ]/ H$ need to be real spherical, i.e. both 
$(\gf_i, \hf)$ need to be one of the strongly real spherical pairs of Lemma \ref{strong spherical}. 
Using the lemma we can easily check that there are only the
following possibilities, and we claim that none of these lead to a real spherical pair
$(\gf_1\oplus\gf_2,\hf)$.
\begin{alignat*}{3}
&\gf_1=\so(n+1,\C)  &\quad&\gf_2=\gf_1                               &\quad&\hf=\so(n,\C)\\ 
&\gf_1=\so(p+1,q)   &&\gf_2=\gf_1 \text{ or } \so(p,q+1)        &&\hf=\so(p,q)         \\
&\gf_1=\su(p+1,q)   &&\gf_2=\gf_1 \text{ or } \su(p,q+1)        &&\hf=\su(p,q)         \\
&\gf_1=\sp(p+1,q)   &&\gf_2=\gf_1 \text{ or } \sp(p,q+1)        &&\hf=\sp(p,q)         \\
&\gf_1= \sF^2_4     &&\gf_2=\gf_1 \text{ or } \so(2,8) \text{ or } \so(1,9)       &&\hf=\so(1,8)         
\end{alignat*}

Consider first $\gf_1=\gf_2=\so(n+1,\C)$ and $\hf=\so(n,\C)$ with $n\ge 5$.  
Note that $ S_1=S_2=\SO(n-1,\C)$ 
up to isomorphism, and hence
$H=\SO(n,\C)$ is not homogeneous for 
$ S_1\times S_2$
by the factorization result of Onishchik (see  Proposition \ref{Oni2}), except when $n=8$.
However, a simple dimension count  with (\ref{dim bd})
excludes that $\SO(8,\C)$ is real spherical
for $\SO(7,\C)\times\SO(7,\C)$. 

\par The second case is more subtle. If this 
leads to a real spherical pair, 
then $\hf=\so(p,q)= \sf_1+\sf_2$ with
$ \sf_1,\sf_2=\so(p,q-1)$ or $\so(p-1,q)$, up to isomorphism.
According to Onishchik's list, this is only possible if $p+q=8$,
so that $\hf_\C=\so(8,\C)=\so(7,\C)+\spin(7,\C)$.
In addition, it follows from  Part I Lemma 2.4 that 
 $\sf_1$ and $\sf_2$ are both non-compact.  
Since moreover $H$ is real spherical for 
the action of  $S_1\times S_2$
we conclude by dimension count that
 $\sf_1$ or $\sf_2$ must be isomorphic to $\so(1,6)$. 
This forces that $\hf$ is either isomorphic to $\so(2,6)$ or $\so(1,7)$, 
and hence in particular its real rank is at most $2$.
The embedding of $\spin(7,\C)$ in $\so(8,\C)$ is given by the
8-dimensional spin representation. This representation is known to
yield a real representation of $\spin(p,q)$ where $p+q=7$
if and only if $p-q$ is congruent to $0$, $1$ or $7$ modulo $8$.
This then happens only for $\spin(7)$ and $\spin(3,4)$. 
The first possibility is excluded since  $\sf_1$ and $\sf_2$ 
are supposed to be non-compact, and the other one is excluded
as $\spin(3,4)$ has real rank $3$ and hence does not embed into $\hf$.

\par  In the remaining three cases 
the result of Onishchik immediately excludes the 
relevant factorizations of $\hf$.
\end{proof}

\subsection{The case where $\hf$ is semi-simple and not simple}

We begin with the case where
the projection $\hf\to\hf_1$ is an isomorphism
and where $\gf_2=\hf_2$.
The general case will be built on this.
The next lemma classifies these cases. 

\begin{lemma} \label{verification 2a}
For $k = 2$, the list of all spherical pairs $(\gf,\hf)$ with $\hf$ semi-simple but not simple, 
$\hf \simeq \hf_1$ and $\gf_2 = \hf_2$ is given by part {\rm (2a)} of Theorem~\ref{thm:semisimplespherical}.
Moreover, the structural algebras are as indicated in the diagrams.
\end{lemma}

\begin{proof}  Let us first translate the problem and slightly change notation. Let $\gf$ be non-compact simple 
and $\hf\subset \gf$ be a semi-simple real spherical subalgebra
such that $\hf=\hf_1 \oplus \hf_2$ with $\hf_2$ non-compact simple.  Let  $\sf=\str(\gf,\hf)$
be the 
structural algebra and $\sf_2$ the projection of 
$\sf$ to $\hf_2$. The  
pairs which are called for in the lemma arise as 
$(\gf\oplus\hf_2, \hf)$, and according to 
Corollary \ref{tower-corollary} 
 the pair $(\gf\oplus\hf_2, \hf)$
is real spherical if and only if 
$(\hf_2 \oplus \sf_2, \sf_2)$ is real spherical. In particular, a necessary condition 
is that $(\hf_2, \sf_2)$ is real spherical.

\par We begin with the classical cases.  From Tables 
\ref{lcaph_class_symm} - \ref{lcaph_KKPS_class} we obtain the following Table \ref{Table X},
of real spherical pairs $(\gf, \hf)$ as just mentioned, for which 
$(\hf_2, \sf_2)$ is real spherical.

\newcounter{snumber}
\newcommand*{\slabel}[1]{{\rm (\refstepcounter{snumber}\thesnumber\label{#1})}}

\begin{table}[ht]
{\footnotesize
$$\begin{tabular}{l l l}
&$(\gf,\hf)$  & $(\hf_2,  \sf_2)$ \\ \hline
\slabel{s1}& $(\sl(n,\C), \sl(n-2,\C)+\sl(2,\C))$&$(\sl(2,\C), \C)$  for $n\geq 5$\\
\slabel{s2}& $(\sl(n,\R), \sl(n-2,\R)+\sl(2,\R))$&$ (\sl(2,\R), \R)$ for $n\geq 5$ \\
\slabel{s3}&$(\su(p,q), \su(p-1,q-1)+\su(1,1)) $&$ (\su(1,1), i\R)$  for $(p,q)\neq(2,2), (1,2)$\\
\slabel{s4}&$(\su(1,q), \ef+\su(1,q_2))$&$ (\su(1,q_2), \uf(q_2))$  for   $q_2\geq 1$  and $q-q_2\geq 2$\\
\slabel{s5} &$(\sl(n,\Hb), \sl(1,\Hb)+\sl(n-1,\Hb))$&$ (\sl(n-1,\Hb),\sl(n-2,\Hb))$\\
\slabel{s6}&$(\so(1,q), \ef+\so(1,q_2))$&$ (\so(1,q_2), \so(q_2)) $  for $q_2>1$ and $q-q_2\geq 3$\\
\slabel{s7}&$(\so^*(2n), \so^*(2n-4)+\su(2)+\sl(2,\R)) $&$(\sl(2,\R), \so(2))$\\
\slabel{s8}&$(\sp(n,\C), \sp(n-1,\C)+\sp(1,\C)) $&$(\sp(1,\C), \sp(1,\C))$\\
\slabel{s9} &$(\sp(n,\C), \sp(n-1,\C)+\sp(1,\C)) $&$(\sp(n-1,\C),\sp(n-2,\C)+\sp(1,\C))$\\
\slabel{s10}&$(\sp(n,\C), \sp(n-1,\C)+\sp(1,\R))$&$ (\sp(1,\R), \sp(1,\R))$\\
\slabel{s11} &$(\sp(n,\C), \sp(n-1,\C)+\sp(1,\R))$&$ (\sp(n-1,\C),\sp(n-2,\C)+\sp(1,\R))$\\
\slabel{s12} &$(\sp(n,\C), \sp(n-1,\C)+\sp(1))$&$ (\sp(n-1,\C),\sp(n-2,\C)+\sp(1))$\\
\slabel{s13}&$(\sp(n,\R), \sp(n-1,\R)+\sp(1,\R)) $&$(\sp(1,\R), \sp(1,\R))$\\
\slabel{s14} &$(\sp(n,\R), \sp(n-1,\R)+\sp(1,\R)) $&$(\sp(n-1,\R),\sp(n-2,\R)+\sp(1,\R))$\\
\slabel{s15}&$(\sp(n,\C), \sp(n-2,\C)+ \sp(2,\C))$&$ (\sp(2,\C),\sp(1,\C) + \sp(1,\C))$\\
\slabel{s16}&$(\sp(n,\R), \sp(n-2,\R)+ \sp(2,\R))$&$ (\sp(2,\R),\sp(1,\R) + \sp(1,\R))$\\
\slabel{s17}&$(\sp(p,q), \sp(p-1,q-1)+\sp(1,1)) $&$ (\sp(1,1), \sp(1) +\sp(1))$\\
\slabel{s18}&$(\sp(p,q),\sp(1)+ \sp(p-1,q))$&$ (\sp(p-1,q), \sp(p-1,q-1)+\sp(1))$\\
\slabel{s19}&$(\sp(1,q), \sp(q-q_2)+ \sp(1,q_2))$&$ (\sp(1, q_2), \sp(1) + \sp(q_2-1))$\\
\end{tabular}$$}
\medskip
\centerline{\rm\Tabelle{Table X}}
\end{table}

For the possibilities for $\ef$ in (\ref{s4}) and (\ref{s6}), see Table \ref{e-table}.

We claim that the following cases from above
give rise to a strongly spherical pair 
$(\hf_2, \sf_2)$, and hence $(\gf\oplus\hf_2,\hf)$ is real spherical. 
In each case we refer to the diagram in  Theorem \ref{thm:semisimplespherical}(2a)
which displays $(\gf\oplus\hf_2,\hf)$.

{\footnotesize
$$\begin{tabular}{lllll}
(\ref{s1})--(SS\ref{case4.7})&
(\ref{s4})--(SS\ref{case4.17})&
(\ref{s8})--(SS\ref{case4.10})&
(\ref{s15})--(SS\ref{case4.12})&
(\ref{s18})--(SS\ref{case4.20})
\\
(\ref{s2})--(SS\ref{case4.7})&
(\ref{s6})--(SS\ref{case4.16})&
(\ref{s10})--(SS\ref{case4.10})&
(\ref{s16})--(SS\ref{case4.12})&
(\ref{s19})--(SS\ref{case4.18})
\\
(\ref{s3})--(SS\ref{case4.9})&
(\ref{s7})--(SS\ref{case4.19})&
(\ref{s13})--(SS\ref{case4.10})&
(\ref{s17})--(SS\ref{case4.14})&
\end{tabular}$$}

For (\ref{s1}) and (\ref{s2}) this is well known. We also
note that all cases where $ \sf_2$ is compact 
are strongly spherical, since they reduce to 
$(\hf_2, \sf_2)$ by Lemma 
\ref{lemma stand2}(\ref{compact factor}).
These are (\ref{s3}), (\ref{s4}), (\ref{s6}), (\ref{s7}), (\ref{s17}) and (\ref{s19}). 
As for (\ref{s5}), we can exclude it as 
$(\hf_2, \sf_2)$ does not show up in
Lemma \ref{strong spherical}.
In (\ref{s8}), (\ref{s10}), (\ref{s13}) strong sphericality is obvious.
Both (\ref{s15}) and its real form (\ref{s16}) 
produce a spherical pair $(\gf\oplus\hf_2,\hf)$
since the complex form of this pair already belongs to the 
list of Brion-Mikityuk.
Further, (\ref{s18}) is an $\sp(1)$-extension of (S\ref{case4.4}),
hence strongly spherical.
In all these cases it is then easy to verify $\str(\gf\oplus\hf_2,\hf)$ 
in the corresponding diagram by means of (\ref{l descent}).

\par Among the classical cases we are left with 
(\ref{s9}), (\ref{s11}), (\ref{s12}), and  (\ref{s14}).
Note that (\ref{s9}) and (\ref{s14}) 
are special cases of (\ref{s8}), (\ref{s15}) and (\ref{s13}), (\ref{s16})
when $n=2,3$. Apart from these the
remaining cases can be excluded with the dimension bound 
  (\ref{dim bd}).

\par We turn to the exceptional cases. We recall from \cite{Brion}, \cite{Mik}
that here there are no complex strongly spherical pairs so that
we can disregard all cases where $\gf$ is quasisplit  (see Subsection \ref{subsub Lan} for the definition of quasisplit). 
Inspecting Tables \ref{lcaph_except_symm} - \ref{lcaph_KKPS_class}
we are left to check the following two cases for 
$(\gf,\hf,\hf_2)$:

\begin{enumerate}\setcounter{enumi}{19}
\item\label {e1}  $(\sE_6^3, \su(5,1) + \sl(2,\R), \sl(2,\R))$ with $\str(\gf,\hf)=\uf(2)+\uf(2)$. 
\item \label{e2}  $(\sE_7^3, \so(10,2)+ \sl(2,\R), \sl(2,\R))$ with $\str(\gf,\hf)=\so(6) \oplus \so(2) \oplus \so(2,1)$. 
\end{enumerate}

In these two cases we need to determine $ \sf_2$. Moreover, the particular case is spherical if and only if 
$ \sf_2\neq 0$. 
We will now show that this happens in both cases, and by that obtain the diagrams
(SS\ref{case4.22}) and (SS\ref{case4.21}).

\par We begin with (\ref{e1}) and $\gf=\sE_6^3$.  Note that the fact that 
$ \sf=\lf_\hf$ is compact implies 
that $\lf$ is elementary, hence the Levi part of a minimal parabolic subalgebra of $\gf$.
Observe that  
$\gf$ is Hermitian with maximal compact subalgebra 
$\kf=\so(10)\oplus \so(2)$. We recall that  
 $\zf(\kf)=\so(2)\subset \mf=\uf(4)\subset \lf$  
with $\mf$ defined in Subsection \ref{subsub Lan}.  
 Hence  $\zf(\kf)\subset  \sf$ as $ \sf=\uf(2) \oplus \uf(2)$
 contains the center of $\mf$.
  
We claim that $\zf(\kf)\not\subset \hf_1=\su(5,1)$ which implies $ \sf_2\neq 0$. 
Indeed, if $\zf(\kf)\subset\hf_1$, then $\zf(\kf)$ would act trivially on $\hf_2=\sl(2,\R)$, and 
since $\hf_2\cap\kf^\perp\neq \{0\}$ that would be a contradiction to the fact that 
$\zf(\kf)$ acts without non-trivial fixed points on $\kf^\perp$.
We have thus shown that $(\hf_2, \sf_2)=(\sl(2,\R), \so(2,\R))$ and thus 
$\str(\hf_2\oplus \sf_2, \sf_2)=\{0\}$. Moreover, 
$\sf\cap\hf_1=\su(2)\oplus\su(2)\oplus \uf(1)$
and thus  it follows from (\ref{l descent}) that
$\str(\gf\oplus\hf_2,\hf)$ 
is stated in (SS\ref{case4.22}).

\par We move on with the exceptional case $\gf=\sE_7^3$ with $\hf=\hf_1 \oplus \hf_2= \so(10,2) \oplus \sl(2,\R)$. 
Notice that the procedure mentioned in  Appendix \ref{appendix} 
to determine  the structural algebra
$\sf$ also gives $\lf$. In this particular case 
we obtain $\lf=\lf_1 \oplus \lf_2 \oplus \af_Z := \so(8) \oplus \so(2,1) \oplus \R^2$ 
with $\af_Z=\R^2$ the center of $\lf$. 
Further $ \sf=\so(6) \oplus \so(2,1) \oplus \so(2)$  with $\so(2,1)\subset \hf_1$ by  Part I, 
Remark 8.2 (a).
In order to show that $ \sf_2\neq 0$ it is sufficient to show that $\so(2) \not\subset\hf_1$. 
Assume the contrary and consider the $\hf$-module $V:=\gf_\C/\hf_\C$. 
Let $V_{32}$ be the $32$-dimensional spin representation of $\hf_{1,\C}$, then it is easy to see 
that $V=V_{32}\otimes \C^2$ 
as $\hf_\C$-module. Hence if $\so(2) \subset \hf_1$, we deduce that $V$, considered as 
$ \sf_{\C}$-module, decomposes with even multiplicities. 
On the other hand $V\simeq (\lf/ \sf\oplus \uf)_\C$ as $ \sf_{\C}$-module.  

\par In the next step we investigate the $ \sf$-module $\uf$.  We recall from \cite{OS}  
that the root system $\Sigma(\gf, \af_Z)$ 
is of type $BC_2$, 
$$\Sigma(\gf,\af_Z)=
\{ \pm \e_1, \pm 2\e_1, \pm \e_2, \pm 2\e_2, \pm \e_1 \pm\e_2\}.$$ 
We decompose $\uf=\uf_0 \oplus\uf_1$ where $\uf_0$ corresponds to the short roots and $\uf_1$ to 
the non-short roots.  All root spaces are $\lf$-modules.  

\par We first claim that $\uf_1$ decomposes with even multiplicities as an $\sf$-module.  
According to \cite{OS},  the long root spaces are 
real one-dimensional, hence trivial as 
$[\lf,\lf]$-modules and thus also as $\sf$-modules. 
We now assert that root spaces corresponding to roots  from $\{\pm \e_1\pm\e_2\}$ are  isomorphic as 
$\sf$-modules. 
Indeed, we have just seen 
that $\gf^{\pm 2\e_i}\simeq \R$ are trivial $\sf$-modules.  Now take for example the root space 
$\gf^{\e_1-\e_2}$
and $0\neq X_i \in\gf^{2\e_i}$. Then, finite dimensional $\sl(2)$-representation theory yields that 
bracketing with $X_2$ results in an $\sf$-equivariant isomorphism
$\gf^{\e_1 -\e_2} \to \gf^{\e_1+\e_2}$.  Iterating 
then implies  that all $\gf^\alpha$ with $\alpha\in \{\pm \e_1 \pm\e_2\}$ are isomorphic as 
$\sf$-modules (even as $[\lf,\lf]$-modules).   Our claim follows.

\par We are left with $\uf_0$ which consists of short roots $\gf^{\pm\e_i}$.   First note that $\gf^{\e_i}$ is 
isomorphic to $\gf^{-\e_i}$  as $\lf_\hf$-module via the bracketing argument from above. 
According to \cite{OS} the $\gf^{\e_i}$ are real 8-dimensional.  Let us fix one, say $W:=\gf^{\e_1}$. 
Since $W_\C$ is a  prehomogeneous $\lf_\C$-module by Proposition \ref{L-geometry} it 
follows that there are three possibilities for $W$  (triality): 
$$ \R^8 \qquad \spin(8)_+ \qquad \spin(8)_-$$ 
where $\R^8$ indicates the standard representation. Observe that 
$W$ is uniquely determined by its branching to $\so(6)\simeq \su(4)$ which is respectively 
\begin{equation}\label{branching} \R^6 \oplus \R^2  \qquad \C^4 \qquad  (\C^4)^*\, .\end{equation}
To sum up: if $\gf^{\e_1}$ and $\gf^{\e_2}$ are isomorphic, then $\uf_0$ decomposes with even multiplicities 
as an $\lf_\hf$-module. Otherwise, the $\lf_\hf$-module  $\uf_0$ is a sum of two which are listed in
(\ref{branching}).

\par Recall that $V$ was assumed to decompose with even multiplicities. We have just shown that 
$\uf_1$ decomposes with even multiplicities and determined the possible branchings for $\uf_0$. 
Now, $\so(8,\C)/ \so(6,\C) \oplus \so(2,\C) =\C^6\oplus (\C^6)^*$ 
and we observe that $\C^6$ and $(\C^6)^*$
are inequivalent as $\so(6,\C) \otimes\so(2,\C)$-modules.  Together with our branching results for 
$\uf_0$ and $\uf_1$ we thus obtain a contradiction. 

\par It follows that $(\hf_2,\sf_2)=(\sl(2,\R), \so(2,\R))$ and hence $\str(\hf_2\oplus\sf_2,\sf_2)=\{0\}$. 
In addition $\sf\cap\hf_1=\so(6) \oplus\so(2,1)$ by Remark 8.2 in Part I, 
and we obtain $\str(\gf\oplus\hf_2,\hf)$ via (\ref{l descent}). 
\end{proof}

{\bf Proof of Theorem~\ref{thm:semisimplespherical} (2b):}  Let us now consider the case of a semi-simple 
and not simple real spherical subalgebra $\hf\subset \gf=\gf_1\oplus\gf_2$ with $\hf_i\neq\gf_i$ for $i=1,2$.
We use the notation of (\ref{decomp_h}).
Since $\hf^0\neq 0$ both $\hf_1$ and $\hf_2$ cannot be simple. Hence
it is no loss of generality to assume that $\hf_1=\hf_1' \oplus \hf^0$
is semi-simple and not simple. 
First observe that $(\gf_1 \oplus \hf_2, \hf)$ is real spherical by Lemma \ref{lemma stand2} (\ref{reduction condition}).
Then, by Lemma \ref{lemma stand1} (\ref{standard3}) we deduce that $(\gf_1 \oplus\hf^0, \hf_1)$ is real 
spherical. Let $\hf^{0,+}\triangleleft \hf^0$ be a simple non-compact factor. 
It follows that $(\gf_1 \oplus\hf^{0,+},\hf_1)$ must show up in the list of 
Theorem ~\ref{thm:semisimplespherical} (2a). From  this list we now deduce  
that $\hf_1$ has only two simple factors, except for 
$\gf_1=\so^*(2n+4)$ or $(\gf_1,\hf_1)=(\so(1,n),\so(4)+ \so(1,n-4))$.

\par Arguing similarly we also obtain that $(\hf^{0,+}\oplus\gf_2, \hf_2)$ is real spherical and thus must show 
up in the lists of Theorem ~\ref{thm:semisimplespherical} (1b), (2a).

\par Recall  from Section \ref{dec L_i}
that $(\gf,\hf)$ is real spherical if and only if $H^0$ is spherical as an  $S_1'' \times S_2''$-variety,
and  (see (\ref{iso L_i})) that 
then $\hf^0=\sf_1''+ \sf_2''$,
as a necessary condition for real sphericality. 

\par Let us assume first that $\hf^0=\hf^{0,+}$.  Inspecting the list in Theorem \ref{thm:semisimplespherical} (2a)
we see that $\hf^0$ is either symplectic or of rank one. 
We begin with the symplectic  case and recall that symplectic Lie algebras do not admit 
non-trivial factorizations.  Hence if $(\gf,\hf)$ is real spherical, then this forces that 
 $\sf_1''$ or $\sf_2''$ equals
$\hf^0$.  If $ \sf_1''=\hf^0$, then this means that we are in the situations (\ref{s8}), (\ref{s10}) or (\ref{s13}) in 
Table \ref{Table X}.  We can draw the same conclusion if 
$ \sf_2''=\hf^0$.  Thus in case $\hf^0=\hf^{0,+}$
is symplectic we may assume that $(\gf_1\oplus \hf^0, \hf_1)$ is one of the pairs 
$(\gf\oplus\hf_2,\hf)$ in (\ref{s8}), (\ref{s10}) 
or (\ref{s13}). In particular $\hf^0= \sf_1''=\sp(1,\K)$. 
Being in this situation $(\gf,\hf)$ then is real spherical 
if and only if $ \sf_2''$ contains $\gl(1,\K)$, $\so(2,\K)$ or $\sp(1,\R)$.

This situation 
only occurs for $(\hf^0\oplus\gf_2,\hf_2)$ being from the list 
of Theorem \ref{thm:semisimplespherical} (2a), to be precise
for 
$$(\ref{s1}), (\ref{s2}), (\ref{s3}), (\ref{s4})  \ \text{for $q_2=1$}, (\ref{s6})\  \text{for $q_2=2,3$}, 
(\ref{s7}), (\ref{s8}), (\ref{s10}), 
(\ref{s11})\ \text{for $n=2$},
(\ref{s13}), (\ref{e1}), (\ref{e2})$$
in Table \ref{Table X}.  In all these cases the structural algebras $\sf$  can be read off via 
(\ref{formula L}). In particular,  
 we have $\sf = \sf_1' \oplus \sf_2'$   except 
when $(\hf^0\oplus\gf_2,\hf_2)$ is also of type  (\ref{s8}), (\ref{s10}) or (\ref{s13}). In the 
latter case we have 
 $\sf = \sf_1' \oplus \gl(1,\K)\oplus \sf_2'$, 
where $\K=\R,\C$ is determined by 
 $\hf^0=\sp(1,\K)$, and where $\gl(1,\K)$ is embedded tridiagonally
 into $\hf=\hf_1' \oplus \hf^0 \oplus\hf_2'$.

We move on where $\hf^0=\hf^{0,+}$ is of rank one and not symplectic,
and hence equals $\so(1,p)$ or $\su(1,p)$.
For that we first  recall the following fact from  Part I, Lemma 2.4:  
If $ \rf=\rf_1 +\rf_2$ is a non-trivial factorization of 
a  simple algebra  $\rf$ with one factor compact, then 
 $\rf$ is compact. 

We begin with the case where $\hf^0=\so(1,p)$ for $p\geq 4$. Then 
 $\sf_1''$ is compact and hence we must have 
$ \sf_2''=\hf^0$
which is not possible.  Likewise we can exclude the case where $\hf^0=\su(1,p)$ for $p>1$. 

\par To summarize, the cases where $\hf^0=\hf^{0,+}$ give rise to (SS\ref{case4.23}) - (SS\ref{case4.38}). 
It remains to show that $\hf^0=\hf^{0,+}$.  
Suppose the contrary and let $\hf^0=\hf^{0,+}\oplus\hf^{0,-}$ with $\hf^{0,-}\neq \{0\}$.  Let $\hf_i^{-}$ be the projection 
of $\hf^{0,-}$ to $\hf_i$. We consider 
$\tilde \hf= \hf_1^- \oplus \hf$ and observe that $\tilde\hf^0=\hf^{0,+}$. Hence the pair $(\gf,\tilde\hf)$
must be one of the cases   (SS\ref{case4.23}) - 
(SS\ref{case4.38}).  
Inspecting this list we quickly see that this is impossible. Indeed, let us begin with the case where 
$\tilde\hf$ has three factors. Then $\tilde \hf_1'\simeq \tilde \hf_2'$ 
since $\tilde\hf \simeq \hf_1^- \oplus \hf^{0,+}\oplus \hf_1^{-}$.
Hence $(\gf,\tilde \hf)$ needs to be  (SS\ref{case4.27}) for $n=m$ and $\K=\K'$. 
A simple dimension count then shows that 
(\ref{dim bd}) is violated and thus 
$(\gf,\hf)$ is not real spherical. Likewise
the case where $\tilde \hf$ has four factors is ruled out right away. 
\qed

\subsection{The case where $\hf$ is not semi-simple}

\
\smallskip

{\bf Proof of Theorem~\ref{thm:semisimplespherical} (3a):}
We only have to determine the cases in Theorem~\ref{thm:semisimplespherical}(1),(2)
where one can enlarge $\hf_i$ inside $\gf_i$ for $i=1$ or $2$ to have a non-trivial center. 
Going through the cases one by one we see that this is possible for
(S2), (S5), (S6), (S7), (SS9), (SS10), (SS14), (SS15), (SS21), (SS22), (SS25), (SS26), and (SS27).
This then results in the asserted list.
\qed

For the proof of (3b) of the theorem
we consider as  earlier
first the case where $\gf_2=\hf_2$. 
We assume that $\hf$ is reductive, but not semi-simple.

\begin{lemma}\label{lemma non-ss}  Suppose that $\gf$ is non-compact simple and 
$\hf=\hf'\oplus\hf''$ is a reductive subalgebra
of $\gf$ with $\hf'$ not semi-simple and $\hf''$ non-compact simple. 
Assume $(\gf,\hf)$ is real spherical, and
let $\sf''$ be the projection of $\str(\gf,\hf)$ to $\hf''$.
Then $(\gf\oplus \hf'', \hf)$ is 
real spherical with $(\gf\oplus\hf'', [\hf,\hf])$ not real spherical if and only if 
$(\gf,\hf,\hf'',\sf'')$ is one of the following:

\begin{enumerate}
\item\label{r0}  $(\sl(n+2,\K), \sf(\gl(n,\K)\oplus\gl(2,\K)), \sl(2,\K), \gl(1,\K))$ for $\K=\R,\C$ and $n=1,2$, 
\item\label{r00} $ (\su(2,2), \sf(\uf(1,1)\oplus\uf(1,1)), \su(1,1), \uf(1))$,
\item \label{r3} $(\su(p+1,q+1), \uf(p,q+1), \su(p,q+1), \uf(p,q))$ for $p=q,q+1$, 
\item \label{r1} $(\sl(n,\K), \gl(n-1,\K), \sl(n-1,\K), \gl(n-2,\K))$ for $\K=\R,\C$ and $n\geq 3$,
\item \label{r2} $(\sl(n+1,\Hb), \gl(n,\Hb)+\ff, \sl(n,\Hb), \gl(n-1,\Hb)+\ff)$  for $\ff\subset \sl(1,\Hb) $, 
\item\label{r8} $(\so(1,n),\so(2)+\so(1,n-2),\so(1,n-2),\so(n-2))$,
\item \label{r6} $(\sE_6^4, \gl(1,\R)+\so(9,1), \so(9,1), \so(1,1)+\spin(7))  $. 
\end{enumerate}
Moreover, the structural algebras $\str(\gf\oplus\hf'',\hf)$ are given by 
$$\begin{tabular}{c l}
$(\gf,\hf)$  & $\str(\gf\oplus\hf'',\hf)$ \\ \hline
{\rm(\ref{r0})} & $0$\\
{\rm(\ref{r00})}& $0$\\ 
{\rm(\ref{r3})} &  $0$\\
{\rm(\ref{r1})} &  $0$\\
{\rm(\ref{r2})} & $ \ff$\\ 
{\rm(\ref{r8})} &{$\so(n-3)$}\\
{\rm(\ref{r6})} & $\sG_2$
\end{tabular}$$
\end{lemma}

\begin{proof} First note that $(\gf\oplus \hf'', \hf)$  is real spherical  if and only 
if $(\hf'', \lf'')$ is strongly real spherical by Corollary \ref{tower-corollary}.  Under the 
additional assumption 
that $(\gf\oplus\hf'', [\hf,\hf])$ is not real spherical the asserted list is extracted 
from Tables 
\ref{lcaph_class_symm} - \ref{lcaph_KKPS_class}.  

The table for $\str(\gf \oplus \hf'', \hf)$
follows inductively via (\ref{l descent}).  This is straightforward for (\ref{r0}) - {(\ref{r8})}. 
For (\ref{r6}) we obtain with Corollary \ref{tower-corollary} applied iteratively: 
{\belowdisplayskip=-12pt  
\begin{eqnarray*} \str(\gf\oplus\hf'',\hf)&\simeq&  \str(\so(9,1) \oplus[\so(1,1)\oplus \spin(7)],\so(1,1)\oplus\spin(7))\\
&\simeq& \str([\so(1,1)\oplus \spin(7)] \oplus \sG_2,\sG_2)\\ 
&\simeq&  \sG_2\, .\end{eqnarray*}
}
\end{proof}

{\bf Proof of Theorem~\ref{thm:semisimplespherical}(3b):}   
Let us assume first that $(\gf_1\oplus\gf_2,\hf)$ is of the type where $\gf_2=\hf_2$. Then we can use Lemma 
\ref{lemma non-ss}, from which we obtain 
all diagrams of (3b) except 
(R\ref{case4.58}) and (R\ref{case4.60}).
\par Having classified all cases with $\hf$ reductive and $\gf_2=\hf_2$ we can now complete the 
proof.  Suppose that  $(\gf_1\oplus\gf_2,\hf)$ is  real spherical 
with $\hf$ reductive and not semi-simple,  that $(\gf_1\oplus\gf_2,[\hf,\hf])$ is
not real spherical, and that $\hf_i\neq\gf_i$ for $i=1,2$. 
It is no loss of generality to assume that $\hf_1$ is 
not semi-simple. 
We let $\hf^0\triangleleft\hf$ be a simple non-compact ideal 
with non-zero projections to $\hf_1$ and $\hf_2$. 
Then 
$\hf_i\simeq \tilde\hf_i:=\hf_i'\oplus\hf^0$ for $i=1,2$. 
The fact that $(\gf,\hf)$ is real spherical implies 
that 
$(\gf_1\oplus \hf^0, \tilde \hf_1)$ and $(\hf^0 \oplus \gf_2, \tilde\hf_2)$ are both real spherical. 
Note that $(\gf_1\oplus \hf^0, \tilde \hf_1)$ is part of what was already  just classified in Theorem \ref{thm:semisimplespherical}(3a,b). Likewise all cases for $(\hf^0\oplus\gf_2,\tilde \hf_2)$ with $\tilde \hf_2$ possibly semi-simple were already classified.
Moreover if we let $\lf_i^0$ be the projection of $\str(\gf_i,\hf_i)$ to $\hf^0\subset \hf_i$, then 
$\hf^0$ has to be real spherical for the action of $\lf_1^0 \oplus \lf_2^0$. 
\par Using the already obtained tables it is now an elementary procedure to determine all the 
remaining cases. 
As an example we do the computation for the 
case where $(\gf_1\oplus \hf^0, \tilde \hf_1)$ is  
(R\ref{case4.56}).
In this case $(\gf_1\oplus \hf^0, \tilde \hf_1)= (\sl(n+2,\K) \oplus \sl(2,\K), \gl(2,\K))$ with $n=1,2$ and $\lf_1^0=\gl(1,\K)$.
A quick inspection of the already 
obtained lists implies that $(\gf_2\oplus \hf^0, \tilde\hf_2)$ is of one of the forms:
\begin{itemize}
\item $(\gf_1\oplus \hf^0, \tilde \hf_1)$, 
\item $(\hf^0 \oplus \hf^0, \hf^0)$, 
\item (SS\ref{case4.10}) with the reversed roles of $\K$ and  $\K'$,
and $\lf_2^0=\sl(2,\K)$.
\end{itemize}
The first two of these give nothing new, and the  
last case gives (R\ref{case4.58}).

 The remaining cases (R\ref{case4.48})-(R\ref{case4.53a}), (R\ref{case4.57})-(R\ref{case4.59}) and (R\ref{case4.47})  are treated similarly. The two cases
where $(\gf_1\oplus \hf^0, \tilde \hf_1)$ is
(R\ref{case4.57}) and (R\ref{case4.42})
contribute the values $p=1$ and $p=0$
respectively, of (R\ref{case4.60}).
The other cases do not contribute anything new.
\qed

\smallskip
This completes the proof of Theorem \ref{thm:semisimplespherical}.

\section{More than two factors}

 Having classified all cases with two factors and their corresponding structural algebras, it is now a manageable task to complete 
the classification. The key is Proposition \ref{four exclusion} which limits the investigation to the case $k=3$. 
We start our examination with the real spherical triple spaces.

\par Consider $G_1\times \dots\times G_k$. 
For $i,j,\ldots\in\{1,2,\dots k\}$ we write $H_{ij\dots}$ 
for the projection
of $H$ to $G_{ij\dots}=G_i\times G_j\times\cdots$ and 
we set $Z_{ij\dots}= G_{ij\dots}/ H_{ij\dots}$.

\newcommand*{\pnlabel}[1]{{\rm (BM\refstepcounter{pairnumber}\thepairnumber\label{#1})}}

\begin{lemma} \label{triple}Let $\gf$ be a simple non-compact Lie algebra. Then $(\gf \oplus \gf \oplus\gf , \diag(\gf))$
is real spherical if and only if $\gf=\so(1,n)$ for $n\geq 2$.  The
corresponding diagram reads
$$ \triple{\so(1,n)}{\so(n-2)}{n\geq 2}{case5.0}$$
\end{lemma}

\begin{proof}   Recall the Langlands decomposition $\pf =\mf +\af+\nf$ of the 
be a minimal parabolic of $\gf$ (see Subsection \ref{subsub Lan}),  and note that 
$Z_{12}$ is real spherical with structural algebra 
$\mf+\af$ according to diagram (S\ref{case4.1}).
We consider the tower $H'= H_{12}\times H_3=
\diag_2(G) \times G \supset H =\diag_3(G)$. Then according to 
Proposition \ref{tower-factor}, $Z$ is real spherical if and only if 
$MA\times P$ has an open orbit on $G$
and in this case the structural algebra is 
$\str(\gf\oplus\mf\oplus\af,\mf\oplus\af)$. This means that $MA$ has an open orbit on $N$. 
Since $MA$ preserves restricted root spaces, this implies that $\dim A$ equals the number of positive roots for the restricted 
root system of $\gf$ with respect to $\af$. It follows that  
the restricted root system is of type $A_1$, i.e. $\gf \simeq 
\so(1,n)$ for some $n\geq 2$.  Conversely, in these cases $MA$ has an open orbit on $N$.
\end{proof}

Suppose that $(\gf_1, \hf_1)$ and $(\gf_2, \hf_2)$ are reductive pairs such that 
$\hf_1=\hf_1'\oplus\hf^0$ and $\hf_2=\hf_2'\oplus\hf^0$ 
share a common ideal $\hf^0$ (up to isomorphism). 
Then we refer to the reductive pair 
\begin{equation}\label{defi glue}
 (\gf,\hf)=(\gf_1\oplus\gf_2,\hf_1'\oplus\hf_2'\oplus \diag(\hf^0)).
\end{equation}
as the {\it glueing} of $(\gf_1, \hf_1)$ and $(\gf_2,\hf_2)$ {\it along} $\hf^0$.

\begin{theorem}\label{Thm k=3}
The strictly indecomposable real spherical spaces with $k=3$ are the triple cases 
from Lemma \ref{triple} and the following: 
$$
\begin{tabular}{c} 
\hline
$\repNM{\sp(n,\K)}{\sp(2,\K'')}{\sp(m,\K')\qquad}
{\begin{matrix}\sp(n-1,\K)\\ \quad| \\ \sp(n-2,\K)\end{matrix}}
{\begin{matrix}\sp(1,\K'')\\ | \\ 0\end{matrix}}
{\begin{matrix}\sp(1,\K'')\\ | \\ 0\end{matrix}}
{\begin{matrix}\sp(m-1,\K')\\ |\quad \\ \sp(m-2,\K')\end{matrix}}
{m,n\geq 1}{\K''\subset\K, \K'}{case5.1}$\\ \hline
$\repMM{\sp(n,\K)}{\sp(m,\K')}{\sp(k,\K'')}
{\begin{matrix}\sp(n-1,\K)\\ \qquad| \\ \sp(n-2,\K)\end{matrix}}
{\begin{matrix}\sp(m-1,\K')\\ | \\ \sp(m-2,\K')\end{matrix}}
{\begin{matrix}\sp(1,\K''')\\ | \\ 0\end{matrix}}
{\begin{matrix}\sp(k-1,\K'')\\ |\qquad \\ \sp(k-2,\K'')\end{matrix}}
{m,n,k\geq1}{\K'''\subset\K,\K',\K'' } {case5.2}$ \\ \hline
\end{tabular}
$$

Here $\K,\K',\K'',\K'''=\R,\C$ subject to the following additional conditions:
\begin{enumerate}
\item In {\rm (BM\ref{case5.1})} one has $\K'=\K''$ if $m=1$, 
{resp. $\K=\K''$ if $n=1$.}
\item In {\rm (BM\ref{case5.2})} one has $\K'''=\K$ if $n=1$, $\K'''=\K'$ if $m=1$ and $\K'''=\K''$ if $k=1$. 
\end{enumerate}
and the absolutely spherical cases 
are those with $\K=\K'=\K''=\K'''$.
\end{theorem}

\begin{proof}  
Let $(\gf,\hf)$ be a 
strictly indecomposable spherical pair with $k=3$.
Without loss of generality we may assume that $Z_{12}$ and 
$Z_{23}$ are strictly indecomposable. Hence both show up in the list of Theorem  
\ref{thm:semisimplespherical}.  

We begin with the case where $\gf_3=\hf_3$,
and write $\hf_{12}=\hf_{12}'\oplus \hf^0$ with
$\hf^0\simeq \hf_3$. Let 
$\sf_{12}^0$ be the projection of $\str(\gf_{12},\hf_{12})$ to 
$\hf^0$. Then a  necessary and sufficient 
condition for $Z$ to be real spherical is that
$(\hf^0,\sf^0_{12})$ is strongly real spherical.
By inspecting the lists in Theorem \ref{thm:semisimplespherical} we see that this can only happen if either 
$\gf_1=\gf_2=\hf^0$ (case (S\ref{case4.1})) or 
$\hf^0=\sp(1,\K)$ with $\sf_{12}^0=\gl(1,\K)$
(cases (SS\ref{case4.10}), (SS\ref{case4.27})).
The first case was treated in Lemma \ref{triple},
and leads to $\gf_1=\gf_2=\gf_3=\so(n,1)$, $n\ge 2$. 
The other cases can be combined 
in the following diagram
\begin{equation}\label{BigM}
\repMiviw{\sp(n,\K)}{\sp(k,\K')}
{\sp(n-1,\K)}{\sp(1,\K'')}{\sp(k-1,\K')}
{\sp(n-2,\K)}{\gl(1,\K'')}{\sp(k-2,\K')}
\end{equation}
with $n,k\geq 1$ and $\K,\K',\K''=\R,\C$ with 
$\K''\subset\K, \K'$. If $n=1$, resp. $k=1$, we require
in addition that $\K''=\K$, resp. $\K''=\K'$. 
Note that with $n=k=1$ also to the first two cases of
Lemma \ref{triple} are included.

There are two possible choices of $\hf^0$ in $\hf_{12}$. 
It can be the middle factor $\sp(1,\K'')$, or it can be one 
of the others when $n$ or $k$ is $2$, 
say $k=2$ and $\hf^0=\sp(1,\K')$. The first choice
leads to the cases $m=1$ of (BM\ref{case5.2}).
For the second choice we need $(\sp(1,\K'),\gl(1,\K''))$
to be strongly spherical, and hence $\K'=\K''$.
This then leads to the cases $m=1$ of (BM\ref{case5.1}).

Let us now consider the case where $\hf_3\subsetneq \gf_3$. Since $Z$
is strictly indecomposable, there exists a non-compact
simple common factor $\hf^{0,+}$ of $\hf_{12}$ and $\hf_3$.
According to 
Lemma \ref{lemma stand2}(2) and Lemma \ref{lemma stand1},  
the pair $(\gf_{12}\oplus \hf^{0,+}, \hf_{12})$
is then real spherical, hence one of the pairs
we just determined. In particular $Z_{12}$ is either
a group case of $\so(n,1)$, or it is of type (\ref{BigM}). 
Moreover $\hf^{0,+}$ is either
$\so(n,1)$ for some $n$,  or it is $\sp(1,\R)$, $\sp(1,\C)$.

By switching the roles of $Z_{12}$ and $Z_{23}$ 
we see that the same limitations  apply to $Z_{23}$. In particular,
if $Z_{12}$ is a group case with $\gf_1=\gf_2=\so(n,1)$ for $n\ge 4$, 
then $Z_{23}$ has to be of the same type with
$\gf_3=\gf_2=\so(n,1)$, and we are in the case of Lemma
\ref{triple}. Excluding that case we infer
that $Z_{12}$ is of type (\ref{BigM}) and that for some $m>1$
\begin{equation}\label{g3h3}
(\gf_3,\hf_3)= (\sp(m,\K'''), \sp(1,\K'''') \oplus\sp(m-1,\K'''))
\end{equation}
with $\K''''\subset \K'''$.
Moreover $\str(\gf_3,\hf_3)=\sp(1,\K'''')\oplus\sp(m-2,\K''')$.

If $\hf^0=\hf^{0,+}$ is simple the glueing between $Z_3$ and $Z_{12}$ 
takes place along the factor
$\sp(1,\K'''')$ of $\hf_3$, and as before there are two choices
in $\hf_{12}$. These two choices then lead to the remaining cases
of (BM\ref{case5.2}) and (BM\ref{case5.1}).
Finally, the possibility for $n=k=m=2$ that $Z_3$ is glued to
$Z_{12}$ along both factors of $\hf_3$ is easily excluded by
dimension count.
\end{proof}

\begin{prop}\label{four exclusion}  {\rm (Exclusion of four and more factors)}  Let $Z$ be a strictly indecomposable real spherical space
attached to $G=G_1\times \ldots\times G_k$. Then $k\le 3$.
\end{prop}

\begin{proof}
It is sufficient to consider the case of $k=4$.  We argue by contradiction and assume that there exists a strictly indecomposable
real spherical space $Z$ of this length.

We observe that after reordering we may assume that $Z_{123}$ and 
$Z_{234}$
are  strictly indecomposable (this follows for example
from the simple fact that every connected graph with 4 vertices contains
a path of length 3). 

It follows in particular that both $Z_{123}$ and $Z_{234}$ has to be one of the spaces
listed in Theorem \ref{Thm k=3}.  We claim that it cannot be the triple space. Indeed, if $Z_{123}$ were the triple space with $\gf_i=\so(1,n)$ for $i=1,2,3$, then 
$Z_{123}$ and $Z_4$ would be glued together along 
$\hf^0\simeq \so(1,n)$ and $\hf^0$ would be spherical for the 
action of $\hf^0 \oplus [\lf\cap\hf]_{123}$.  Now $[\lf\cap\hf]_{123}=\so(n-2)$ is compact and that would mean 
that $\so(n-2)$ is a spherical subalgebra of $\hf^0=\so(1,n)$ which is not the case by the classification 
of  Part I (note that the spherical 
subalgebra $\su(4)$ of $\so(1,8)$ is isomorphic to $\so(6)$, but the subalgebras are not 
conjugate within $\so(1,8)$).

\par Hence $Z_{123}$ has to be one of the cases (BM\ref{case5.1}) or (BM\ref{case5.2}). By symmetry this 
holds for $Z_{234}$ as well. In particular we have 
$(\gf_4,\hf_4)=(\sp(s,\tilde \K), \sp(1,\tilde \K') \oplus\sp(s-1,\tilde \K))$
with $\tilde \K,\tilde\K'=\R,\C$, $\tilde\K'\subset \tilde\K$ and $s\geq 1$. 
Then $Z_{123}$ and $Z_4$  are  glued together along $\hf^0=\sp(1,\tilde \K')$ or $\hf^0=\sp(1,\K)$ if $s=2$. 
In order for the resulting space $Z=Z_{1234}$ to be real spherical this would require that the projection 
of $[\lf\cap\hf]_{123}$ to $\hf^0$ is real spherical. 
But this projection is $0$ by Theorem \ref{Thm k=3} and  we obtain a contradiction.
\end{proof}

\section{Indecomposability versus strict indecomposability}\label{Section not strictly}

Let $\gf$ be a real semisimple Lie algebra and $\hf$ an algebraic reductive subalgebra.  
The goal of this section is to describe how one can determine 
whether an indecomposable, but not strictly indecomposable, pair $(\gf, \hf)$ is real spherical.

\begin{lemma} Let $(\gf, \hf)$ be an indecomposable  real spherical pair which 
not strictly indecomposable. Then there is a splitting $\gf=\gf_1 \oplus\ldots\oplus\gf_k$
into ideals such that with $\hf_i$ the projection of $\hf$ to $\gf_i$ one has: 
\begin{enumerate}
\item $\hf_{\rm n}=(\hf_1)_{\rm n} \oplus \ldots \oplus (\hf_k)_{\rm n}$.  
\item $(\gf_i, \hf_i)$
is a strictly indecomposable real spherical pair for all $1\leq i\leq k$.
\end{enumerate} 
 In particular, if $\gf$ is semisimple without compact factors, then each $(\gf_i, \hf_i)$  appears either in Theorem 
\ref{thm:semisimplespherical} or Theorem \ref{Thm k=3}. 
\end{lemma} 

\begin{proof} By definition, if $(\gf,\hf)$ is not strictly indecomposbale, there exists 
a splitting of $\gf$ and $\hf_{\rm n}$ as indicated.  Moreover, as $G/H$ projects onto 
$G/H_i$, it follows that each $(\gf_i,\hf_i)$ is real spherical (see also Lemma \ref{lemma stand1}). 
We assume that $k$ is maximal and then each $(\gf_i, \hf_i)$ is strictly indecomposable. 
\end{proof}

\par For our objective to describe all indecomposable real spherical pairs $(\gf,\hf)$ which are not strictly 
indecomposable, we may thus assume that $\gf=\gf_1\oplus\ldots\oplus\gf_k$ 
with $\hf_{\rm n} = (\hf_1)_{\rm n}\oplus \ldots \oplus (\hf_k)_{\rm n}$ such that each 
$(\gf_i,\hf_i)$ be a strictly indecomposable real spherical pair. 

Let $\hf':=\hf_1 \oplus\ldots\oplus \hf_k$ and note that $\sf(\gf,\hf')= \sf(\gf_1,\hf_1)\oplus
\ldots \oplus \sf(\gf_k,\hf_k)$.  We denote by $\cf_i$ the projection 
of $\sf(\gf_i,\hf_i)$ to $(\hf_i)_{\rm el}$.  Since $(\hf_{\rm n})'=\hf_{\rm n}$ we thus obtain from 
Proposition \ref {tower-factor} applied to the tower $\gf\supset \hf'\supset \hf$ the following criterion: 

\begin{prop}\label{referees suggestion}
{\rm (Criterion for sphericality)} In the setup described above, $(\gf,\hf)$ is real spherical if and only if 
$$(\hf_1)_{\rm el}\oplus\ldots\oplus (\hf_k)_{\rm el}= \hf_{\rm el} + \cf_1 +\ldots + \cf_k$$
\end{prop}

In the sequel we will use this criterion to analyse the situation further. 
We begin with the building blocks of $k=2$ and each $\gf_i$ non-compact simple.

\begin{prop}\label{IND} {\rm(Indecomposable pairs with two factors)}  \label{compact exceptions} Suppose that 
$\gf_1, \gf_2$ are both non-compact simple and set $\gf=\gf_1\oplus \gf_2$. 
Let $\hf\subset \gf$ be a reductive subalgebra such that $(\gf, \hf)$ is an indecomposable real 
spherical pair which is not strictly indecomposable.  Set $\hf^0:= \hf/  [(\gf_1\cap \hf)\oplus (\gf_2\cap \hf)]$.

Then either $\hf^0\simeq\R^l \oplus (i\R)^m \oplus \sp(1)^n$ for some $l,m,n\le 1$, or
$\hf^0$ is simple and $(\gf,\hf)$ is equivalent to one of the following pairs:

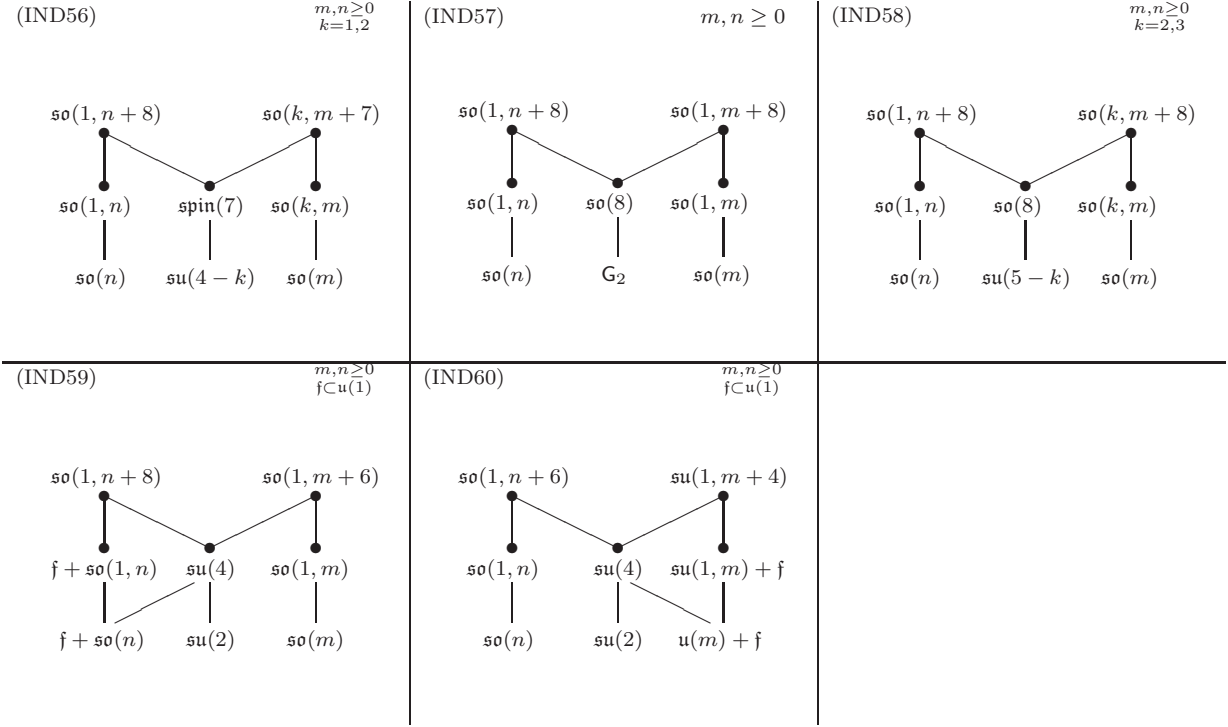
\renewcommand*{\pnlabel}[1]{{\rm (IND\refstepcounter{pairnumber}\thepairnumber\label{#1})}}

$$\begin{tabular}{c|c|c} \repMiii{\so(1,n+8)} {\so(k,m+7)}{\ \so(1, n)}{\ \spin(7)} 
{\ \so(k,m)} {\  \  \ \so(n)}{\!\su(4-k)}{\ \ \ \so(m)}{m,n\geq 0\atop k=1,2}{c1}
& \repMiii{\so(1,n+8)} {\so(1,m+8)}{\ \so(1, n)}{\ \so(8)} {\so(1,m)} {\quad\so(n)}
{\quad \sG_2}{\quad\so(m)}
{m,n\geq 0}{c5}
& \repMiii{\so(1,n+8)} {\so(k,m+8)}{\ \so(1, n)}{\ \so(8)} {\so(k,m)} {\quad\so(n)}
{\!\su(5-k)}{\quad\so(m)}
{m,n\geq 0\atop k=2,3}{c4}
\\ \hline
   \repMiiiLD{\so(1,n+8)} {\so(1,m+6)}{\ff + \so(1, n)\ \ }{\ \ \su(4)} {\ \so(1,m)} {\ \ff+ \so(n)}
   {\ \ \su(2)}{\ \ \ \so(m)}{m,n\geq 0 \atop  \ff\subset \uf(1)}{c2}
&  \repMiiiRD{\so(1,n+6)} {\su(1,m+4)}{\ \so(1, n)\ \ }{\ \ \su(4)} {\su(1,m)+\ff } {\ \ \ \so(n)}
{\ \ \su(2)}{\ \uf(m)+\ff }{m,n\geq 0\atop 
\ff \subset \uf(1)}{c3}
\end{tabular}
$$

In diagram {\rm (IND\ref{c1})} the left diagonal embedding of $\spin(7) \simeq \so(7)$ is the 
spin embedding of $\spin(7)$ into 
$\so(8)$ whereas the right diagonal embedding is the standard embedding 
of $\so(7)$. In diagrams {\rm (IND\ref{c5})}-{\rm (IND\ref{c4})} the left 
diagonal embedding of $\so(8)$ 
into $\so(1,n+8)$ is the triality automorphism followed by the standard embedding of
$\so(8)$, and the right diagonal embedding is just the standard embedding 
of $\so(8)$. 

In diagrams {\rm (IND\ref{c2})} and {\rm (IND\ref{c3})} 
we employ the standard isomorphism $\su(4)\simeq \so(6)$. Moreover, 
the lower diagonal 
lines refer to morphisms which are only non-trivial on $\ff$.

\end{prop}
\begin{proof} The proof is a matter of bookkeeping followed by a simple application 
of Onishchik's list of factorizations (see Proposition \ref{Oni2}). From Tables 
\ref{lcaph_class_symm}--\ref{lcaph_KKPS_class} one collects all pairs
$(\gf, \hf)$ for which $\hf$ contains a compact simple ideal $\hf^0$
that admits a non-trivial factorization according to Onishchik
(that is, $\hf^0=\su(2n)$ ($n\ge 2$), $\so(2n)$ ($n\ge 3$), or $\so(7)$).
For every such pair the projection $\cf^0$ of $\str(\gf,\hf)$ to $\hf^0$
is determined from the tables. Then for two such pairs $(\gf, \hf_1)$ and 
$(\gf, \hf_2)$ where $\hf_1$ and $\hf_2$ have a common simple compact factor, say 
$\hf^0$, one checks with Onishchik's list whether
$ \hf^0 = \cf^0_1 + \cf^0_2$. This happens precisely in the
cases listed in the proposition.
\end{proof}

\begin{cor}  Let $\gf$ be semi-simple without compact factors, and let
$(\gf, \hf)$ be an indecomposable real spherical reductive pair.
Suppose that $(\gf,\hf)$ is not strictly indecomposable, such that 
$\hf^0:= \hf/ [(\gf_1 \cap \hf)\oplus\ldots\oplus (\gf_k\cap \hf)]$ is 
not of the the type $\R^l \oplus (i\R)^m \oplus \sp(1)^n$ for any $l,m,n\ge 0$.  
Then $(\gf, \hf)$ is one of the pairs {\rm (IND\ref{c1})-(IND\ref{c3})}
listed in Proposition \ref{IND}.
\end{cor}

\begin{proof} By assumption, we first note that 
that there exists a pair of 
simple ideals of $\gf$, say $\gf_1$ and 
$\gf_2$, such that the projection of $(\gf,\hf)$ to 
$\gf_1\oplus\gf_2$ belongs to (IND\ref{c1})-(IND\ref{c3}).
Moreover $\str(\gf,\hf)$ is elementary in all cases (IND\ref{c1})-(IND\ref{c3}),
and the list of Onishchik (see Proposition \ref{Oni2}) readily excludes 
further possibilities. 
\end{proof}

 \subsection{Examples with arbitrarily many factors}
 
 In case $\cf_i=(\hf_i)_{\rm el}= \R^l \oplus (i\R)^m \oplus \sp(1)^n$ one can construct 
 a lot of examples of real spherical pairs with arbitrarily many factors, i.e. the bound of 
Proposition \ref{four exclusion} of $k\leq 3$ is not valid for general indecomposable pairs.

 This was first observed by Mikityuk \cite[page 545]{Mik}) who listed all complex spherical pairs 
 $(\gf,\hf)$ with $\gf$ simple  such that $\cf=\gl(1,\C)$ with $\hf= \cf \oplus [\hf,\hf]$. 
 This is for example the case for $(\gf, \hf)=(\so(2n,\C),\gl(n,\C))$
for $n\ge 5$ odd.  In case each $(\gf_i, \hf_i)$ is complex spherical with $\cf_i=\gl(1,\C)$ 
one obtains via 

$$ (\gf_1\oplus\ldots\oplus \gf_k,   ([\hf_1,\hf_1] \oplus \ldots\oplus [\hf_k,\hf_k]) + \diag \gl(1,\C))$$
a complex spherical pair with arbitrarily many factors. 

More interesting are perhaps the cases of strictly indecomposable real spherical pairs $(\gf, \hf)$
for which $\hf$ contains $\sp(1)$ as a factor such that $\hf=\sp(1)\oplus \tilde\hf$ as a direct sum of Lie algebras.

By inspecting our tables we see that this happens if and only if $(\gf,\hf)$ belongs to
the family of pairs 
\begin{equation}\label{first family}
(\so(n,1),\so(n-4q,1)\oplus\sp(q)\oplus\sp(1))
\end{equation}
(see Table \ref{lcaph_KKPS_class}),  
\begin{equation}\label{second family}
(\sp(p,q+1),\sp(p,q)\oplus\sp(1))
\end{equation}
(see Table \ref{lcaph_class_symm}) or the family (see (SS\ref{case4.20}))
\begin{equation}\label{third family}
(\sp(p+1,q) \oplus \sp(p,q), \sp(p,q)\oplus \sp(1))\, .
\end{equation}
Suppose that each $(\gf_i,\hf_i)$ is of the type (\ref{first family}) - (\ref{third family}) and 
decompose $\hf_i=\sp(1)\oplus \tilde \hf_i$. Then we obtain via 
$$ (\gf_1\oplus\ldots\oplus \gf_k,   (\tilde \hf_1 \oplus \ldots\oplus \tilde \hf_k) + \diag \sp(1))$$
real spherical pairs with arbitrarily many factors.

\appendix
\section{Tables for the structural algebra $\lf\cap \hf$}\label{appendix}
In this appendix we list the structural algebra $\str(\gf,\hf)=\lf\cap \hf$  
for all real spherical pairs $(\gf, \hf)$ with $\gf$ simple and $\hf$ reductive.
The recipe how this can be obtained is explained 
in  Part I,  Sect.~8.  
 We start with the symmetric cases (Berger's list \cite{Berger}), move on to 
the non-symmetric absolutely spherical cases obtained from Kr\"amer's list \cite{Kr} 
(see  Part I for the determination of all 
real forms in Kr\"amer's list),  and finally catalogue the remaining cases obtained in 
 Part I. 
 
\par In the tables below the following conventions hold:  $n = p + q = k + l$ with $p\leq q$ and $k\leq l$. Further $p=p_1 +p_2$, $q=q_1 + q_2$ and $k = p_1 + q_1$, $l= p_2+q_2$.  Notice that both $p\leq q$ and $k\leq l$ force $q_2\geq p_1$.
If $\hf\subset \gl(n,\C)$, then we use the notation $\sf[\hf]:=\hf\cap\sl(n,\C)$.

Tables \ref{lcaph_class_symm} - \ref{lcaph_Kraemer_sporadic} are separated by horizontal lines with each segment corresponding to the real forms
of a complex spherical space listed on top of the segment.

In most of the cases the embedding of $\lf\cap\hf$ into $\hf$ is unique
up to conjugation. The cases which cannot be decided just 
by the form of $\hf$ and $\lf\cap\hf$, and which are of relevance for 
this paper, are discussed in   Part I, Remark 8.2.

\subsection{The classical symmetric Lie algebras} [See also  \cite{baba}.]
$$
\tiny{{\arraycolsep=1pt}
\begin{array}{l l l l l}
 & \gf             & \hf                                                   & \str(\gf,\hf)  & \\  \hline
\zz & \sl(n,\C)    & \so(n,\C)                                             & 0           & \\
\zz & \sl(n,\R)    & \so(p,q)                                              & 0           & \\
\zz & \su(p,q)     & \begin{Cases}\so(p,q) \\ \so^*(2p),~p=q\end{Cases}    & \so(q-p)    & \\
\zz & \sl(n,\HH)   & \so^*(2n)                                             & (i\R)^n     & \\   \hline

\zz & \sl(2n,\C)   & \sp(n,\C)                                             & \sp(1,\C)^n & \\
\zz & \sl(2n,\R)   & \sp(n,\R)                                             & \sp(1,\R)^n & \\
\zz & \su(2p,2q)   & \begin{Cases}\sp(p,q)\\\sp(2p,\R),~p=q\end{Cases}     & \sl(2,\C)^p + \sp(q-p,0) & \\
\zz & \sl(n,\HH)   & \sp(p,q)                                              & \sp(1,0)^n  & \\ \hline

\zz & \sl(n,\C)    & \sf[\gl(p,\C)+\gl(q,\C)]                           &\sf[\gl(q-p,\C)+\C^p] & \\
\zz & \sl(n,\R)    & \begin{Cases}\sl(p,\R)+\sl(q,\R)+\R\\ \sl(p,\C) + i\R,~p=q \end{Cases} &\sf[\gl(q-p,\R)+\R^p] & \\
(11a) &  \su(p,q)  & \sf[\uf(p_1,q_1)+\uf(p_2,q_2)]                        & \begin{Cases} \sf[\uf(p_2-q_1,q_2-p_1)+(i\R)^{p_1+q_1}] \\\sf[\uf(q_1-p_2)+\uf(q_2-p_1)+(i\R)^{p_1+p_2}]\end{Cases} & \begin{array}{l}p_2\geq q_1\\p_2\leq q_1 \end{array} \\
(11b) &  \su(p,p)  & \sl(p,\C)+\R                                          &  (i\R)^{p-1} & \\
\stepcounter{zeile}

(12a) & \sl(n,\HH) & \sl(p,\HH)+\sl(q,\HH)+\R                              &\sf[\gl(q-p,\HH)+\gl(1,\HH)^p] & \\
(12b) & \sl(n,\HH) & \sl(n,\C)+i\R                                         & \begin{Cases}\sf[\gl(1,\HH)^{\frac{n}{2}}]\\ \sf[\gl(1,\HH)^{[\frac{n}{2}]}]+\uf(1)\end{Cases} & \begin{array}{l}n\text{ even} \\n\text{ odd} \end{array} \\ \hline
\stepcounter{zeile}

\zz & \so(2n,\C)    & \gl(n,\C) & \begin{Cases}\sl(2,\C)^{\frac n2}\\\sl(2,\C)^{\frac {n-1}2}+\C\end{Cases} & \begin{array}{l}n\text{ even} \\n\text{ odd} \end{array} \\
(14a) & \so(p,q)    & \uf(\frac{p}{2},\frac{q}{2})                        & \uf(\frac{q-p}{2})+\sl(2,\R)^{\frac{p}{2}} & \\
(14b) & \so(p,p)    & \gl(p,\R)                                           & \begin{Cases}\sl(2,\R)^{\frac{p}2}\\\sl(2,\R)^{\frac{p-1}2}+\R\end{Cases} & \begin{array}{l}p\text{ even} \\p\text{ odd} \end{array}  \\ \stepcounter{zeile}
(15a) & \so^*(2n)   & \uf(p,q)                                            & \begin{Cases}\sl(1,\HH)^{\frac n2}\\
\sl(1,\Hb)^{{\frac n2} -1} + \su(1,1) +i\R \\\sl(1,\HH)^{\frac{n-1}2}+i\R\end{Cases}&\begin{array}{l}n\text{ even},\  p \text{ even} \\ n\text{ even}, \ p \text{ odd} \\ n\text{ odd} \end{array} \\ 
(15b) & \so^*(2n)   & \gl(\frac{n}{2},\HH)               & \sl(1,\HH)^{\frac n2} & \,\,\, n\text{ even} \\ \hline\stepcounter{zeile}

\zz   & \so(n,\C)  & \so(p,\C)+\so(q,\C)                                  & \so(q-p,\C) & \\
(17a) & \so(p,q)   & \so(p_1,q_1)+\so(p_2,q_2) & \begin{Cases}\so(p_2-q_1,q_2-p_1)\\\so(q_1-p_2)+\so(q_2-p_1)\end{Cases} & \begin{array}{l}p_2\geq q_1\\p_2\leq q_1 \end{array} \\
(17b) & \so(p,p)   & \so(p,\C)                                            & 0 & \\
\stepcounter{zeile}

(18a) & \so^*(2n)  & \so^*(2p)+\so^*(2q)                                  & \so^*(2q-2p)+(i\R)^p & \\
(18b) & \so^*(2n)  & \so(n,\C)                                            & (i\R)^{[\frac{n}{2}]} & \\ \hline\stepcounter{zeile}

\zz & \sp(n,\C)    & \gl(n,\C)                                            & 0 & \\
\zz & \sp(n,\R)    & \begin{Cases}\uf(p,q)\\\gl(n,\R)\end{Cases}          & 0 & \\
\zz & \sp(p,q)     & \begin{Cases}\uf(p,q)\\\gl(p,\HH),~p=q \end{Cases}   & \uf(q-p)+(i\R)^p & \\   \hline
 (\ref{r0})
\zz & \sp(n,\C)    & \sp(p,\C)+\sp(q,\C)                                  & \sp(q-p,\C)+\sl(2,\C)^p & \\
\zz & \sp(n,\R)    & \begin{Cases}\sp(p,\R)+\sp(q,\R)\\\sp(p,\C),~p=q\end{Cases} & \sp(q-p,\R)+\sl(2,\R)^p & \\
(24a) & \sp(p,q)   & \sp(p_1,q_1)+\sp(p_2,q_2)                            & \begin{Cases}\sp(p_2-q_1,q_2-p_1)+\sp(1)^{p_1+q_1}\\\sp(q_1-p_2)+\sp(q_2-p_1)+\sp(1)^{p_1+p_2}\end{Cases} & \begin{array}{l}  p_2\geq q_1\\p_2\leq q_1 \end{array} \\
(24b) & \sp(p,p)   & \sp(p,\C)                                            & \sp(1)^p & \\ \hline
\end{array}}$$
\centerline{\Tabelle{lcaph_class_symm}}

\subsection{The exceptional symmetric Lie algebras}

$$\tiny{
\begin{array}{l c r}
\begin{array}{l l l}
\gf      & \hf         & \str(\gf,\hf) \\ \hline
\sE_6^\C & \sp(4,\C)   & 0\\
\sE_6^1  & \begin{Cases}\sp(4,\R)\\\sp(4)\\\sp(2,2)\end{Cases} & 0\\
\sE_6^2  & \begin{Cases}\sp(4,\R)\\\sp(1,3)\end{Cases} & 0\\
\sE_6^3  & \sp(2,2)    & \so(4)\\
\sE_6^4  & \sp(1,3)    & \so(4)+\so(4)\\  \hline

\sE_6^\C & \sl(6,\C)+\sl(2,\C) & \C^2\subseteq\sl(6,\C)\\
\sE_6^1  & \begin{Cases}\sl(6,\R)+\sl(2,\R)\\\sl(3,\HH)+\su(2)\end{Cases} & \R^2\\
\sE_6^2  & \begin{Cases}\su(6)+\su(2)\\\su(2,4)+\su(2)\\\su(3,3)+\sl(2,\R)\end{Cases} & (i\R)^2\\
\sE_6^3  & \begin{Cases}\su(2,4)+\su(2)\\\su(1,5)+\sl(2,\R)\end{Cases} & \uf(2)+\uf(2)\\
\sE_6^4  & \sl(3,\HH)+\su(2) & \so(5)+\so(3)+\R\\  \hline

\sE_6^\C & \so(10,\C)+\C & \sl(4,\C)+\C\\
\sE_6^1  & \so(5,5)+\R   & \sl(4,\R)+\R\\
\sE_6^2  & \begin{Cases}\so(4,6)+i\R\\\so^*(10)+i\R\end{Cases} & \uf(2,2)\\
\sE_6^3  & \begin{Cases}\so(10)+i\R\\\so(2,8)+i\R\\\so^*(10)+i\R\end{Cases} & \uf(4)\\
\sE_6^4  & \so(1,9)+\R   & \spin(7)+\R\\  \hline
  
\sE_6^\C & \sF_4^\C      & \so(8,\C)\\
\sE_6^1  & \sF_4^1       & \so(4,4)\\
\sE_6^2  & \sF_4^1       & \so(3,5)\\
\sE_6^3  & \sF_4^2       & \so(1,7)\\
\sE_6^4  & \begin{Cases}\sF_4^2\\\sF_4\end{Cases} & \so(8)\\ \hline

\sE_7^\C & \sl(8,\C)     & 0\\
\sE_7^1  & \begin{Cases}\sl(8,\R)\\\su(8)\\\su(4,4)\\\sl(4,\HH)\end{Cases} & 0\\
\sE_7^2  & \begin{Cases}\su(2,6)\\\su(4,4)\end{Cases} & \so(2)^3\\
\sE_7^3  & \begin{Cases}\su(2,6)\\\sl(4,\HH)\end{Cases} & \so(4)+\so(4)\\  \hline
&&\\
\end{array}
&\qquad&
\begin{array}{l l l} 
\gf      & \hf         & \str(\gf,\hf) \\\hline

\sE_7^\C & \so(12,\C)+\sl(2,\C) & \sl(2,\C)^3\subseteq\so(12,\C)\\
\sE_7^1  & \begin{Cases}\so(6,6)+\sl(2,\R)\\\so^*(12)+\su(2)\end{Cases} & \sl(2,\R)^3\\
\sE_7^2  & \begin{Cases}\so(12)+\su(2)\\\so(4,8)+\su(2)\\\so^*(12)+\sl(2,\R)\end{Cases} & \su(2)^3\\
\sE_7^3  & \begin{Cases}\so(2,10)+\sl(2,\R)\\\so^*(12)+\su(2)\end{Cases} & \so(6)+\so(2)+\sl(2,\R)\\ \hline
\sE_7^\C  & \sE_6^\C+\C                                                & \so(8,\C)\\
\sE_7^1   & \begin{Cases}\sE_6^1+\R\\\sE_6^2+i\R\end{Cases}            & \so(4,4)\\
\sE_7^2   & \begin{Cases}\sE_6^2+i\R\\\sE_6^3+i\R\end{Cases}           & \so(2,6)+\so(2)\\
\sE_7^3   & \begin{Cases}\sE_6^3+i\R\\\sE_6^4+\R\\\sE_6+i\R\end{Cases} & \so(8)\\

\sE_8^\C  & \so(16,\C)                                                 & 0\\
\sE_8^1   & \begin{Cases}\so(8,8)\\\so(16)\\\so^*(16)\end{Cases}     & 0\\
\sE_8^2   & \begin{Cases}\so(4,12)\\\so^*(16)\end{Cases}               & \so(4)+\so(4)\\\hline

\sE_8^\C  & \sE_7^\C+\sl(2,\C)                                         & \so(8,\C)\\
\sE_8^1   & \begin{Cases}\sE_7^1+\sl(2,\R)\\\sE_7^2+\su(2)\end{Cases}& \so(4,4)\\
\sE_8^2   & \begin{Cases}\sE_7^2+\su(2)\\\sE_7^3+\sl(2,\R)\\\sE_7+\su(2)\end{Cases} & \so(8)\\ \hline

\sF_4^\C  & \sp(3,\C)+\sl(2,\C)                                        & 0\\
\sF_4^1   & \begin{Cases}\sp(3,\R)+\sl(2,\R)\\\sp(3)+\su(2)\\\sp(1,2)+\su(2)\end{Cases}& 0\\
\sF_4^2   & \sp(1,2)+\su(2)                                          & \so(4)+\so(3)\\   \hline

\sF_4^\C  & \so(9,\C)                                                  & \spin(7,\C) \\
\sF_4^1   & \so(4,5)                                                   &  \spin(3,4) \\
\sF_4^2   & \begin{Cases}\so(9)\\\so(1,8)\end{Cases}                 &\spin(7)\\ \hline

\sG_2^\C  & \sl(2,\C)+\sl(2,\C)                                        & 0\\
\sG_2^1   & \begin{Cases}\sl(2,\R)+\sl(2,\R)\\\su(2)+\su(2)\end{Cases} & 0\\\hline
\end{array}\end{array}}
$$
\centerline{\Tabelle{lcaph_except_symm}}

\subsection{Two general cases of symmetric Lie algebras}

$$\tiny{
\begin{array}{l l l l}
\gf      & \hf         & \str(\gf,\hf) &\\ \hline
& & & \gf_\C \ \text{complex simple}, \\
\gf_\C& \gf_\R & \cf_\R &  \gf_\R \subset\gf_\C \ \text{real form}, \\
& & & \cf_\R \subset \gf_\R \ \text{fund.  Cartan}. \\ \hline
\gf&\gf&\gf& \gf \ \text{simple} \\ \hline
\end{array}}$$

\subsection{The non-symmetric absolutely spherical cases}

 $$\tiny{
 \begin{array}{l l l l}
\gf          & \hf          & \str(\gf,\hf)\\ \hline
\sl(2n+1,\C) & \sp(n,\C)+\ff & \ff & \ff \subset \C \\
\sl(2n+1,\R) & \sp(n,\R)+\ff & \ff & \ff \subset \R \\
\su(n,n+1)   & \sp(n,\R)+i\ff & i\ff & \ff \subset \R \\ 
\su(2p,2q+1) & \sp(p,q)+i\ff & \sp(q-p)+i\ff & \ff \subset \R \\ \hline

\sl(n,\C)    &\sl(p,\C)+\sl(q,\C) +\ff  & \sf[\sl(q-p,\C)+\C^p]+\ff & p<q, \ff\subset \C, \dim_\R \ff\leq 1\\
\sl(n,\R)    &\sl(p,\R)+\sl(q,\R)   & \sf[\sl(q-p,\R)+\R^p] & p<q\\
\su(p,q)     &\su(p_1,q_1)+\su(p_2,q_2)& \begin{Cases}\sf[\su(p_2-q_1,q_2-p_1)+(i\R)^{p_1+q_1}],~ p_2 \geq q_1\\\sf[\sf[\uf(q_1-p_2)+\uf(q_2-p_1)]+(i\R)^{p_1+p_2}],~ p_2\leq q_1 \end{Cases}& p_1+q_1< p_2+q_2 \\
\sl(n,\HH)   & \sl(p,\HH)+\sl(q,\HH)& \sf[\sl(q-p,\HH)+\gl(1,\HH)^p]\\  \hline
    
\so(2n+1,\C) & \sl(n,\C)+\C                 & 0\\
\so(n,n+1)   & \sl(n,\R)+\R & 0 \\
\so(2p,2q+1) & \su(p,q)+i\R & \uf(q-p) \\
 
\so(2n,\C) & \sl(n,\C)+ \ff   & \sl(2,\C)^{\frac {n-1}2}+\ff & \ff\subset\C, \dim_\R \ff \leq 1, n \ \hbox{odd}  \\
\so(n,n) & \sl(n,\R) & \sl(2,\R)^{\frac{n-1}2} &  n \ \hbox{odd}  \\
\so(2p,2q)   & \su(p,q)     & \su(q-p)+\sl(2,\R)^p&  p+q \text{ odd}\\
\so^*(2n)    & \su(p,q)                 & \sl(1,\HH)^{\frac{n-1}2} & n=p+q \ \hbox{odd}  \\ 
\hline
\sp(n+1,\C)  & \sp(n,\C)+\C & \sp(n-1,\C)+\C\\
\sp(n+1,\C) & \sp(n,\C) + \sp(1) & \sp(n-1,\C)+\sp(1)\\
\sp(n+1,\C) & \sp(n,\C) + \sp(1,\R) & \sp(n-1,\C) +\sp(1,\R)\\
\sp(n+1,\R)  & \sp(n,\R)+\ff & \sp(n-1,\R)+\ff&
\ff\in \{\R,i\R\}\\
\sp(p,q)     & \begin{Cases}\sp(p-1,q)+i\R\\\sp(p,q-1)+i\R\end{Cases} & \sp(p-1,q-1) + i \R \\
\hline
 \end{array}}$$
\centerline{\Tabelle{lcaph_Kraemer_class}}


$$\tiny{
\begin{array}{l l l l}
\gf         & \hf        & \str(\gf,\hf) \\ \hline
\so(7,\C)   & \sG_2^\C   & \sl(3,\C)\\
\so(3,4)    & \sG_2^1    & \sl(3,\R)\\
\hline
\so(8,\C)   & \sG_2^\C   & \sl(2,\C)\\
\so(4,4)    & \sG_2^1    & \sl(2,\R)\\
\so(3,5)    & \sG_2^1    & \sl(2,\R)\\
\so(1,7)    & \sG_2      & \su(3)\\
\hline
\so(9,\C)   & \spin(7,\C)     & \sl(3,\C)\\
\so(4,5)    & \spin(3,4)      & \sl(3,\R)\\
\so(1,8)    & \spin(7)        & \sG_2\\ \hline
\so(10)     & \spin(7,\C)+\C  & \sl(2,\C)\\
\so(5,5)    & \spin(3,4)+\R   & \sl(2,\R)\\
\so(4,6)    & \spin(3,4)+i\R  & \sl(2,\R)\\
\so(2,8)    & \spin(7)+i\R    & \su(3)\\
\so(1,9)    & \spin(7)+\R     & \sG_2\\
\so^*(10)   & \begin{Cases}\spin(1,6)+i\R\\\spin(2,5)+i\R\end{Cases} & \sl(1,\HH)+i\R\\  \hline
\sG_2^\C    & \sl(3,\C)       & \sl(2,\C)\\
\sG_2^1     & \begin{Cases}\sl(3,\R)\\\su(1,2)\end{Cases} & \sl(2,\R)\\  \hline
\sE_6^\C   & \so(10,\C) +\ff      & \sl(4,\C)+\ff  & \ff \subset \C, \dim_\R \ff \leq 1\\
\sE_6^1     & \so(5,5)        & \sl(4,\R)\\
\sE_6^2     & \begin{Cases}\so(4,6) \\\so^*(10)\end{Cases} & \su(2,2)\\
\sE_6^3     & \begin{Cases}\so(10)\\\so(2,8) \\\so^*(10)\end{Cases} & \su(4)\\
\sE_6^4     & \so(1,9)& \spin(7)\\
\hline 
 \end{array}}$$
\centerline{\Tabelle{lcaph_Kraemer_sporadic}}

\subsection{Non-absolutely spherical simple pairs $(\gf, \hf)$}

The table below treats the cases classified in  Part I.

$$\tiny{
 \begin{array}{l l l l}
\gf          & \hf                                              & \str(\gf,\hf)               & \\ \hline

\sl(n,\Hb)   &\sl(n-1,\Hb)+\ff,~\ff\subseteq\C                  & \sl(n-2,\Hb)+\ff          & n\geq 3\\
\sl(n,\Hb)   & \sl(n,\C)                                        & \gl(1,\Hb)^{[{n\over 2}]} & n\ \mathrm{odd}\\
\su(p,q)     & \su(p_1,q_1)+\su(p_2,q_2)                        & \begin{Cases}\sf[\su(p_2-q_1,q_2-p_1)+(i\R)^{p_1+q_1}],~ p_2 \geq q_1 \\\sf[\sf[\uf(q_1-p_2)+\uf(q_2-p_1)]+(i\R)^{p_1+p_2}],~p_2\leq q_1\end{Cases}           & (p_1,q_1)\neq (q_2,p_2) \\
\su(1,n)     & \su(1,n-2q) +\sp(q)+ \ff,~ \ff \subset \uf(1)   & \su(n-2q) +\sp(q-1) +i\R  +\ff & 1\leq q\leq\frac{n}{2} 
\\ \hline
   
\sp(p,q)     & \su(p,q)                                         & \sf [\uf(q-p) + i \R ^q]  & \\
\sp(p,q)     & \begin{Cases}\sp(p-1,q)\\ \sp(p, q-1)\end{Cases} & \sp(p-1,q-1)              & \\ 
\so(2p,2q)   & \su(p,q)                                         & \su(q-p) + \sl(2,\R)^p    & p < q\\
\so(2p+1, 2q) & \su(p,q)                                         & \su(q-p)                  & p\neq q-1,q \\\hline 

\so(3,n)     & \so(3,n-8) + \spin(7)                            & \so(n-8) + \so(3),~\so(3)\subset\spin(7) & n\geq 8 \\
\so(2,n)     & \so(2,n-8) + \spin(7)                            & \so(n-8) + \su(3)  & n\geq 8 \\
\so(2,n)     & \so(2,n-7) + \sG_2                               & \so(n-7) + \su(2)           & n\geq 7\\
\so(1,n)     & \so(1,n-8) + \spin(7)                            & \so(n-8) + \sG_2           & n\geq 8 \\
\so(1,n)     & \so(1,n-7) + \sG_2                               & \so(n-7) + \su(3)      & n\geq 7 \\
\so(1,n)     & \so(1,n-16) + \spin(9)                           & \so(n-16) + \spin(7)           & n\geq 16 \\
\so(1,n)     & \so(1,n-2r) + \su(r) + \ff,~\ff \subset\uf(1)    & \so(n-2r)  + \su(r-1)+\ff          & 2\leq r\leq\frac{n}{2} \\  
\so(1,n)     & \so(1,n-4r) + \sp(r) + \ff,~\ff \subset\sp(1)    & \so(n-4r) + \sp(r-1) +\ff              & 2\leq r\leq\frac{n}{4} \\ 

\so(3,6)     & \so(2)+ \sG_2^1                                  & 0                         & \\
\so(4,7)     & \so (3) + \spin(3,4)                             & 0                         & \\
\so^*(2n)    & \so^*(2n-2)                                      & \so^*(2n-4)               & n\geq 5 \\
\so^*(10)    & \begin{Cases}\spin(1,6) \\ \spin(2,5)\end{Cases} & \sl(1,\Hb)                & \\ \hline
\sE_6^4     & \sl(3,\Hb) + \ff , ~\ff\subset \uf(1)& \so(5) +\ff & \\
\sE_7^2      & \begin{Cases}\sE_6^3\\ \sE_6^2 \end{Cases}       & \so(2,6)                  & \\
\sF_4^2      & \sp(1,2) + \ff,~\ff \subset \uf (1)              & \so(4)+\ff                & \\
\hline
 \end{array}}$$
\centerline{\Tabelle{lcaph_KKPS_class}}

\bigskip 
{The following two tables related to rank one 
groups are used in Thm.~\ref{thm:semisimplespherical}. }

$$\tiny{
 \begin{array}{l l l l l
}
\gf          & \hf                                              & \ef               &\ef_0 &\\ \hline
\su(1,n)     & \su(1,p) +\su(n-p)  &\su(n-p)  & {\su(n-p-1)}  & 1 \leq p <n\\  
\su(1,n)     & \su(1,n-2p) +\sp(p)  & \sp(p) & {\sp(p-1)}  & 1\leq p<\frac{n}{2} \\ \hline
\so(1,n)      & \so(1,p)    +\so(n-p)  &  \so(n-p)  & \so(n-p-1) & 2\leq p\leq n -3  \\
\so(1,n)     & \so(1,n-8) + \spin(7)                            & \spin(7)          & \sG_2 & n\geq 10\\
\so(1,n)     & \so(1,n-7) + \sG_2                               &\sG_2 & \su(3)   & n\geq 9    \\
\so(1,n)     & \so(1,n-16) + \spin(9)                          &\spin(9)  & \spin(7) & n\geq 18 \\
\so(1,n)     & \so(1,n-2r) + \su(r)   & \su(r) & \su(r-1) & n-2r, r\geq 2  \\
\so(1,n)     & \so(1,n-4r) + \sp(r)     &\sp(r) & \sp(r-1) & n-4r, r\geq2  \\
\so(1,n)     & \so(1,n-4r) + \sp(r) +\sp(1)    &\sp(r) & \sp(r-1)+\sp(1) & n-4r, r\geq2  \\

\hline
 \end{array}}$$
\centerline{\Tabelle{e-table}}

The restriction on the indices in Table \ref{e-table} is such that $\hf\neq \gf$ is semi-simple and non-compact.

$$\tiny{
 \begin{array}{l l l l l}
\gf          & \hf                                              & {\mathfrak i}              &{\mathfrak i}_0 &\\ \hline
\su(1,n)     & \sf[\uf(1,p) +\uf(n-p)]  &\uf(n-p)  & \uf(n-p-1) & 1 \leq p <n\\  
\su(1,n)     & \su(1,n-2p) +\sp(p) +\uf(1)  & \sp(p){+\uf(1)}  & \sp(p-1) +\uf(1)  & 1\leq p<\frac{n}{2} \\ \hline
\so(1,n)     & \so(1,n-2) + \so(2)   & \so(2)  & 0 &  \\
\so(1,n)     & \so(1,n-2r) + \uf(r)   & \uf(r) & \uf(r-1) & n-2r, r\geq 2  \\
\so(1,n)     & \so(1,n-4r) + \sp(r)  +\uf(1)  &\sp(r){+\uf(1)}  & \sp(r-1) +\uf(1)& n-4r, r\geq2  \\

\hline
 \end{array}}$$
\centerline{\Tabelle{i-table}}

\subsection{Strongly spherical pairs}

We conclude with a classification of the  strictly indecomposable
strongly spherical pairs $(\gf,\hf)$,
with $\gf$ semi-simple and non-compact. This 
classification
is easily extracted from our tables, in view of
Lemma \ref{lemma stand2}.  In the following table we 
exclude symmetric pairs, as they
were previously tabled in \cite{KM},  and cases with $\hf$ compact.

$$\!\!\!\!\!\!\!{\small \begin{array}{llll}
\gf&\hf\\
\hline
\su(p,q)&\su(p-1,q)&&p,q\ge 1;\, p-1,p\neq q\\
\su(n,1)&\su(n-q,1){+}\su(q)& & n \ge 3; 1\le q\le n\\
\su(n,1)&\su(n-2q,1){+}\sp(q)+\ff&\ff\subseteq\uf(1)&1\le q\le\frac n2\\
\sl(n,\Hb)&\sl(n-1,\Hb)+\R+\ff&\ff\subseteq\uf(1)&n\ge3   \\
\hline
\sp(p,q)&\sp(p-1,q)+\ff&\ff\subseteq\uf(1)& p,q\ge1\\
\hline
 \so(2n,2)&\su(n,1)&&n\ge 2\\
 \so^*(2n+2)&\so^*(2n)&&n\ge 3\\
\so(n,1)&\so(n-2q,1)+\su(q)+\ff&\ff\subseteq\uf(1)&2\le q\le\frac n2\\
\so(n,1)&\so(n-4q,1)+\sp(q)+\ff&\ff\subseteq\sp(1)&2\le q\le\frac n4\\
\so(n,1)&\so(n-16,1)+\spin(9)&&n\ge16\\
\so(n,1)&\so(n-7,1)+\sG_2&&n\ge7\\
\so(n,1)&\so(n-8, 1)+\spin(7)&&n\ge8\\
\hline
\end{array}}$$
\centerline{\Tabelle{strsph}}

\section{The geometry of restricted root spaces in symmetric spaces}\label{AppB}

Throughout this appendix we let $\gf$ be a semi-simple  real Lie algebra and $\hf\subset \gf$
a symmetric subalgebra. Let $\sigma: \gf \to \gf$ be the involution with fixed point algebra $\hf$
and $\theta$ a Cartan involution of $\gf$ which commutes with $\sigma$.  
We denote by $\gf=\kf +\sf$ the decomposition into $\theta$-eigenspaces where 
$\kf\subset \gf$ corresponds to the $+1$-eigenspace and constitutes a maximal compact subalgebra
of $\gf$.  Notice that $\sf=\kf^\perp$ where the orthogonal complement is taken with respect to 
the Cartan-Killing form.  Likewise we record the decomposition into $\sigma$-eigenspaces 
$\gf = \hf + \hf^\perp$. 
Let now $\af_0 \subset \kf^\perp \cap\hf^\perp$ be a maximal abelian subspace which we inflate 
to a maximal abelian subspace $\af\subset \kf^\perp$.  Let $\Sigma=\Sigma(\gf, \af)\subset \af^*$ 
be the corresponding restricted root system and recall a result of Rossmann 
 \cite{Rossmann} which asserts that 
$\Sigma_0:= \Sigma|_{\af_0}\bs \{0\}$ is a root system (as $\Sigma$,  possibly reduced). 
We let $\Sigma_0^+\subset \Sigma_0$ be a  set of positive roots which we inflate to 
a  set $\Sigma^+$  of positive roots for $\Sigma$.   Set $\mf:=\zf_\kf(\af)$,
 the centralizer of $\af$ in $\kf$, and let $\nf\subset \gf$ 
be the   sum of the root spaces for the roots of $\Sigma^+$.  Then 
$\pf:=\mf + \af +\nf$ defines a minimal parabolic subalgebra of $\gf$ with 
$\gf = \hf +\pf$ and thus exhibiting  $(\gf,\hf)$ as a spherical pair.
For any $\alpha\in \Sigma_0$ we denote by $\gf^\alpha$ the corresponding root space and define 
a unipotent subalgebra $\uf:=\bigoplus_{\alpha \in \Sigma_0^+} \gf^\alpha$.  Then with 
$\lf:= \zf_\gf(\af_0)$ we obtain via $ \qf:= \lf  \ltimes \uf$
a parabolic subalgebra of $\gf$ which contains $\pf$. Moreover 
$\lf=\qf\cap\theta(\qf)=\qf\cap \sigma(\qf)$ is $\sigma$-stable 
with $\lf_{\rm n} \subset \lf \cap \hf$.  This shows that $\qf$ is the unique parabolic subalgebra
containing $\pf$, which is adapted to $(\gf,\hf)$.  
Furthermore $\af_0\simeq \af/ \af\cap \hf=\af_Z$.

\par Consider the involution $\tau:=\sigma\circ \theta$ and note that 
$\af_0$ is $\tau$-fixed. Hence every root space $\gf^\alpha$ is $\tau$-stable and every root space
$\gf^\alpha$ decomposes into $\gf^{\alpha, +}\oplus \gf^{\alpha, -}$ according to the eigenspaces of $\tau$.  
Moreover $\gf^+:=\gf^\tau$ is $\sigma$-stable and $\sigma$ coincides with $\theta$ on $\gf^+$; in other words: 
$(\gf^+, \gf^+\cap \hf)$ is a Riemannian subsymmetric 
pair of $(\gf, \hf)$. Notice that $\lf^+:=\lf\cap\gf^+
= \af_0 +\mf^+$  with $\mf^+$ the centralizer of $\af_0$ in $\kf\cap\hf$. 

\par Let $G=\operatorname{Int}(\gf)$ be the adjoint group of 
$\gf$,  and let $M^+\subset L\subset G$ denote the
connected subgroups corresponding to the subalgebras 
$\mf^+\subset\lf\subset\gf$. Then
each root space $\gf^\alpha$, $\alpha\in \Sigma_0=\Sigma(\gf, \af_0)$,  is  naturally a module for $L$, 
and also for the extension $L_\tau:=L \rtimes\{ \1,\tau\}$,
and the subspaces $\gf^{\alpha,\pm}$ are $\tau$- 
and $M^+$-invariant.

This appendix is 
concerned with the basic $L$-geometry of the 
root spaces $\gf^\alpha$, summarized in the next proposition:

\begin{prop} \label{L-geometry} Let $(\gf, \hf)$ be a symmetric pair and $\af_0\subset \hf^\perp\cap\kf^\perp$ 
be a maximal abelian subspace.  Then with notation as above the following 
holds for every root space $\gf^\alpha$ for $\alpha\in \Sigma_0$:  

\begin{enumerate} 
\item\label{eins} $\gf^\alpha$ decomposes into finitely many $L$-orbits. 
In particular, there exists an open $L$-orbit, and hence $\gf^\alpha_\C$ is a prehomogeneous vector space. 
\item \label{eineinhalb} If $\dim \gf^{\alpha, \epsilon}>1$ then
$M^+$ acts transitively on the unit sphere of $\gf^{\alpha, \epsilon}$ ($\epsilon\in\{\pm\}$).
\item \label{zwei} $\gf^\alpha$ is irreducible as a real representation of $L_\tau$.

\end{enumerate}
\end{prop}

\begin{rmk}\label{Kostant-remark} 
(a) In case $\sigma=\theta$ we have $\gf^+=\gf$, and $M^+$  
is the identity component of the centralizer of $\af$ in $K$.
In this case the transitivity in (\ref{eineinhalb}) is 
a well-known result of Kostant  \cite[Thm.~2.1.7]{Kostant} of which (\ref{eins})
and (\ref{zwei}) are also immediate consequences.
\par (b) In general $\gf^\alpha$ is not irreducible as a real $L$-module. This
is for example easy to see in the case of $(\gf,\hf)=(\hf\oplus\hf,\diag(\hf))$.
\end{rmk}

\begin{proof}  Fix $\alpha\in \Sigma_0^+$ and consider the graded subalgebra
\begin{equation}\label{g(alpha)} 
\gf(\alpha):= 
\gf^{-2\alpha}\oplus\gf^{-\alpha}\oplus\lf\oplus\gf^{\alpha}\oplus\gf^{2\alpha}
\end{equation}
This subalgebra is $\sigma$ and $\theta$-stable and the same holds for the semi-simple
subalgebra $\gf[\alpha]$ generated by $\gf^{\alpha}$ and $\gf^{-\alpha}$.  
With $\lf[\alpha]= \lf\cap \gf[\alpha]$ we then obtain
$$\gf[\alpha]  = 
\gf^{-2\alpha}\oplus\gf^{-\alpha}\oplus\lf[\alpha]\oplus\gf^{\alpha}\oplus\gf^{2\alpha}$$
and $\af_0[\alpha]=\af_0 \cap \lf[\alpha]$ is one-dimensional.
It is clear that it suffices to show
the assertions of the proposition for $\gf[\alpha]$, 
and from now on we assume that $\gf=\gf(\alpha)=\gf[\alpha]$.

To prove (\ref{eins}) we observe that (\ref{g(alpha)}) provides a 
$\Z$-grading of $\gf$.
Then a result of  Vinberg (cf. \cite[Prop. 2]{V}) implies that 
the action $L_\C $ on $(\gf^\alpha)_\C$ 
has only finitely many orbits.  Let $\Oc\subset (\gf^\alpha)_\C$ be one of them and set 
$\Oc_\R:= \Oc\cap \gf^\alpha$.  Then $\Oc_\R$ is a finite union of $L$-orbits by a result of Whitney \cite{Whit} on connected components of real points of smooth complex varieties.

For (\ref{eineinhalb}) we first note that $M^+$ is compact and preserves the
Cartan-Killing form, hence it acts on the unit sphere of $\gf^{\alpha,\epsilon}$
and the orbits are closed. It then suffices to show that there is an open orbit.
Since $A_0$ acts by scalar multiplication,
it is equivalent to show that $L^+:=M^+A_0$ has an open orbit on $\gf^{\alpha,\epsilon}$.
This follows by applying (\ref{eins})  to the $\sigma$-invariant 
Lie subalgebra of $\gf$,
$$ \gf^\epsilon := \gf^{-2\alpha,+} + \gf^{-\alpha, \epsilon} + 
\lf^+ + \gf^{\alpha, \epsilon}+ \gf^{2\alpha,+}\, .$$

We move on to (\ref{zwei}). Let us first observe that if $\af_0=\af$,
that is, if $\af_0$ is maximal abelian in $\sf$, then by applying (\ref{eineinhalb}) 
for the involution $\sigma=\theta$ (as in Remark \ref{Kostant-remark}(a))
it follows that $\gf^\alpha$ is irreducible already as an $L$-module.
We assume from now on then that $\af_0\subsetneq\af$,  
and let $Y\in\af\cap\hf$ be non-zero.
Since $\gf=\gf[\alpha]$ it follows that $[Y,\gf^\alpha]\neq\{0\}$.
Note that $\tau(Y)=-Y$ and hence
$[Y,\gf^{\alpha,\pm}]\subset\gf^{\alpha,\mp}$. Since $\ad(Y)$ acts semisimply,
it follows that $[Y,\gf^{\alpha,\epsilon}]\neq\{0\}$ for each $\epsilon$.

\par Let $U \subset \gf^\alpha$ be a non-zero $L_\tau$-invariant subspace.
Since $U$ is 
$\tau$-invariant then $U=U^+\oplus U^-$ where $U^\pm=U\cap\gf^{\alpha,\pm}$,
and each $U^\pm$ is $M^+$-invariant.
It follows from (\ref{eineinhalb}) that $U^\epsilon$ equals 
$\{0\}$ or $\gf^{\alpha,\epsilon}$ for each $\epsilon$.
It now follows from the $\ad(Y)$-invariance that
$U=\gf^\alpha$, and we are done. 
\end{proof}


\begin{thebibliography} {10}

\bibitem{baba} K. Baba,
{\it Local orbit types of the isotropy representations for semisimple pseudo-Riemannian symmetric spaces},
 Diff. Geom. Appl. {\bf 38} (2015), 124--150.
 
\bibitem{Berger} M. Berger, {\it Les espaces sym\'etriques noncompacts},
Ann. Sci. \'Ecole Norm. Sup. (3) {\bf 74} (1957), 85--177. 

\bibitem{BP} P. Bravi and G. Pezzini, {\it The spherical systems of the wonderful reductive subgroups}, 
J. Lie Theory 25 (2015), 105--123.

\bibitem{Brion} M. Brion,  {\it Classification des espaces homog\`enes sph\'eriques},  Compositio Math. {\bf 63}(2) (1987), 189--208.


\bibitem{DKS} P. Delorme, B. Kr\"otz and S. Souaifi, 
{\it The constant term of tempered functions on a real spherical space}, 
 arXiv:1702.04678 

 

\bibitem{Mol} J. Frahm (= M\"ollers), {\it Symmetry breaking operators for strongly spherical reductive pairs and 
the Gross-Prasad conjecture for complex orthogonal groups},
arXiv:1705.06109

 
 
\bibitem{KK} F. Knop and B. Kr\"otz, {\it Reductive group actions},  arXiv: 1604.01005 

\bibitem{KKPS}  F. Knop, B. Kr\"otz, T. Pecher and H. Schlichtkrull, 
{\it Classification of reductive real spherical pairs I. The simple case},  Transformation Groups,
doi 10.1007/s00031-017-9470-5

\bibitem{KKSS} F. Knop, B. Kr\"otz, E. Sayag  and H. Schlichtkrull, {\it Simple compactifications and polar decomposition of homogeneous real spherical spaces},  Selecta Math. N.S. {\bf 2} (2015), 1071--1097.


\bibitem{KKSS2} \bysame, {\it Volume growth, temperedness and integrability of matrix coefficients on a real spherical
space},  J. Funct.  Anal. {\bf 271} (2016), 12--36.


\bibitem{KKS} F. Knop, B. Kr\"otz and H. Schlichtkrull, {\it The local structure theorem for 
real spherical spaces}, Compositio Math. {\bf 151} (2015), 2145--2159. 



\bibitem{KKS2} \bysame, {\it The tempered spectrum of a real spherical space}, 
Acta Math. {\bf 218} (2017), 201--295.


\bibitem{KnSt} F. Knop and B. Van Steirteghem, {\it Classification of smooth affine spherical varieties}, Transformation Groups {\bf 11} (2006), 
495--516.

\bibitem{KM}  T. Kobayashi and T. Matsuki, {\it Classification of finite-multiplicity symmetric pairs}, Transformation Groups {\bf 19 (2)} (2014), 
457--493


\bibitem{KoOs} T. Kobayashi and T. Oshima, {\it Finite multiplicity theorems for induction and restriction},
Adv. Math. {\bf 248} (2013), 921--944.

\bibitem{Kostant} B. Kostant, {\it On the existence and irreducibility of certain series of representations.} Lie groups and their representations (Proc. Summer School, Bolyai János Math. Soc., Budapest, 1971), pp. 231–329. Halsted, New York, 1975. 


\bibitem{Kr} M. Kr\"amer, {\it Sph\"arische Untergruppen in kompakten zusammenh\"angenden Liegruppen}, 
Compositio math. {\bf 38 (2)} (1979), 129--153. 



\bibitem{KKOS} B. Kr\"otz, J.J. Kuit, E. Opdam and H. Schlichtkrull, {\it The infinitesimal
characters of discrete series for real spherical spaces}, arXiv:1711.08635

\bibitem{KS1} B. Kr\"otz and H. Schlichtkrull, {\it Finite orbit decomposition 
of real flag manifolds}, J. EMS {\bf 18} (2016), 1391--1403.

\bibitem{KS2} \bysame, {\it Multiplicity bounds and the subrepresentation theorem for real spherical spaces},
Trans. Amer. Math. Soc. {\bf 368} (2016), 2749--2762.


\bibitem{Mik}   I.V. Mikityuk, {\it On the integrability of Hamiltonian systems with homogeneous configuration spaces}, 
Math. USSR Sbornik, {\bf 57(2)} (1987), 527--546.


\bibitem{Oni} A.L. Onishchik, {\it Inclusion relations among transitive compact transformation groups}, 
Trudy Moskov. Mat. Obshch. {\bf 11} (1962), 199--242; Amer. Math. Soc. Transl. (2) {\bf 50} (1966), 5--58.


\bibitem{OS} T. Oshima and J. Sekiguchi, {\it The restricted root system of a semisimple symmetric pair}, 
Group representations and systems of differential equations (Tokyo, 1982), 433--497, Adv. Stud. Pure Math.,  
{\bf 4}, North-Holland, Amsterdam, 1984. 


 \bibitem{Rossmann} W. Rossmann, {\it The structure of semisimple symmetric spaces},
Canad. J. Math. {\bf 31} (1979), 157--180.


\bibitem{SV}  Y.~Sakellaridis and A.~Venkatesh, {\it Periods and harmonic analysis on spherical varieties}, 
 Ast\'erisque 396 (2017).

\bibitem{Timashev} D. Timashev, Homogeneous Spaces and
Equivariant Embeddings, Enc. of Math. Sciences {\bf 138}, Springer Verlag 2011.

 \bibitem{VSV} V.S. Varadarajan, {\it Spin(7)-subgroups of 
SO(8) and Spin(8)}, Expo. Math. {\bf 19} (2001),  163--177. 




\bibitem{V} E.B. Vinberg, {\it The Weyl group of a graded Lie algebra}, Math. USSR 
Izv. {\bf 10} (1976), 463--495.


\bibitem{Whit} H. Whitney,
{\it Elementary Structure of Real Algebraic Varieties},
Annals of Math. {\bf 66}, (1957), 545--556.


\end{thebibliography}
\end{document}